\documentclass[12pt, titlepage]{article}

\usepackage{color}
\usepackage{float}
\usepackage[bf,small]{caption2}
\usepackage{comment}
\usepackage{amsmath}
\usepackage{bigints}
\usepackage{enumerate}
\usepackage{afterpage}
\usepackage{rotating, booktabs}
\usepackage{amsopn}
\usepackage{amsfonts}
\usepackage{amsthm}
\usepackage{amssymb}
\usepackage{amsbsy}
\usepackage{gensymb}
\usepackage{multirow}
\usepackage{titling}
\usepackage{natbib}
\usepackage{tocbibind}
\usepackage[a4paper,colorlinks,breaklinks,bookmarksopen,bookmarksnumbered]{hyperref}
\usepackage{epsfig,psfrag}
\usepackage{graphicx}
\usepackage{amsmath}
\usepackage{epstopdf}
\usepackage{tabularx}
\usepackage{subfigure}
\usepackage[mathscr]{euscript}
\usepackage[margin=1in]{geometry}
\usepackage{xr-hyper}
\usepackage{amsfonts} % Mathesymbole
\usepackage{calc}
\usepackage{titlesec}

\newcommand{\ueq}[1][]{%
  \if\relax\detokenize{#1}\relax
    \sbox0{$\underbrace{=}_{}$}%
    \mathrel{\mathmakebox[\wd0]{=}}
  \else
    \mathrel{\underbrace{=}_{\mathclap{#1}}}
  \fi}

\newcommand {\ctn}{\cite}

\usepackage{url}
\newcommand{\bx}{{\bf x}}

\newcommand{\bs}{{\bf s}}
\newcommand{\bS}{{\bf S}}
\newcommand{\by}{{\bf y}}
\newcommand{\bSigma}{\Sigma}
\newcommand{\bmu}{\boldsymbol{\mu}}

\newcommand{\bzero}{\boldsymbol{0}}
\newcommand{\bz}{\boldsymbol{z}}
\newcommand{\bbeta}{\boldsymbol{\beta}}
\newcommand{\bvartheta}{\boldsymbol{\vartheta}}
\newcommand{\bgamma}{\boldsymbol{\gamma}}

\newcommand{\yn}{{\bf y}_n}
\newcommand{\bepsilon}{{\bf \epsilon}}
\newcommand{\bta}{\boldsymbol{\theta}}
\newcommand{\bTheta}{\boldsymbol{\Theta}}
\newcommand{\e}{\ensuremath{\epsilon}}

\newtheorem{theorem}{Theorem}
\newtheorem{lemma}{Lemma}
\newtheorem{result}{Result}
\newtheorem{corollary}{Corollary}

\theoremstyle{definition}

\newtheorem{remark}{Remark}

\newcommand{\topline}{\hrule height 1pt width \textwidth \vspace*{2pt}}
\newcommand{\botline}{\vspace*{2pt}\hrule height 1pt width \textwidth \vspace*{4pt}}
\newtheorem{algo}{Algorithm} 

\numberwithin{equation}{section}
\numberwithin{algo}{section}
\numberwithin{table}{section}
\numberwithin{figure}{section}

\title{\Large\bf Bayes Factor Asymptotics for Variable Selection in the Gaussian Process Framework}
\vspace{.4cm}
\author{\large {\it Minerva Mukhopadhyay and Sourabh Bhattacharya} \thanks{\it
Department of Mathematics and Statistics, Indian Institute of Technology, Kanpur and
 Interdisciplinary Statistical Research Unit, Indian Statistical
Institute, 203, B. T. Road, Kolkata 700108.
Corresponding e-mail: minervam@iitk.ac.in.}}
\date{\vspace{-0.5in}}
 
\begin{document}
 \maketitle
  \vspace{.2 in}
 %\\ \vspace{.2 in}
\begin{abstract}
Although variable selection is one of the most popular areas of modern statistical research, much of its development has taken place in the classical paradigm compared to
the Bayesian counterpart. Somewhat surprisingly, both the paradigms have focused almost primarily on linear models, in spite of the vast scope offered by
the model liberation movement brought about by modern advancements in studying real, complex phenomena.

In this article, we investigate general Bayesian variable selection in models driven by Gaussian processes, which allows us to treat 
linear, non-linear and nonparametric models, in conjunction with even dependent setups, in the same vein. We consider the Bayes factor route to
variable selection, and develop a general asymptotic theory for the Gaussian process framework in the ``large $p$, large $n$" settings even with $p\gg n$, 
establishing almost sure exponential convergence of the Bayes factor under appropriately mild conditions. The fixed $p$ setup is included as a special case.

To illustrate, we apply our result to variable selection in linear regression, Gaussian process model with squared exponential covariance function
accommodating the covariates, and a first order autoregressive process with time-varying covariates. 
We also follow up our theoretical investigations with ample simulation experiments in the above regression contexts 
and variable selection in a real, riboflavin data consisting of $71$ observations
but $4088$ covariates.
%, modeled using linear and Gaussian process regression.
For implementation of variable selection using Bayes factors, we develop a novel and effective general-purpose 
transdimensional, transformation based Markov chain Monte Carlo algorithm, which has played a crucial role in our simulated and real data applications.  

\end{abstract}

\tableofcontents

\section{Introduction}
\label{sec:intro}

The importance of variable selection is undeniable, since most statistical procedures involve a large number of observed variables, or covariates, 
only a few of which are expected to have significant influence on the experiment and future prediction.
% from the associated model.  
It is thus important to judiciously select those few important covariates from a relatively large pool of available covariates. 
%The task is not as simple as it sounds, since 
This task involves multiple challenges. Even in the simple classical linear regression setup, false inclusion or exclusion of the variables 
may lead to false inclusion or exclusion of correlated variables. 
%That in turn 
%False inclusion and exclusion may 
%leads to increased variance of the coefficients and biased coefficients, respectively. 
That, in turn, can influence the variance of predictions and hence root mean square error (RMSE), and bias of the predictions.  
The most popular methods developed in the classical paradigm, 
%those based on significance tests and forward-backward selection based methods, 
%are quite vulnerable to these problems. Indeed, these methods, which are only applicable to nested models, also lead to different sets of selected variables. 
%The 
the penalty based methods such as the Akaike Information Criterion and the Bayesian Information Criterion, are not immune to these problems,
the former having the ill reputation of preferring models consisting of relatively large number of variables. The latter employs a more appropriate penalty
 and is preferable, but in practice, can lead to underfitting. The popular LASSO method (see, for example, \ctn{Tibshirani96}) often has the effect of drastically reducing RMSE, 
but at the cost of increasing prediction errors. See \ctn{Heinze18} for a relatively recent review regarding several of these issues; see also
 \ctn{Draper98}, \ctn{Weisberg05}. Asymptotic theory of the variable selection criterion in multiple regression has been considered in \ctn{Nishii84} and \ctn{Shao97};
 see also \ctn{Eubank99} and \ctn{Giraud15} for various issues regarding asymptotic variable selection in linear models. 
 Since even for simple linear regression 
models the variable selection issues can be of significant concern, it is well imaginable how grave the issues can be in the case of more realistically complex models 
such as nonlinear and nonparametric regression. 
%Unfortunately, variable selection other than in linear regression and its variants does not seem to have received the deserved attention.

Apart from some of the issues touched upon, all the classical methods of variable selection have the major drawback of selecting a single set of variables without 
quantifying the uncertainty associated with such selection. This calls for the Bayesian paradigm of variable selection, which is also rich in its repertoire of 
philosophies and methodologies. One philosophy is Bayesian model averaging, which recommends a mixture of all possible models for better prediction (see \ctn{BMA_review} for a review). 
%Thus, rather than selecting a single set of variables, all the models associated with different subsets of the set of covariates are probabilistically combined in a mixture form. 
The mixture weights ensure that important sets of covariates receive substantial weights compared to less significant sets.
Another philosophy is to infer from the posterior distribution of the regression coefficients (see, for e.g., \cite{IshwaranRao2005}).
Another philosophy is to obtain the posterior distribution of the subsets of the covariates, and from a single posterior that encapsulates all the relevant information including the possibility of all the covariates. Covariate selection in this case proceeds by stochastic search variable selection methods, which often involve variable-dimensional Markov chain Monte Carlo (MCMC) procedures (see \ctn{OHara_2009} for a review). 
%
%A large class of methods guaranteeing posterior consistency have been proposed in the literature, see see, for example, 
Even though these methods are
usually computationally demanding, most of them avoid the problems faced by the classical
variable selection ideas. For details regarding various ideas on Bayesian model and variable
selection along with relevant computational strategies, see, for example,
\ctn{Gilks96}, \ctn{DiCiccio97}, \ctn{Han01}, \ctn{Fernandez01}, \ctn{Moreno08}, \ctn{Casella09}, \ctn{Ando10}, \ctn{Bayarri12}, \ctn{Johnson12}, 
\ctn{Hong12}, \ctn{marin14}, \ctn{Dawid15}.
 Asymptotic theories on Bayesian variable selection can be found in \ctn{Moreno10}, \ctn{Shang11}, \ctn{Moreno15}, \ctn{Minerva15}. However, most of these theories are developed in the linear regression setup.

However, perhaps the most principled way of comparing the subsets of covariates is offered by Bayes factors, through the ratio of the posterior and prior odds 
associated with the competing models, which follows directly from the coherent procedure of Bayesian hypothesis testing of preferring one model compared to other. 
The idea is also closely related to the aforementioned principle of obtaining posterior distributions of the covariate subsets.
For a general account of Bayes factors and its numerous advantages, see, for example, \ctn{Kass95}. However, careless use
of Bayes factors can lead to selecting the more parsimonious but wrong model in large samples even in very simple setups for ill-chosen priors, 
 as the well-known Jeffreys-Lindley-Bartlett paradox demonstrates (see \ctn{Jeffreys39}, \ctn{Lindley57}, \ctn{Bartlett57}, \ctn{Robert93}, \ctn{Villa15} for details). 
It is thus of utmost importance to carefully investigate the asymptotic theory of Bayes factors
in different setups and construct appropriate priors that ensure consistency in the sense that the Bayes factor selects the correct set of covariates asymptotically.
Note that priors that ensure consistency of posterior distributions need not guarantee consistency of Bayes factors, which is again demonstrated by the
Jeffreys-Lindley-Bartlett and information paradox (see, for example, Section 2.3 of \ctn{Liang08}). 
Here, the prior ensures posterior consistency, but not Bayes factor consistency. 
Thus, the asymptotic theory of Bayes factors does not follow from the asymptotic theory of posterior distributions.

% For details regarding various ideas on Bayesian model and variable selection along with relevant computational strategies, see, for example, 
%\ctn{Gilks96}, \ctn{DiCiccio97}, \ctn{Han01}, \ctn{Fernandez01}, \ctn{Moreno08}, \ctn{Casella09}, \ctn{Ando10}, \ctn{Bayarri12}, \ctn{Johnson12}, 
%\ctn{Hong12}, \ctn{marin14}, \ctn{Dawid15}.
% Asymptotic theories on Bayesian variable selection can be found in \ctn{Moreno10}, \ctn{Shang11}, \ctn{Moreno15}, \ctn{Minerva15}. 

Compared to the asymptotic theory of posterior distributions, that of 
Bayes factors for general model selection have seen relatively slow development. 
Indeed, most of the theory for variable selection using Bayes factors %, indeed, of general Bayesian variable selection methods, 
have hitherto concentrated around nested linear regression models; see, for example, \ctn{Guo09}, \ctn{Liang08}, \ctn{Moreno10}, \ctn{Rousseau12}, \ctn{Wang14}, \ctn{Kundu14}, 
\ctn{Choi15}. But see also \ctn{Wang16} for a non-nested setup. 
This seems to be a very restrictive setup for the Bayesian framework,
particularly 
%since the Bayesian paradigm encourages model liberation, 
in light of the current advancement in research on highly complex physical phenomena, where
simplistic models are untenable. For a general account of advancements in the area of Bayes factor asymptotics, see \ctn{Chib16}, which also asserts the same fact. 

Although variable selection has been considered in nonlinear and nonparametric frameworks such as generalized linear models, generalized additive models,
additive partial linear models, generalized additive partial linear models, semiparametric additive partial linear
models, additive nonparametric regression models (see, for example, \ctn{Chen99}, \ctn{Huang10}, \ctn{Liu11}, \ctn{Marra11}, \ctn{Meyer02}, \ctn{Ntz03}, 
\ctn{Reich09}, \ctn{Shively99}, \ctn{Wang11}, \ctn{Wang07}, \ctn{Banerjee14}), Bayes factor is not the selection criterion 
for the existing approaches.

It is thus crucially important to build appropriate asymptotic theory for Bayes factors with respect to variable selection in general
setups.% that encompass both simple and complex models. 

Recognizing this requirement, our endeavor in this paper is to establish consistency of Bayes factors
for variable selection in models driven by Gaussian processes. The Gaussian process framework enables us to consider linear and nonlinear, parametric as well as nonparametric
models including appropriate dependence structures, under the same umbrella, allowing the usage of a general body of mathematical apparatus to establish
our asymptotic theory. Encouragingly, such a treatment allowed us to guarantee almost sure exponential convergence of the Bayes factor in favour of the true set
of covariates under reasonably mild, verifiable assumptions, not only as the sample size increases indefinitely, but also as the total number of available covariates
increase with the sample size, possibly at faster rates, defining the so-called ``large $p$, large $n$" paradigm, which also includes the fixed $p$ situation as a special case.
We are not aware of any asymptotic theory of Bayes factors in the ``large $p$, large $n$" scenario.

%As we shall show, our large $p$, small $n$ results include the fixed $p$ situation as a special case. Note that even in the fixed $p$ situation,
%% (let alone the large $p$, small $n$ paradigm), 
%almost sure exponential convergence is in stark contrast
%with just ``in probability" convergence theories of Bayes factors for variable selection developed so far for linear models and their various versions, including the
%models additive in the functions of the covariates.

We follow up our general Bayes factor convergence result with both theoretical and simulation based 
illustrations of asymptotic variable selection in linear regression model, nonparametric Gaussian process 
model where the exponential covariance function encapsulates the covariates to be selected, and a first order autoregressive model consisting of time-varying
covariates.

The rest of this paper is structured as follows. We introduce our general setup for Bayes factor based variable selection in Section \ref{sec:general_setup}.
%For better tractability, we divide our main result and proof of Bayes factor convergence into two sections. 
Section \ref{sec:as_conv} %\ref{sec:convergence_expectation} 
%provides the proof of convergence of the expectation of the log-Bayes factor, scaled by the sample size, while Section  uses this result to 
%finally complete the proof of 
shows almost sure convergence of the Bayes factor of any model with respect to the true model. 
Section \ref{sec:illustrations} %\ref{sec:linear_regression} and \ref{sec:gp_illustration}
provides illustrations of our main result with variable selection in linear regression and in a Gaussian process model with squared exponential covariance function.
Generalization of our results to the case of unknown error variance is provided in Section \ref{sec:unknown_error_variance}, using a conjugate prior.
Section \ref{sec:bf_int} provides further generalization of our result, assuming arbitrary priors on compact spaces for all other parameters and hyperparameters.
In Section \ref{sec:correlated_errors} we treat the case of correlated errors and present the problem of time-varying covariate selection in a first order autoregressive model
as an illustration, establishing almost sure exponential convergence of the relevant Bayes factor.
The important case of misspecification is dealt with in Section \ref{sec:misspecification}, where again almost sure exponential convergence of Bayes factor in favour 
of selection of the best possible subset of covariates, is established. 
In Section \ref{sec:overview_sim} an overview of our simulation and real data experiments are provided; complete details are relegated to the supplement. 
Finally, we summarize our work, make concluding remarks, and provide future directions
in Section \ref{sec:conclusion}.
%The auxiliary results, along with their proofs, are presented in the Appendix.

\section{General setup for Bayes factor based variable selection}
\label{sec:general_setup}
Let $y_i$ and $\bx_i$ denote the $i$-th response variable and the associated vector of covariates, $i=1,\ldots,n$.
We assume that the covariate $\bx$ consists of $p~(>1)$ components, and that it is required to select a subset of
the $p$ components that best explains the response variable $y$. We allow $p$ to grow with $n$ at a rate $p=O(n^{r})$, $r>0$.
%$0<r<\frac{1}{4\omega}$, for some $0<\omega<1$. As we shall bring out, this $\omega$ is related to the sparsity assumption related to our model}

Let $\bs$ denote any subset of the indices ${\bS} =\{1,2,\ldots,p\}$, and $\bx_{\bs}$ denote the co-ordinates of $\bx$ associated with $\bs$.
To relate $\bx_{\bs}$ to $y$ we consider the following nonparametric regression setup:
\begin{equation}
y=f(\bx_{\bs})+\epsilon,
\label{eq:nonpara_regression}
\end{equation}
where $\epsilon\sim N(0,\sigma^2_{\epsilon})$ is the random error and the function
$f(\cdot)$ is considered unknown. We assume that
$f:\mathfrak X\mapsto {\rm I\!R}$, where $\mathfrak X=\cup_{\ell=1}^p {\rm I\!R}^{\ell}$.
% ; and $\mathfrak X_{\ell}$ is the $\ell$-dimensional Euclidean space.

By assuming this framework we include the possibility that the domain of $f$ can range from one to $p$-dimensional. 
We further assume that there exists a true set of regressors, $\bx_0$, which influences the dependent variable $y$. 
%and therefore function $f(\bx_0)$ is the true function. 
Our problem is to identify $\bx_0$, i.e., the set of active regressors. 
Note that we do not consider any specific form of the function. Irrespective of the functional form, we are interested in identifying the set of active regressors.

\subsection{The Gaussian process prior}
\label{subsec:gp}
We assign a Gaussian process prior for $f(\cdot)$ which leads, for any given subset $\bs$ and covariate values $\left\{\bx_{i,\bs}; i=1,\ldots,n\right\}$, 
to the joint multivariate normal distribution of $\left(f(\bx_{1,\bs}),\ldots,f(\bs_{n,\bs})\right)^T$
with mean and variance-covariance matrix as follows:
\begin{align}
\bmu_{n,\bs}&=\left(\mu(\bx_{1,\bs}),\ldots,\mu(\bx_{n,\bs})\right)^T;
\notag \\ %\label{eq:mean_vector}\\
\bSigma_{n,\bs}&=\left((Cov\left(f(\bx_{i,\bs}),f(\bx_{j,\bs})\right)\right);~~i=1,\ldots,n;~j=1,\ldots,n.
\label{eq:cov_matrix}
\end{align}
The marginal distribution of $\by_n=(y_1,\ldots,y_n)^T$ is then the $n$-variate normal, given by
\begin{equation*}
\by_n\sim N_n\left(\bmu_{n,\bs},\sigma^2_{\epsilon} I_n+\bSigma_{n,\bs}\right),
%\label{eq:marginal}
\end{equation*}
where $ I_n$ is the identity matrix of order $n$. We denote this marginal model by $\mathcal M_{\bs}$.
It will be increasingly evident as we proceed, that this relatively simple consideration is the key to unlocking a sufficiently general asymptotic theory of Bayes factors
for variable selection that allows handling of wide range of situations including parametric, nonparametric, independence and dependence, using the same basic concept
and mathematical manoeuvre. 
%{\color{blue} Thus, we do away with the notion of independence, which has been the driving force so far for all the Bayes factor theories of variable selection.}

\subsection{The true model}
\label{subsec:true_model}
%\paragraph{The true model:} 
We assume that there exists exactly one particular subset $\bs_0$ of $\bS$ which is actually associated with the data generating process of $y$, which is termed as the \emph{true} subset. The evaluation procedure of the proposed set of model selection basically rests on its ability to identify this true subset, irrespective of the form of the function $f$. 
%In a sense that once such a set is identified, considerable amount of time and money could be saved by discarding the other regressors in future research, and this does not depend on the functional form of relation between the response and the regressors.

% Let us denote the true subset of covariate indices by $\bs_0$, and the true set of uniformly bounded basis functions
% by $$\left\{{\mathcal K}_{j,\bs_0}=\prod_{\ell\in\bs_0}K_{j\ell};~j=1,2,\ldots\right\}.$$
% %which best explains the responses $\by_n$.
% To distinguish the true model from the rest we add a index $t$ to the coefficients of the true model. The true function $f_t(\cdot)$ is then given by
% \begin{equation}
% f_t(\bx_{\bs_0})=\sum_{j=1}^{\infty}a^t_j\prod_{\ell\in\bs_0} K_{j\ell}(x_{\ell,\bs_0}),
% \label{eq:true_f}
% \end{equation}
% where $ a^t_j\sim N\left( m^t_j,\sigma^{2t}_j\right)$, with
% $\sum_{j=1}^{\infty} |m^t_j|<\infty$ and $\sum_{j=1}^{\infty} \sigma^{t}_j<\infty$.
We denote the mean vector and the covariance matrix of the Gaussian process prior associated with the true model
by $\bmu^t_{n,\bs_0}$ and $\bSigma^t_{n,\bs_0}$, respectively, and denote the corresponding
marginal distribution of $\by_n$ as ${\mathcal M}^t_{\bs_0}$.
For notational simplicity we drop the suffix $n$ from $\bmu_{n,\bs}$, $\bmu^t_{n,\bs_0}$, $\bSigma_{n,\bs}$ and $\bSigma^t_{n,\bs_0}$.

\subsection{The Bayes factor for covariate selection}
\label{subsec:BF}
It follows from the general model setup and the Gaussian process prior that the Bayes factor of any model $\mathcal M_{\bs}$ to the true model $\mathcal M^t_{\bs_0}$ 
associated with the data %$\left(\by_n, X\right)$, is given by % given uniform prior distribution on the model space is given by
is given by
\begin{eqnarray}
BF^{n}_{\bs,\bs_0}
&&=\frac{\mathcal M_{\bs}(\by_n)}{{\mathcal M}^{t}_{\bs_0}(\by_n)}= \frac{\left|\sigma^2_{\epsilon} I_n+\bSigma_{\bs}\right|^{-1/2}}{\left|\sigma^2_{\epsilon} I_n+\bSigma^{t}_{\bs_0}\right|^{-1/2}} \notag\\
&&\hspace{.55 in} \times \frac{
\exp\left\{-\left(\by_n-\bmu_{\bs}\right)^T\left(\sigma^2_{\epsilon} I_n+\bSigma_{\bs}\right)^{-1}
\left(\by_n-\bmu_{\bs}\right)/2\right\}}
{\exp\left\{-\left(\by_n-\bmu^t_{\bs_0}\right)^T
\left(\sigma^2_{\epsilon} I_n+\bSigma^{t}_{\bs_0}\right)^{-1}
\left(\by_n-\bmu^{t}_{\bs_0}\right)/2\right\}},
\label{eq:bf}
\end{eqnarray}
which is the ratio of the marginal likelihoods of the observed data $\by_n$, under the model $\mathcal M_{\bs}$ to the true model $\mathcal M^t_{\bs_0}$. This is the same as
the ratio of the posterior odds and prior odds for $\bs$ and $\bs_0$, for any prior on the models. 
If the models for $\bs$ and $\bs_0$ have the same prior distribution, then (\ref{eq:bf}) is the same as the posterior odds.

The main aim of this paper is to establish that (\ref{eq:bf}) converges to zero exponentially fast as $n\rightarrow\infty$, if $\bs\neq\bs_0$.
We shall begin with known $\sigma^2_{\epsilon}$ and other parameters, but will subsequently generalize our theory when such quantities are unknown, and
almost arbitrary, albeit sensible priors, are assigned to them.
In the next two sections we establish our main result on almost sure convergence of the log-Bayes factor.

\section{Almost sure convergence of the log-Bayes factor}
\label{sec:as_conv}
In this section we investigate Bayes factor consistency of Gaussian process regression in strong sense. We will show that for $\bs\neq\bs_0$, 
there exists an $\omega_{\bs}\in[0,1]$, and $\delta_{\bs}>0$ such that
 $$\limsup_n \frac{1}{n^{1+2r\omega_{\bs} }}\log BF^n_{\bs,\bs_0}\stackrel{a.s.}{=} -\delta_{\bs}.$$
The quantities $\omega_{\bs}$, for $\bs\subseteq\bS$, as we shall make precise in the applications, 
is related to the sparsity conditions of the underlying model $\bs_{0}$ and the competing model $\bs$. One way to interpret $\omega_{\bs}$ is to set $O(p^{\omega_{\bs}})=O(n^{r\omega_{\bs}})$ as the difference in effective
dimensionality of the true model $\bs_{0}$ and competing model $\bs$. 
Thus, when the effective dimensionality of the models indexed by $\bs$ and $\bs_{0}$ remain bounded, as $n\rightarrow\infty$, then $\omega_{\bs}=0$.
Note that depending upon the value of $\omega_{\bs}$, we can compare models of different dimensionalities. 
%In particular, if one of $\omega_{\bs}$ and $\omega_{\bs_0}$ is zero, then one of the competing models is of bounded dimensionality, while the other grows with $n$.

%we prove that the Bayes factor  of the model induced by $\bs$ and that induced by $\bs_0$ converges to zero almost surely, for any choice of $\bs$.
%that the expected log Bayes factor of the model induced by $\bs$ for any choice of $\bs$ and that induced by $\bs_0$ converges to a negative quantity under $\bs_0$, when normalized by $n$. 
%$$\max_{\bs\neq \bs_0} \limsup_n \frac{1}{n}\log BF^n_{\bs,\bs_0}\xrightarrow{p} -\delta ,\quad \mbox{ for some ~~} \delta>0.$$
We first state the assumptions under which the result holds.
\begin{itemize}
\item[$\left(A1\right)$] Let
$\Delta_{n,\bs}\stackrel{def}{=}(\bmu_{\bs}-\bmu^t_{\bs_0})^T \left(\sigma^2_{\epsilon} I_n+\bSigma_{\bs}\right)^{-1}
		(\bmu_{\bs}-\bmu^t_{\bs_0})$. We assume that for any $\bs\subseteq \bS$, for some $\omega_{\bs}\in[0,1]$ 
		and $\xi_{\bs}>0$,
	 $$\liminf_n \frac{1}{n^{1+2r\omega_{\bs}}} \Delta_{n,\bs} = \xi_{\bs}.$$
\end{itemize}
%\vskip5pt , as $n\rightarrow\infty$

%=B_{\bs_0} B^T_{\bs_0}\left(\sigma^2_{\epsilon} I_n+\bSigma_{n,\bs}\right)^{-1}
\noindent Define $A_{n,\bs}=\left(\sigma^2_{\epsilon} I_n+\bSigma_{\bs_0}^{t}\right)\left(\sigma^2_{\epsilon} I_n+\bSigma_{\bs}\right)^{-1}$. We further assume the following:
\begin{itemize}
%\item[($A6^\prime$)] $tr(A_n)=O\left(n^q\right)$, for $q<1$.
\item[($A2$)] Let $\lambda_1\geq \cdots \geq \lambda_n>0$ be the eigenvalues of $A_{n,\bs}$, then for $\omega_{\bs}$ defined in (A1),
~ $\lambda_{\max}(A_{n,\bs}) =O\left(p^{2\omega_{\bs}} \right)=O\left(n^{2r\omega_{\bs}} \right)$.
\vskip5pt
% $tr(A_n)=\zeta n+O\left(1\right)$, where $\zeta\geq 1$. %for $r<2$.
\item[($A3$)] Finally we assume that for all $\bs$, and for $\omega_{\bs}$ defined in (A1),

{\centering 	$\|\bmu_{\bs}-\bmu^t_{\bs_0}\|^2=O\left(n^{1+b}p^{2\omega_{\bs}}\right)
		=O\left(n^{1+b+2r\omega_{\bs}}\right)$, for some $b<1/2$.
		\par}
	\vskip10pt
%	{\color{red} where $0\leq\omega\leq 1$.} %and $0<r<\frac{1}{4\omega}$.} %=O\left(n^{(3-r)/2}\right)$ for some $r>0$.
% $\Sigma_{n,\bs}-\Sigma^t_{n,\bs_0}$ is non-negative definite. This assumption is justified on the ground that the data under the true model is often likely
% to have smaller variability than that under any other model.
% \item[($A6$)] $\frac{\sigma^2_{\epsilon}+\lambda_i\left(\bSigma^{t}_{n,\bs_0}\right)}{\sigma^2_{\epsilon}+\lambda_i\left(\bSigma_{n,\bs}\right)}\rightarrow c_{\bs}$
% as $i\rightarrow\infty$, where $0<c_{\bs}\leq 1$. The condition $0<c_{\bs}\leq 1$ is a consequence of ($A5$).
\end{itemize}

We will show that, the quantity $\Delta_{n,\bs}$ in (A1) is asymptotically equivalent to the Kullback-Leibler (KL) divergence between the marginal density of ${\bf y}_{n}$ 
under $\bs$ and that under $\bs_{0}$, in most of the frameworks including linear model. 
Thus requiring (A1) is same as requiring positive KL divergence between $\mathcal M_{\bs}$ and $\mathcal M^t_{\bs_0}$ after proper scaling.  
%Further, if we assume that the orders of the effective dimensionalities of the competing models are known, 
%and the sparsity parameter $\omega_{\bs}$ is chosen based on that information, 
%then (A1) seems to be a reasonable assumption. 
Assumptions (A2) and (A3) are reasonable and verifiable restrictions.

In our illustrations with linear and Gaussian process regression, 
we will show that $p^{\omega_{\bs}}$ can be interpreted essentially as the cardinality of set difference of $\bs$ and $\bs_{0}$. 
%i.e., $O\left(p^{\omega_{\bs}}=|\bs \Delta\bs_{0}|$. 
Further, with our illustration with a first-order autoregressive regression model, we demonstrate that if $\|\bmu_{\bs}\|^2=O\left(n\|\bs\|^2\right)$, 
then $p^{\omega_{\bs}}$ may be interpreted essentially as $\max\left\{|\bs|,|\bs_0|\right\}$.

%Henceforth, unless mentioned otherwise, we shall use $\omega=\max\{\omega_{\bs},\omega_{\bs_0}\}$ for notational convenience, since we shall generally
%focus on comparing two models, one being the true model $\bs_0$, and another model indexed by $\bs~(\neq\bs_0)$. Further results, for comparing
%two wrong models (the misspecified situation), will be derived from this principle. 

%Henceforth, unless mentioned otherwise, we shall write $\omega$ instead of $\omega_{\bs}$ for notational convenience.
%, as the result on Bayes factor consistency depends on the two models $\bs$ and $\bs_{0}$ only, and $\bs_{0}$ is assumed to be fixed.

Our first result shows that limit supremum of the expected log Bayes factor of any model and the true model is negative, when scaled by $n^{1+2r\omega_{\bs}}$. 
\begin{result}
\label{theorem:mean_bf}
Assume ($A1$) holds for some $\omega_{\bs}\in[0,1]$. Then for some $\delta_{\bs}>0$ depending upon $\bs$ ($\neq\bs_0$),
\begin{equation}
%\max_{\bs\neq\bs_0} 
	\underset{n\rightarrow\infty}{\limsup}~E_{\bs_0} \left(\frac{1}{n^{1+2r\omega_{\bs}}}\log BF^{n}_{\bs,\bs_0}\right)= -\delta_{\bs}, \notag
% \label{eq:conv_bf_mean}
\end{equation}
	for the same choice of $\omega_{\bs}$ as given in (A1).
\end{result}
% =\frac{1}{2}\left(\log c^{-1}_{\bs}+\zeta+\xi_{\bs}-1\right)
\begin{proof}
From (\ref{eq:bf}) we find that the expectation of logarithm of the Bayes factor is given by
\begin{align}
	\frac{1}{n^{1+2r\omega_{\bs}}}&E_{\bs_0}\left[\log\left(BF^{n}_{\bs,\bs_0}\right)\right]\notag \\
	&=\frac{1}{2n^{1+2r\omega_{\bs}}}\log \frac{\left|\sigma^2_{\epsilon}  I_n+\bSigma^{t}_{\bs_0}\right|}{\left|\sigma^2_{\epsilon} I_n+\bSigma_{\bs}\right|} \notag	\\
	 & \qquad -\frac{1}{2n^{1+2r\omega_{\bs}}}E_{\bs_0}\left[\left(\by_n-\bmu_{\bs}\right)^T
\left(\sigma^2_{\epsilon}I_n+\bSigma_{\bs}\right)^{-1}
\left(\by_n-\bmu_{\bs}\right)\right]\notag\\
	&\qquad +\frac{1}{2n^{1+2r\omega_{\bs}}}E_{\bs_0}\left[\left(\by_n-\bmu^t_{\bs_0}\right)^T
\left(\sigma^2_{\epsilon} I_n+\bSigma^t_{\bs_0}\right)^{-1}
\left(\by_n-\bmu^t_{\bs_0}\right)\right].
\label{eq:mean_log_BF_1}
\end{align}
To evaluate the first part in the above equation, note that
\begin{eqnarray}
	\frac{1}{2n^{1+2r\omega_{\bs}}} \log \frac{\left|\sigma^2_{\epsilon} I_n+\bSigma^{t}_{\bs_0}\right|}{\left|\sigma^2_{\epsilon} I_n+\bSigma_{\bs}\right|}
	=\frac{1}{2n^{1+2r\omega_{\bs}}} \log \left|A_{n,\bs} \right|=\frac{1}{2n^{1+2r\omega_{\bs}}} \sum_{j=1}^{n} \log \lambda_j (A_{n,\bs}).\notag
%  &=& \frac{1}{2n}\log \prod_{i=1}^n\frac{\sigma^2_{\epsilon}+\lambda_i\left(\bSigma^{t}_{n,\bs_0}\right)}{\sigma^2_{\epsilon}+\lambda_i\left(\bSigma_{n,\bs}\right)}\notag\\
%  &=&\frac{1}{2n}\sum_{i=1}^{n} \log \left[\frac{\sigma^2_{\epsilon}+\lambda_i\left(\bSigma^{t}_{n,\bs_0}\right)}{\sigma^2_{\epsilon}+\lambda_i\left(\bSigma_{n,\bs}\right)}\right]\notag\\
% &\rightarrow &\frac{1}{2}\log c_{\bs},~\mbox{as $n\rightarrow\infty$}~~\mbox{[due to ($A6$)]}.
%\label{eq:first_part}
\end{eqnarray}
% Note that $\frac{1}{2}\log c_{\bs}<0$, due to (A5).
For the second term of (\ref{eq:mean_log_BF_1}) we obtain
\begin{align}
E_{\bs_0}\left[\left(\by_n-\bmu_{\bs}\right)^T\left(\sigma^2_{\epsilon} I_n+\bSigma_{\bs}\right)^{-1}
\left(\by_n-\bmu_{\bs}\right)\right]\hspace{2 in} \notag\\
%&=tr\left[\left(\sigma^2_{\epsilon}I_n+\bSigma_{\bs}\right)^{-1}
%\left(\sigma^2_{\epsilon} I_n+\bSigma^t_{\bs_0}\right)\right]+\left(\bmu_{\bs}-\bmu_{\bs_0}\right)^T\left(\sigma^2_{\epsilon} I_n+\bSigma_{\bs}\right)^{-1}
%\left(\bmu_{\bs}-\bmu^t_{\bs_0}\right)\quad \notag\\
= tr(A_{n,\bs})
+\left(\bmu_{\bs}-\bmu_{\bs_0}^{t}\right)^T\left(\sigma^2_{\epsilon} I_n+\bSigma_{\bs}\right)^{-1}
\left(\bmu_{\bs}-\bmu^t_{\bs_0}\right)= tr(A_{n,\bs})+\Delta_{n,\bs} .\notag
% \\
% %&=O\left(n^q\right)+\left(\bmu_{n,\bs}-\bmu_{n,\bs_0}\right)^T\left(\sigma^2_{\epsilon} I_n+\bSigma_{n,\bs}\right)^{-1}
% &=\frac{\zeta}{2}+O\left(n^{-1}\right)+\frac{\xi_{\bs}}{2}+O\left(n^{q-1}\right)~~\mbox{[due to ($A3$) and ($A4$)]}\notag\\
% &\rightarrow \frac{\zeta}{2}+\frac{\xi_{\bs}}{2},~\mbox{as $n\rightarrow\infty$}.
\label{eq:second_part}
\end{align}
The last term of (\ref{eq:mean_log_BF_1}) is given by
\begin{equation}
E_{\bs_0}\left[\left(\by_n-\bmu_{\bs_0}^{t}\right)^T\left(\sigma^2_{\epsilon} I_n+\bSigma_{\bs_0}^{t}\right)^{-1}
\left(\by_n-\bmu_{\bs_0}^{t}\right)\right]=n.\notag
%\label{eq:third_part}
\end{equation}
Using the above facts and from (\ref{eq:mean_log_BF_1}) observe that
\begin{eqnarray}
2E_{\bs_0}\left[\log\left(BF^{n}_{\bs,\bs_0}\right)\right]
+\Delta_{n,\bs} =\sum_{i=1}^{n} \left(\log \lambda_i - \lambda_i +1 \right). \notag
\end{eqnarray}
Note that $g(x)=\log x-x+1$ is an increasing function on $(0,1]$ and decreasing function on $(1,\infty)$, having maximum at $0$. Thus $\sum_i \left(\log \lambda_i-\lambda_i+1 \right) \leq 0$.

Thus, combining the above facts and (A1) we write
% so that combining (\ref{eq:first_part}), (\ref{eq:second_part}) and (\ref{eq:third_part}) yields
\begin{equation}
	\limsup_n \frac{1}{n^{1+2r\omega_{\bs}}}E_{\bs_0}\left[\log\left(BF^{n}_{\bs,\bs_0}\right)\right]
\leq -\xi_\bs/2. \notag %\label{eq:lnBF_order}
\end{equation}
%Noting the facts that there are finite number of models the proof follows. %, i.e., finite number of $\xi_\bs$, and $\limsup$ always exists,
% The result (\ref{eq:conv_bf_mean}) follows from (\ref{eq:lnBF_order}).
Hence, there exists $\delta_{\bs}>0$ depending upon $\bs$ such that
\begin{equation}
%\max_{\bs\neq\bs_0} 
	\underset{n\rightarrow\infty}{\limsup}~E_{\bs_0} \left(\frac{1}{n^{1+2r\omega_{\bs}}}\log BF^{n}_{\bs,\bs_0}\right)= -\delta_{\bs}.  \notag
% \label{eq:conv_bf_mean}
\end{equation} \qed
\end{proof}

\noindent Next we will prove $L_4$ convergence of $\log\left(BF^{n}_{\bs,\bs_0}\right)/n^{1+2r\omega_{\bs}}$ towards its expectation, which in turn would
imply $L_2$ convergence.

\vskip5pt
Let $B_{\bs_0}$ be the appropriate matrix associated with the Cholesky factorization of $\sigma^2_{\epsilon} I_n+\bSigma^t_{\bs_0}$, i.e., $\sigma^2_{\epsilon} I_n+\bSigma^t_{\bs_0}=B_{\bs_0}B^T_{\bs_0}$, and $C_{n,\bs}=B^T_{\bs_0}\left(\sigma^2_{\epsilon} I_n+\bSigma_{\bs}\right)^{-1}B_{\bs_0}$.
Then $\by_n-\bmu^t_{\bs_0}=B_{\bs_0}\bz_n$, with $\bz_n\sim N_n\left(\bzero,I_n\right)$. Then
$$\left(\by_n-\bmu^t_{\bs_0}\right)^T\left(\sigma^2_{\epsilon} I_n+\bSigma^t_{\bs_0}\right)^{-1}
\left(\by_n-\bmu^t_{\bs_0}\right)=\bz^T_n\bz_n,$$ and
%\begin{align*}
$\left(\by_n-\bmu^t_{\bs_0}\right)^T\left(\sigma^2_{\epsilon} I_n+\bSigma_{\bs}\right)^{-1}
\left(\by_n-\bmu^t_{\bs_0}\right)
=\bz_n^T C_{n,\bs} \bz_n. $
%=\bz^T_nB^T_{\bs_0}\left(\sigma^2_{\epsilon} I_n+\bSigma_{\bs}\right)^{-1}B_{\bs_0}\bz_n
%\label{eq:power4_1}
%\end{align*}
Note further that $A_{n,\bs}$ and $C_{n,\bs}$ have the same eigenvalues. Thus, by assumption (A2), 
\begin{equation}
	\lambda_{\max}(C_{n,\bs})=O\left(n^{2r\omega_{\bs}}\right).
	\label{eq:A2_Cn}
\end{equation}
% \begin{equation*}
%  E\left(\bz^T_nB^T_{\bs_0}\left(\sigma^2_{\epsilon} I_n+\bSigma_{\bs}\right)^{-1}B_{\bs_0}\bz_n\right)=tr(A_{n,\bs}).
% \end{equation*}

\vskip5pt

\begin{result}
\label{theorem:power4}
	Assume ($A2$) and ($A3$) hold for some $\omega_{\bs}\in[0,1]$. Then
\begin{equation*}
%E_{\bs_0}\left[\frac{1}{n}\log\left(BF^{n}_{\bs,\bs_0}\right)+\delta_{\bs}\right]^4=O\left(n^{4(q-1)}\right).
%\sum_{n=1}^{\infty}E_{\bs_0}\left[\frac{1}{n}\log\left(BF^{n}_{\bs,\bs_0}\right)-E_{\bs_0}\left\{ \displaystyle\frac{1}{n}\log\left(BF^{n}_{\bs,\bs_0}\right) \right\}  \right]^4<\infty.
	n^{-1-2r\omega_{\bs}}\left\{\log\left(BF^{n}_{{\bs},{\bs_0}}\right)-E_{\bs_0}\left[\log\left(BF^{n}_{\bs,\bs_0}\right)\right] \right\}
	\stackrel{a.s.}{\longrightarrow}0,~\mbox{as}~n\rightarrow\infty.
%\label{eq:power4}
\end{equation*}
\end{result}
\begin{proof}
For convenience,
% we shall work with
% $$E_{\bs_0}\left[\frac{1}{n}\left(\log\left(BF^{n}_{\bs,\bs_0}\right)-\tilde E_n\right)
% +\frac{1}{n}\tilde E_n+\delta_{\bs}\right]^4,$$
% where
we write $\tilde E_n : =E_{\bs_0}\left[\log\left(BF^{n}_{\bs,\bs_0}\right)\right]$.
% Observe that
% \begin{equation}
% E_{\bs_0}\left[\frac{1}{n}\left(\log\left(BF^{n}_{\bs,\bs_0}\right)-\tilde E_n\right)+\frac{1}{n}\tilde E_n+\delta_{\bs}\right]^4
% \leq 8\left\{n^{-4}E_{\bs_0}\left[\log\left(BF^{n}_{\bs,\bs_0}\right)-\tilde E_n\right]^4+\left[\frac{1}{n}\tilde E_n+\delta_{\bs}\right]^4\right\}.
% \label{eq:bf_bound}
% \end{equation}
\noindent Now note that for $A_{n,\bs}$, $\bz_n$, $B_{\bs_0}$ and $C_{n,\bs}$ as defined above
\begin{eqnarray}
&&2E_{\bs_0}\left[\log\left(BF^{n}_{\bs,\bs_0}\right)-\tilde E_n\right]^4 \hspace{2.5 in}\notag\\
 &=& E_{\bs_0}\left[-\bz^T_nC_{n,\bs} \bz_n+E_{\bs_0}\left(\bz^T_nC_{n,\bs}\bz_n \right)
+2\bz^T_n B_{\bs_0}^T\left(\sigma^2_{\epsilon}I_n+\Sigma_{\bs}\right)^{-1}\left(\bmu_{\bs}-\bmu^t_{\bs_0}\right)+\right. \notag \\
&& \hspace{3.25 in} \left. \bz^T_n\bz_n-n \right]^4 \notag\\
 &\leq& C\left[E_{\bs_0}\left|\bz^T_n C_{n,\bs} \bz_n-tr(C_{n,\bs})\right|^4
+E_{\bs_0}\left|\bz^T_n B_{\bs_0}^T\left(\sigma^2_{\epsilon}I_n+\Sigma_{\bs}\right)^{-1}\left(\bmu_{\bs}-\bmu^t_{\bs_0}\right)\right|^4 \right. \notag\\
&& \left. \hspace{2.75 in} +E_{\bs_0}\left|\bz^T_n\bz_n-n\right|^4\right]
\label{eq:power4_bound1}
\end{eqnarray}
where $C$ is a positive constant. The above result follows by repeated application of the inequality $(a+b)^q\leq 2^{q-1}(a^q+b^q)$, for non-negative $a$, $b$, where $q\geq 1$.

We first obtain the asymptotic order of the first term of (\ref{eq:power4_bound1}). Note that for any $n$ vector $\bz_n$ and any $n\times n$ matrix $C_n$
\begin{align}
&E\left\{\bz^T_n C_n\bz_n-E_{\bs_0}\left(\bz^T_n C_n\bz_n\right)\right\}^4=E\left(\bz^T_n C_n\bz_n\right)^4\notag\\
&\hspace{.7 in}-4E\left(\bz^T_n C_n\bz_n\right)^3E\left(\bz^T_n C_n\bz_n\right)
+6E\left(\bz^T_n C_n\bz_n\right)^2\left\{E\left(\bz^T_n C_n\bz_n\right)\right\}^2\notag\\
&\hspace{.8 in}-4E\left(\bz^T_n C_n\bz_n\right)\left\{E\left(\bz^T_n C_n\bz_n\right)\right\}^3
+\left\{E\left(\bz^T_n C_n\bz_n\right)\right\}^4.
\label{eq:quad_binomial1}
\end{align}
To evaluate (\ref{eq:quad_binomial1}), we make use of the following results (see, for example, \ctn{Magnus78}, \ctn{Kendall47}).
\begin{align*}
E_{\bs_0}\left(\bz^T_n C_{n,\bs} \bz_n\right)&=tr\left(C_{n,\bs}\right);  \\ %\label{eq:quadmoment1}
E_{\bs_0}\left(\bz^T_n C_{n,\bs} \bz_n\right)^2&=\left[tr\left(C_{n,\bs}\right)\right]^2+2tr\left(C^2_{n,\bs} \right);\\ % \label{eq:quadmoment2}
E_{\bs_0}\left(\bz^T_n C_{n,\bs} \bz_n\right)^3&=\left[tr\left(C_{n,\bs}\right)\right]^3+6tr\left(C_{n,\bs} \right)tr\left(C^2_{n,\bs} \right)+8tr\left(C^3_{n,\bs} \right);\\ % \label{eq:quadmoment3}
E_{\bs_0}\left(\bz^T_n C_{n,\bs} \bz_n\right)^4&=\left[tr\left(C_{n,\bs} \right)\right]^4+32tr\left(C_{n,\bs} \right)tr\left(C^3_{n,\bs} \right)+12\left[tr\left(C^2_{n,\bs} \right)\right]^2\\ %\notag
&\qquad\quad+12\left[tr\left(C_{n,\bs} \right)\right]^2tr\left(C^2_{n,\bs} \right)+48tr\left(C^4_{n,\bs} \right). %\label{eq:quadmoment4}
\end{align*}
Substituting the above expressions in (\ref{eq:quad_binomial1}) we obtain
\begin{equation}
E_{\bs_0}\left\{\bz^T_nC_{n,\bs} \bz_n-E_{\bs_0}\left(\bz^T_n C_{n,\bs} \bz_n\right)\right\}^4
=12\left[tr\left(C_{n,\bs}^2\right)\right]^2+48tr\left(C_{n,\bs}^4\right). \notag
%\label{eq:quad_binomial2}
\end{equation}
If $\lambda_1, \ldots, \lambda_n$ are the eigenvalues of $C_{n,\bs}$, then $\lambda^k_1,\ldots,\lambda^k_n$ are the eigenvalues of 
$C_{n,\bs}^k$, for $k\in \mathbb{N}$. Therefore the above quantity reduces to
\begin{eqnarray}
% E_{\bs_0}\left\{\bz^T_nA_n\bz_n-E_{\bs_0}\left(\bz^T_nA_n\bz_n\right)\right\}^4
%12\left[tr\left(C_{n,\bs}^2\right)\right]^2+48tr\left(C_{n,\bs}^4\right)\hspace{2 in}\notag \\
	12 \left(\sum_i \lambda_i^2\right)^2+48\sum_i \lambda_i^4 \leq Cn \sum_i \lambda_i^4=O\left(n^{2+8r\omega_{\bs}}\right),
\label{eq:quad_binomial3}
\end{eqnarray}
due to %($A2$) 
(\ref{eq:A2_Cn}) and the fact that $\left( \sum_{i=1}^{n} a_i \right)^2 \leq n \sum_{i=1}^{n} a_i^2$.

Let us now obtain the asymptotic order of second term of (\ref{eq:power4_bound1}).
% $E_{\bs_0}\left\{\left(\by_n-\bmu^t_{n,\bs_0}\right)^T\left(\sigma^2_{\epsilon}I_n+\Sigma_{n,\bs}\right)^{-1}\left(\bmu_{n,\bs}-\bmu^t_{n,\bs_0}\right)\right\}^4$.
Note that, the random variable
$\bz^T_n B_{\bs_0}^T\left(\sigma^2_{\epsilon}I_n+\Sigma_{\bs}\right)^{-1}\left(\bmu_{\bs}-\bmu^t_{\bs_0}\right)$ is univariate normal
with mean zero and variance
\begin{equation*}
\hat\sigma^2_n=\left(\bmu_{\bs}-\bmu^t_{\bs_0}\right)^T\left(\sigma^2_{\epsilon}I_n+\Sigma_{\bs}\right)^{-1}\left(\sigma^2_{\epsilon}I_n+\Sigma^t_{\bs_0}\right)
\left(\sigma^2_{\epsilon}I_n+\Sigma_{\bs}\right)^{-1}\left(\bmu_{\bs}-\bmu^t_{\bs_0}\right).
%\label{eq:var2}
\end{equation*}
Observe that
\begin{align}
&\lambda_{\max}\left[\left(\sigma^2_{\epsilon}I_n+\Sigma_{\bs}\right)^{-1}\left(\sigma^2_{\epsilon}I_n+\Sigma^t_{\bs_0}\right)
	\left(\sigma^2_{\epsilon}I_n+\Sigma_{\bs}\right)^{-1}\right]\notag\\
&\leq\lambda_{\max}\left[\left(\sigma^2_{\epsilon}I_n+\Sigma_{\bs}\right)^{-1}\right]
	\lambda_{\max}\left[\left(\sigma^2_{\epsilon}I_n+\Sigma^t_{\bs_0}\right)\left(\sigma^2_{\epsilon}I_n+\Sigma_{\bs}\right)^{-1}\right]\notag\\
&=\frac{\lambda_{\max}(A_{n,\bs})}{\lambda_{\min}(\sigma^2_{\epsilon}I_n+\Sigma_{\bs})}\leq\sigma^{-2}_{\epsilon}\lambda_{\max}(A_{n,\bs}),\notag %\label{eq:eq1} %=O(1),~\mbox{due to (A2).}\notag
\end{align}
by Result \ref{theorem:wang} (in Section \ref{sec:appendix} of the supplement).
 Hence, 
 $$\left(\sigma^2_{\epsilon}I_n+\Sigma_{\bs}\right)^{-1}\left(\sigma^2_{\epsilon}I_n+\Sigma^t_{\bs_0}\right)
\left(\sigma^2_{\epsilon}I_n+\Sigma_{\bs}\right)^{-1}\leq \sigma^{-2}_{\epsilon}\lambda_{\max}(A_{n,\bs})I_n.$$
%Note that, as $\left(\sigma^2_{\epsilon}I_n+\Sigma_{n,\bs}\right)^{-1}\leq I$, and
 Therefore, %$\hat\sigma^2_n\leq \Delta_{n,\bs}
$\hat\sigma^2_n\leq \sigma^{-2}_{\epsilon}\lambda_{\max}(A_{n,\bs}) \|\bmu_{\bs}-\bmu^t_{\bs_0}\|^2=O\left(n^{1+b+4r\omega_{\bs}}\right)$, %O(n^{(3-r)/2})$ for $r>0$.
due to ($A2$) and ($A3$).
% By assumption ($A5$) it follows that $\hat\sigma^2_n\leq \Delta_{n,\bs}$, so that $\hat\sigma^2_n=O\left(n\right)$.
Hence it follows that
\begin{equation}
E_{\bs_0}\left|\bz^T_n B_{\bs_0}^T\left(\sigma^2_{\epsilon}I_n+\Sigma_{\bs}\right)^{-1}\left(\bmu_{\bs}-\bmu^t_{\bs_0}\right)\right|^4
	=3\hat\sigma^4_n=O\left(n^{2+2b+8r\omega_{\bs}}\right). %O\left(n^{3-r}\right).
\label{eq:var2_order}
\end{equation}
Finally, we deal with the third term of (\ref{eq:power4_bound1}).
%$E_{\bs_0}\left\{\left(\by_n-\bmu^t_{n,\bs_0}\right)^T\left(\sigma^2_{\epsilon}I_n+\Sigma^t_{n,\bs_0}\right)^{-1}\left(\by_n-\bmu^t_{n,\bs_0}\right)-n\right\}^4$
%which is the same as $E_{\bs_0}\left(\bz^T_n\bz_n-n\right)^4$.
%where $\tilde E^{(2)}_n=E_{\bs_0}\left(\bz^T_n\bz_n\right)=n$.
As $\bz^T_n\bz_n-n =
\sum_{i=1}^n\left(z^2_i-1\right)$, where, for $i=1,\ldots,n$,
$z^2_i\stackrel{iid}{\sim}\chi^2_1$. By Lemma B of \ctn[p. 68]{Serfling80}, it follows that
\begin{equation}
E_{\bs_0}\left(\bz^T_n\bz_n-n\right)^4=O\left(n^2\right).
\label{eq:power4_4}
\end{equation}
Substituting (\ref{eq:quad_binomial3}), (\ref{eq:var2_order}) and (\ref{eq:power4_4}) in (\ref{eq:power4_bound1}) we obtain
\begin{equation}
	E_{\bs_0}\left[\log\left(BF^{n}_{\bs,\bs_0}\right)-\tilde E_n\right]^4=O\left(n^{2+2b+8r\omega_{\bs}}\right). %=O\left(n^{3-r}\right).
\label{eq:power4_order1}
\end{equation}

\noindent Chebychev's inequality, in conjunction with (\ref{eq:power4_order1}) guarantees that for any $\eta>0$,
\begin{equation*}
	\sum_{n=1}^{\infty}P_{\bs_0}\left(\left|\log\left(BF^{n}_{\bs,\bs_0}\right)-\tilde E_n \right|>n^{1+2r\omega_{\bs}}\eta\right)<\infty,
\end{equation*}
%as ${\color{red}r <\frac{1}{4\omega}}$, 
as $b<1/2$, proving almost sure convergence of $n^{-1-2r\omega_{\bs}}\left\{\log\left(BF^{n}_{{\bs},{\bs_0}}\right)-\tilde E_n \right\}$
to $0$, as $n\rightarrow\infty$.\qed
\end{proof}

Now we state the main theorem, the proof of which follows as an application the above result, and Result \ref{theorem:mean_bf}.
\begin{theorem}[Main theorem]
\label{theorem:as_conv}
	Suppose the assumptions ($A1$)--($A3$) hold for some $\omega_{\bs}\in[0,1]$, and $\delta_{\bs}>0$ depending upon $\bs$ ($\neq\bs_0$), then 
\begin{equation}
%\max_{\bs\neq\bs_0} 
	\limsup_n \frac{1}{n^{1+2r\omega_{\bs}}}\log\left(BF^{n}_{\bs,\bs_0}\right)\stackrel{a.s.}{=} -\delta_{\bs}.  \notag %\label{eq:bf_conv}
\end{equation}
\end{theorem}
\begin{remark}
	\label{remark:remark1}
	Recall that $p^{\omega_{\bs}}$ is related to the effective dimensionality of the model indexed by $\bs$. 
%	Hence, when $\omega_{\bs}=\omega_{\bs_0}=0$, the asymptotic result (\ref{eq:bf_conv})
%	is expected to agree with that for the fixed $p$ situation. 
When $p$ is fixed, then $p^{\omega_{\bs}}$, which can also be interpreted essentially as the difference in the set of covariates under $\bs$ and $\bs_{0}$, is zero. 
	Indeed, keeping $p$ fixed and proceeding exactly in the same way as the proof of Theorem \ref{theorem:as_conv}, and setting $\omega_{\bs}=0$ in assumptions ($A1$)--($A3$), would yield the result 
\begin{equation}
%\max_{\bs\neq\bs_0} 
	\limsup_n \frac{1}{n}\log\left(BF^{n}_{\bs,\bs_0}\right)\stackrel{a.s.}{=} -\delta_{\bs}. \notag
%\label{eq:bf_conv2}
\end{equation}
Further, if $p$ were fixed, then the number of models $2^p-1$, would be finite. In that case, under assumptions ($A1$)--($A3$) (with $\omega_{\bs}=0$ for all $\bs\in\bS$), 
	there would exist $\delta>0$, such that
\begin{equation}
\max_{\bs\neq\bs_0} 
	\limsup_n \frac{1}{n}\log\left(BF^{n}_{\bs,\bs_0}\right)\stackrel{a.s.}{=} -\delta. \notag
%\label{eq:conv_as}
\end{equation}
\end{remark}

\begin{remark}
	\label{remark:remark2}
One can establish a relatively weaker version of consistency result, 
	$$\limsup_n \frac{1}{n^{1-\epsilon+2r\omega_{\bs}}}\log\left(BF^{n}_{\bs,\bs_0}\right)\stackrel{a.s.}{=} -\delta_{\bs}, \qquad \epsilon<1/4, $$
	under a weaker variant of assumption (A1):
	$(A1^{\star})$ $\liminf_{n} n^{-1+\epsilon- 2r\omega_{\bs}} \Delta_{n,\bs} = \xi_{\bs}.$
	However, assumption (A3) should be replaced by $(A3^{\star}):$ $\|\bmu_{\bs}-\bmu^t_{\bs_0}\|^2=O\left(np^{2\omega_{\bs}}\right)=O\left(n^{1+2r\omega_{\bs}}\right)$, for all $\bs$, which,  nonetheless, remains a mild assumption. When $r\omega_{\bs}$ is large, this version of consistency becomes more appropriate than the traditional one.
\end{remark}

\begin{remark}
	\label{remark:remark3}
	Note that Theorem \ref{theorem:as_conv} remains valid even for nested models $\mathcal M_{\bs}$ and $\mathcal M_{\bs_0}$ where one model has $O(p^{\omega_{\bs}})$ 
	number of covariates more than the other, where $\omega_{\bs} \in (0,1]$. 
	%However, if $\omega=0$, then in assumptions (A1)--(A3), $\omega$ must be set to zero.
	%Then (\ref{eq:bf_conv}) of Theorem \ref{theorem:as_conv} holds with $\omega=0$.} 
\end{remark}

\section{Illustrations}
\label{sec:illustrations}
This section provides illustrations of our main result in two different regression contexts, linear regression and Gaussian process regression with squared exponential covariance function.
\subsection{Linear regression}
\label{sec:linear_regression}
For illustration of our Bayes factor theory let us first consider the linear regression.
Let $y_i=\bbeta_{\bs}^T \bx_{i,\bs}+\epsilon_i$, where
$\epsilon_i\stackrel{iid}{\sim}N\left(0,\sigma^2_{\epsilon}\right)$, for $i=1,\ldots,n$. Let
$\bbeta_{\bs}\sim N\left(\bbeta_{0,\bs},g_n \sigma^2_{\bbeta}\left( X_{\bs}^T X_{\bs}\right)^{-1} \right)$.
% where $ I_{\left|\bs\right|}$ is the identity matrix of order $\left|\bs\right|$, the cardinality of $\bs$.
%This is the well known Zellner's $g$ prior. Here we assume $\bS=\{1,2,\ldots,p\}$ where $p=O(n^r) ~(r>0)$, and 
Let $\bs_0~(\subseteq\bS=\{1,2,\ldots,p\})$ be the set of
indices of the true set of covariates, and $p=O(n^r) ~(r>0)$.
Zellner's $g$-prior assigns  $\bbeta_{0,\bs}={\bf 0}$. We instead make our prior more flexible by assuming that for all $\bs$, 
$\|\bbeta_{0,\bs}\|_{L_1}=\sum_{j=1}^{|\bs|}|\beta_{0,j}|=O(|\bs|)$. We further assume that $g_n=O(p^{\omega_{\bs}})$. %$O(p^{\omega}n^{c_{0}})$, for some $\omega\in[0,1]$. and $c_{0}<1/4$. 
%for some $\omega_{\bs}\in[0,1]$.
%Note that, since $|\bs|\leq p$, $\omega_{\bs}$ can not exceed $1$. 
%Thus, $O(p^{\omega_{\bs}})$ describes the effective model dimensionality, 
%with $\omega_{\bs}=0$ denoting that the dimensionality remains bounded as $n\rightarrow\infty$
%.}

We assume that the space of covariates is compact, which, as we show, is sufficient to ensure ($A1$)--($A3$).
% We assume that the space of covariates is compact. This assumption is sufficient to ensure our requisite uniform boundedness, so that (A1) and (A2) are not explicitly needed. Further, let the set of covariates $\left\{x_j:j\in\bS\right\}$ be non-zero.
Observe that Zellner's $g_n$-prior induces a Gaussian process prior on the function $f(\bx_{i,\bs})=\bx_{i,\bs}^T \bbeta_\bs$ with mean function
% For $i=1,\ldots,n$, the mean of $\bbeta_{\bs}^{\prime}\bx_{i,\bs}$ is given by
\begin{equation*}
\mu\left(\bx_{i,\bs}\right)=\bbeta_{0,\bs}^T\bx_{i,\bs}=\bmu_{\bs},
% \label{eq:mean_regression}
\end{equation*}
and the covariance between $\bbeta_{\bs}^T\bx_{i,\bs}$ and $\bbeta_{\bs}^T\bx_{j,\bs}$ is given by
\begin{align*}
%\Sigma_{n,\bs}=
Cov\left(\bbeta_{\bs}^T\bx_{i,\bs},\bbeta_{\bs}^T\bx_{j,\bs}\right)
&=\sigma^2_{\bbeta}g_n \bx^T_{i,\bs}\left(X_{\bs}^TX_{\bs}\right)^{-1}\bx_{j,\bs}.%\notag\\
%&=\left\{\begin{array}{cc}\sigma^2_{\bbeta}\|\bx_{i,\bs}\|^2, & \mbox{if}~~i=j;\\ 0, & \mbox{otherwise.}\end{array}\right.
% \label{eq:cov_regression}
\end{align*}
Therefore, $\Sigma_{\bs}=\sigma^2_{\bbeta}g_n X^T_{\bs}\left(X_{\bs}^TX_{\bs}\right)^{-1}X_{\bs}=\sigma^2_{\bbeta}g_nP_{n,\bs}$, where $P_{n,\bs}$ is the projection matrix on the space of $X_{\bs}$.  %$\sigma^2_{\bbeta}$ and $g_n$ are so chosen that $ \sigma^2_{\bbeta}g_n / \sigma_\epsilon^2=O(1)$.

We verify assumptions (A1)--(A3) under this setup. To see that assumption (A1) holds, we first calculate the Kullback-Leibler divergence between the marginal density of $\yn$ under $\bs$ and that under $\bs_0$, $\mathcal{KL}^n(\bs,\bs_0)$, which is
\begin{align*}
\mathcal{KL}^n(\bs,\bs_0) &\propto tr\left(A_{n,\bs}\right)-\log|A_{n,\bs}| -n\\
& \hspace{.5 in} +\left(\bmu_{\bs}- \bmu_{0,\bs_0}^t \right)^T \left(\sigma_{\epsilon}^2 I_n + \sigma^2_{\bbeta}g_n P_{n,\bs}  \right)^{-1}\left(\bmu_{\bs}- \bmu_{0,\bs_0}^t \right)\\
&=tr\left(A_{n,\bs}\right)-\log|A_{n,\bs}| -n+\Delta_{n,\bs}. 
\end{align*}
% Next observe that
% \begin{eqnarray*}
%  \log|A_{n,\bs}| -n+tr\left(A_{n,\bs}\right)=\sum_{i=1}^{n}\left(\log \lambda_i+\lambda_i-1 \right).
% \end{eqnarray*}
As the eigenvalues of a projection matrix
can only be zero or one, and the traces of $P_{n,\bs_0}$ and $P_{n,\bs}$ are $|\bs_0|$ and $|\bs|$, respectively, we have
\begin{align}
tr\left(A_{n,\bs}\right)&=tr\left[\left(\sigma^2_{\epsilon}I_n+\sigma^2_{\bbeta}g_nP_{n,\bs}\right)^{-1}\left(\sigma^2_{\epsilon}I_n+\sigma^2_{\bbeta}g_nP_{n,\bs_0}\right)\right]\notag\\
&=tr\left[I_n+\sigma^2_{\bbeta}g_n\left(\sigma^2_{\epsilon}I_n+\sigma^2_{\bbeta}g_nP_{n,\bs}\right)^{-1}\left(P_{n,\bs_0}-P_{n,\bs}\right)\right]\notag\\
&=n+\sigma^2_{\bbeta}g_ntr\left(D_{n,\bs}\right),
\label{eq:tr_A}
\end{align}
where $D_{n,\bs}=\left(\sigma^2_{\epsilon}I_n+\sigma^2_{\bbeta}g_nP_{n,\bs}\right)^{-1}\left(P_{n,\bs_0}-P_{n,\bs}\right)$.
By Result \ref{eqn:symmetric2} we have

{\centering $\left(\sigma^2_{\epsilon}\right)^{-1}I_n\geq\left(\sigma^2_{\epsilon}I_n+\sigma^2_{\bbeta}g_nP_{n,\bs}\right)^{-1}\geq\left(\sigma^2_{\epsilon}+g_n\sigma^2_{\bbeta}\right)^{-1}I_n,$\par}

\noindent so that
\begin{equation}
\frac{|\bs_0|-|\bs|}{\sigma^2_{\epsilon}+g_n\sigma^2_{\bbeta}}\leq tr\left(D_{n,\bs}\right)\leq\frac{|\bs_0|-|\bs|}{\sigma^2_{\epsilon}}.
\notag
\end{equation}
Substituting the above in (\ref{eq:tr_A}) yields
\begin{equation*}
n+\sigma^2_{\bbeta}g_n\frac{|\bs_0|-|\bs|}{\sigma^2_{\epsilon}+g_n\sigma^2_{\bbeta}}\leq tr\left(A_{n,\bs}\right)\leq n+\sigma^2_{\bbeta}g_n\frac{|\bs_0|-|\bs|}{\sigma^2_{\epsilon}}.
%\label{eq:tr_A_bounds}
\end{equation*}
As $|\bs_{0}|-|\bs|\leq |\bs \Delta \bs_{0}|$, assuming that $|\bs \Delta \bs_{0}| \leq O(p^{\omega_{\bs}})=O(n^{r\omega_{\bs}})$, %where $0<\tilde\omega<\frac{1}{r}$, 
in conjunction with the assumption that $g_n=O(p^{\omega_{\bs}})$, as $n\rightarrow\infty$,
\begin{equation}
	\frac{tr\left(A_{n,\bs}\right)}{n^{1+2r\omega_{\bs}}}\rightarrow \left\{\begin{array}{ccc}1 &\mbox{if} & \omega_{\bs}=0;\\ 0 & \mbox{if} & \omega_{\bs}\in (0,1].\end{array}\right. \notag
%	\label{eq:trace_limit1}
\end{equation}
%Note that as $(I+\sigma_{\bbeta}^2 g_n P_{n,\bs_0}/\sigma^2_{\epsilon})$ and $(I+\sigma_{\bbeta}^2 g_n P_{n,\bs}/\sigma^2_{\epsilon})^{-1}$ are Hermitian positive semidefinite matrices
%\begin{eqnarray*}
%0\leq tr (A_{n,\bs}) &\leq& tr \left(I+\frac{\sigma_{\bbeta}^2 g_n}{\sigma^2_{\epsilon}} P_{n,\bs_0} \right) tr \left(I+\frac{\sigma_{\bbeta}^2 g_n}{\sigma^2_{\epsilon}} P_{n,\bs} \right)^{-1}\\
%&\leq& \frac{|\bs_0|\left(1+\sigma_{\bbeta}^2 g_n/\sigma^2_{\epsilon} \right)+ (n-|\bs_0|)}{|\bs|\left(1+\sigma_{\bbeta}^2 g_n/\sigma^2_{\epsilon} \right)+ (n-|\bs|)} \rightarrow 1,
%\end{eqnarray*}
%as $n\rightarrow\infty$, since $|\bs|<p$ is finite. The first inequality is due to \citet[Lemma 3]{trace_inequality2} and \cite{trace_inequality}.
Further
\begin{eqnarray}
|A_{n,\bs}|= \frac{\left| I +\sigma_{\bbeta}^2 g_n\sigma^{-2}_{\epsilon} P_{n,\bs_0} \right|}{ \left| I +\sigma_{\bbeta}^2 g_n\sigma^{-2}_{\epsilon} P_{n,\bs} \right| } = \left(1+\frac{\sigma_{\bbeta}^2 g_n}{\sigma^2_{\epsilon}} \right)^{|\bs_0|-|\bs|}.\notag
\end{eqnarray}
Therefore,
%\begin{equation}
	$\log |A_{n,\bs}|/n^{1+2r\omega_{\bs}} 
	= \left(|\bs_0|-|\bs| \right) \log \left(1+ \sigma_{\bbeta}^2 g_n/\sigma^2_{\epsilon} \right)/n^{1+2r\omega_{\bs}} \rightarrow 0,~\mbox{as}~n \rightarrow \infty.$ 
%	\label{eq:det_limit1}
%\end{equation}
Combining the above facts, we get, for all $\omega_{\bs}\in[0,1]$, %(\ref{eq:trace_limit1}) and (\ref{eq:det_limit1}) we obtain
\begin{eqnarray*}
	\frac{1}{n^{1+2r\omega_{\bs}}} \left\{ tr\left(A_{n,\bs}\right) -\log|A_{n,\bs}| -n \right\} \rightarrow 0,~\mbox{as}~n\rightarrow\infty. %-1.
\end{eqnarray*}
% {\color{red} What kind of convergence is this if Xs are stochastic? a.s.?}
Thus
\begin{eqnarray*}
	\liminf_n \frac{1}{n^{1+2r\omega_{\bs}}} \mathcal{KL}^n(\bs,\bs_0) = %-1 +
	\liminf_n \frac{\Delta_{n,\bs}}{n^{1+2r\omega_{\bs}}}.
\end{eqnarray*}
%: $\liminf_n \Delta_{n,\bs}/n^{1+2r\omega}>0$,
Thus assumption (A1) is implied by $\liminf_n n^{-(1+2r\omega_{\bs})} \mathcal{KL}^n(\bs,\bs_0)>0$, 
which is a natural assumption.
Bounded, positive eigenvalues of $\sigma^2_{\epsilon}I_n+\Sigma_{\bs}$ and the second part of Result \ref{eqn:symmetric2} (see Section \ref{sec:appendix} of the supplement) imply that $\Delta_{n,\bs}$ is of the same order
as $\|\bmu_{\bs}-\bmu^t_{0,\bs_0}\|^2$, which again, is $O(n^{1+2r\omega_{\bs}})$, as we show below. Viewing the requirement of ($A1$) from this perspective, it seems natural
to demand that the mean functions of the competing and the true models be distinct in the sense that $\liminf_n\|\bmu_{\bs}-\bmu^t_{0,\bs_0}\|^2/n^{1+2r\omega_{\bs}}>0$.
%As $\liminf_n KL\left( f, f_0 \right)/n > 0$, this in turn implies that $\liminf_n \Delta_{n,\bs}/n>0$, %\geq 1$,
%which imples (A1).

\vskip5pt
To check assumption ($A2$) note that for positive definite Hermitian matrices $A$ and $B$, $\lambda_{\max}(AB) \leq \lambda_{\max}(A) \lambda_{\max}(B)$. Using this fact and as $\sigma_{\bbeta}^2 g_n\sigma^{-2}_{\epsilon}=O(p^{\omega_{\bs}})$, it is easily seen
that $$\lambda_{\max}(A_{n,\bs})\leq \left(1+\sigma_{\bbeta}^2 g_n\sigma^{-2}_{\epsilon} \right) =O(p^{\omega_{\bs}}).$$

\vskip5pt
\noindent Finally we check ($A3$). Note that
%\begin{eqnarray*}
%	\left\|\bmu_{0,\bs}-\bmu^t_{0,\bs_0} \right\|^2 =\left\|X_{\bs}\bbeta_{0,\bs}- X_{\bs_0}\bbeta_{0,\bs_0} \right\|^2 
%	\leq 2 \left( \left\| X_{\bs}\bbeta_{0,\bs} \right\|^2+ \left\| X_{\bs_0}\bbeta_{0,\bs_0} \right\|^2 \right).
%%	= 2\sum_{i=1}^n \left\{\left( x_{i,\bs}^\prime \bbeta_{\bs}\right)^2+\left( x_{i,\bs_0}^\prime \bbeta_{\bs_0}\right)^2\right\}.
%\end{eqnarray*}
%Now, by our assumptions, for any $\bs$, $\|\bbeta_{0,\bs}\|_{L_1}=O(p^\omega n^{c_{0}})$. 
\begin{eqnarray*}
	\left\|\bmu_{\bs}-\bmu^t_{0,\bs_0} \right\|^2 \leq \left\|X_{\bs\Delta \bs_{0}}\bbeta_{0,\bs\Delta\bs_{0}} \right\|^2,
	%	= 2\sum_{i=1}^n \left\{\left( x_{i,\bs}^\prime \bbeta_{\bs}\right)^2+\left( x_{i,\bs_0}^\prime \bbeta_{\bs_0}\right)^2\right\}.
\end{eqnarray*}
as the prior mean of the $j$-th covariate $\beta_{0,j}$ remains same accross different models which include the $j$-th covariate.

Further, recall that for any $\bs$, $\|\bbeta_{0,\bs}\|_{L_1}=O(|\bs|)$. 
Since the covariates lie on a compact space, it follows that
%$\left\| X_{\bs}\bbeta_{0,\bs} \right\|^2=\sum_{i=1}^n\left(\bx^T_{i,\bs}\bbeta_{0,\bs}\right)^2=O\left(n^{1+2c_{0}}p^{2\omega}\right)=O(n^{1+2r\omega+2c_{0}})$.
$\left\| X_{\bs\Delta \bs_{0}}\bbeta_{0,\bs\Delta \bs_{0}} \right\|^2 = \sum_{i=1}^n\left(\bx^T_{i,\bs\Delta \bs_{0}}\bbeta_{0,\bs\Delta \bs_{0}}\right)^2 = O\left(np^{2\omega_{\bs}}\right)=O(n^{1+2r\omega_{\bs}})$, if $|\bs\Delta \bs_{0}|=O\left(p^{\omega_{\bs}}\right)$.
Thus ($A3$) holds.

Thus Theorem \ref{theorem:as_conv} holds for the linear regression setup.
This result is summarized in the form of the following theorem.

\begin{theorem}
	\label{theorem:linreg}
Consider the linear regression model $y_i=\bbeta_{\bs}^T \bx_{i,\bs}+\epsilon_i$, where
$\epsilon_i\stackrel{iid}{\sim}N\left(0,\sigma^2_{\epsilon}\right)$, for $i=1,\ldots,n$. Let
	$\bbeta_{\bs}\sim N\left(\bbeta_{0,\bs},g_n \sigma^2_{\bbeta}\left( X_{\bs}^T X_{\bs}\right)^{-1} \right)$, where $1\leq |\bs|\leq p$ and $p=O\left(n^r\right)$, $r>0$.
Assume that the space of covariates is compact, and 
%Here $\bS=\{1,2,\ldots,p\}$, and $\bs_0~(\subseteq\bS)$ is the set of indices of the true set of covariates.
%$\|\bbeta_{0,\bs}\|_{L_1}=\sum_{j=1}^{|\bs|}|\beta_{0,j}|=O(p^{\omega}n^{c_{0}})$, for $0\leq \omega\leq 1$, $c_{0}<1/4$ and $p=n^r$, where $r>0$.
	$\|\bbeta_{0,\bs}\|_{L_1}=\sum_{j=1}^{|\bs|}|\beta_{0,j}|=O(|\bs|)$. Further, if there exists some $\omega_{\bs}\in [0,1]$ such that 
	$|\bs\Delta \bs_{0}|=O\left(p^{\omega_{\bs}}\right)$, and  $ \mathcal{KL}^n(\bs,\bs_0)/(n^{1+2r\omega_{\bs}})>0$, then for $g_n=O(p^{\omega_{\bs}})$ 
%Also assume that $|\bs\Delta \bs_{0}|=O(p^{\omega})=O(n^{r\omega})$.
the statement of Theorem \ref{theorem:as_conv} holds.
\end{theorem}

\subsection{Gaussian process with squared exponential kernel}
\label{sec:gp_illustration}

We now consider the problem of variable selection in nonparametric model of the form $y=\bx^T_{\bs}\bbeta_{\bs}+
f(\bx_\bs)+\epsilon$, where $f$ belongs to a Hilbert space $\mathcal H$. 
Let $f(\bx_\bs)$ be modeled by a zero-mean Gaussian process with %mean function $\mu\left(\bx_{\bs}\right)=\bx^T_{\bs}\bbeta_{\bs}$, 
squared exponential covariance kernel of the form
\begin{eqnarray}
 Cov\left(f(\bx_\bs), f(\bx_\bs^{\prime}) \right) = \sigma_f^2 \exp \left\{-\frac{1}{2} \left(\bx_\bs-\bx^{\prime}_\bs\right)^T D_\bs \left(\bx_\bs-\bx^{\prime}_\bs\right)  \right\}.
\label{var:rkhs}
 \end{eqnarray}
Here $\sigma_f^2$ can be interpreted as the process variance,
and the diagonal elements of $D_{\bs}$
%%the last part $\exp \left\{- \left(\bx_\bs-\bx^{\prime}_\bs\right)^T D_\bs \left(\bx_\bs-\bx^{\prime}_\bs\right) /2 \right\}$
can be interpreted as the smoothness parameters.
As in the case of linear regression, we consider the Zellner's $g$-prior for $\bbeta_{\bs}$.
Thus, the mean function $\mu_{\bs}$ here is of the same form as in the linear regression case.

We denote the %$n$-dimensional mean vector of model $\bs$ by $\bmu_{\bs}$ and 
covariance matrix by $\Sigma_{\bs}$, as before. Note that the $(i,j)$-th element of $\bSigma_{\bs}$ here is 
$$\sigma^2_{\bbeta}g_n \bx^T_{i,\bs}\left(X_{\bs}^TX_{\bs}\right)^{-1}\bx_{j,\bs}
+\sigma_f^2 \exp \left\{-\frac{1}{2} \left(\bx_{i,\bs}-\bx_{j,\bs}\right)^T D_\bs \left(\bx_{i,\bs}-\bx_{j,\bs}\right)  \right\}.$$

\paragraph{True model} As before we indicate a particular subset of $\bS=\{ 1, \ldots, p\}$ as the true set of regressors $\bs_0$. The corresponding mean vector and variance matrices are denoted by $\bmu_{\bs_0}^{t}$ and $\Sigma_{\bs_0}^{t}$, respectively.

\paragraph{Assumption} Before verifying assumption (A1)-(A3), we state the following assumption on the design matrix. 
\begin{enumerate}[(A4)]
 \item We assume that $\left\{\bx_{j,\bs}:j=1,2,\ldots\right\}$ are such that for all $i\geq 1$,
	 $$\sum_{j\neq i=1}^n\exp\left\{-\frac{1}{2}\left(\bx_{i,\bs}-\bx_{j,\bs}\right)^T D_\bs \left(\bx_{i,\bs}-\bx_{j,\bs}\right)/2\right\}=K_{\bs}=O(1), $$ 
\end{enumerate}
where $K_{\bs}~(>0)$ may depend upon $\bs$.

%\noindent Note that $\left(\bx_{i,\bs}-\bx_{j,\bs}\right)$ are $|\bs|$-vectors, and if each of the quadratic form 
%$\left(\bx_{i,\bs}-\bx_{j,\bs}\right)^T D_\bs \left(\bx_{i,\bs}-\bx_{j,\bs}\right)$ is lower bounded by $c\log \left(n\right)$, 
%for some positive constant $c$, then (A4) holds. 

\subsection*{Verification of the assumptions}
We verify assumptions (A1)--(A3) under this setup and assuming (A4) holds.
%As already mentioned, we consider two different setups: in the first situation we consider non-compact covariate space with some reasonable regularity assumptions,
%and in the second case, $\bx_i$ supported on compact space $D$ where $D$ is supported by product of $|\bs|$ $iid$ $N(0,\sigma^2)$ for some choice of $\sigma^2$.
%{\color{red}(check this part. is it correct?)}

%\paragraph{Case I}
\vskip5pt
First note that assumption (A3) is satisfied in the same way as in the linear regression case.

Before verifying (A1), note that by Gerschgorin's circle theorem,
every eigenvalue $\lambda$ of any $n\times n$ matrix $A$ with $(i,j)$-th element $a_{ij}$ satisfies $|\lambda-a_{ii}|\leq\sum_{j\neq i}|a_{ij}|$, for at least one
$i\in\{1,\ldots,n\}$ (see, for example, \ctn{Lange10}).
In our case it then follows by (A4) that the maximum eigenvalue of the covariance matrix associated with 
$f(\cdot)$ is bounded above by $K_{\bs}$.
Also, the covariance matrix associated with the linear part $\bx^T_{\bs}\bbeta_{\bs}$ is essentially the projection matrix, with maximum eigenvalue $1$.
Hence, using the second part of Result \ref{eqn:symmetric2}, we conclude that the maximum eigenvalue of $\bSigma_{\bs}$ is bounded above by finite $\tilde K_{\bs}>0$.

To verify ($A1$),
% we again assume that $\sum_{j\neq i}\exp\left\{-\frac{b}{2}(x_i-x_j)^2\right\}<\tilde M<\infty$, for all $i\geq 1$. Then
note that $\left(\sigma^2_{\epsilon}I_n+\Sigma_{\bs}\right)^{-1} \succ \left(\sigma^2_{\epsilon}+ \tilde K_{\bs}\right)^{-1}I_n$ by 
the first part of Result \ref{eqn:symmetric2}.
Hence,
\begin{equation}
	n^{-1-2r\omega_{\bs}}\Delta_{n,\bs}>\left(\sigma^2_{\epsilon}+\tilde K_\bs\right)^{-1}n^{-1-2r\omega_{\bs}}\|\bmu_{n,\bs}-\bmu_{n,\bs_0}^{t}\|^2.
\label{eq:rkhs_A1}
\end{equation}
Now, if we wish to enforce distinguishability of only the mean functions of the competing models in the sense that
\begin{equation}
	\underset{n}{\lim\inf}~n^{-1-2r\omega_{\bs}}\|\bmu_{\bs}-\bmu_{\bs_0}^{t}\|^2>0,
\label{eq:rkhs_A1_2}	
\end{equation}
then it is clear from (\ref{eq:rkhs_A1}) that ($A1$) holds.

\noindent Next we check (A2).
% (A2) trivially holds if $\Sigma_{n,\bs}-\Sigma_{n,\bs_0}$ is non-negative definite. As in that case $A_{n,\bs}\leq I$, and (A2) is satisfied.
% For all other cases, we additionally assume that
% \begin{align}\tag{A4}
% -c_\bs {\bf 1} \leq A_{n,s}{\bf 1}=\left(\sigma^2_{\epsilon} I+ \Sigma_{n,\bs} \right)^{-1}\left(\sigma^2_{\epsilon} I+ \Sigma_{n,\bs_0} \right){\bf 1}\leq c_\bs {\bf 1} \label{A4}
% \end{align}
% for some suitable constant $c_\bs\in \mathbb{R}$. Under this assumption by Gerschgorin’s circle theorem it is easy to check that $\lambda_1(A_{n,\bs}) \leq c_\bs=O(1)$.
% Specifically, let us assume that $\left\{\bx_{j,\bs}:j=1,2,\ldots\right\}$ are such that for all $i\geq 1$,
% $\sum_{j\neq i}\exp\left\{-\frac{1}{2}\left(\bx_{i,\bs}-\bx^{\prime}_{j,\bs}\right)^T D_\bs \left(\bx_{i,\bs}-\bx^{\prime}_{j,\bs}\right)/2\right\}<K_{\bs}<\infty$ for some constant $K_{\bs}$ depending on $\bs$. %An example of such a sequence is $\bx_{i,\bs}=(i,\ldots,i)^T$ for $i\geq 1$, and $D_\bs=I_{|\bs|}$.
% By Gerschgorin's circle theorem,
% every eigenvalue $\lambda$ of any $n\times n$ matrix $A$ with $(i,j)$-th element $a_{ij}$ satisfies $|\lambda-a_{ii}|\leq\sum_{j\neq i}|a_{ij}|$, for at least one
% $i\in\{1,\ldots,n\}$ (see, for example, \ctn{Lange10}).
By (A4) the maximum eigenvalue of $\Sigma_{\bs_0}$, $\lambda_{\max}(\Sigma_{\bs_0}^{t})\leq \tilde K_{\bs_0}$.
Then
\begin{eqnarray*}
\lambda_1(A_{n,\bs})&=&\lambda_{\max}\left[\left( \sigma_\epsilon^2 I_n + \Sigma_{\bs}\right)^{-1} \left(\sigma_\epsilon^2 I_n + \Sigma_{\bs_0}^{t}\right)\right]\\
&\leq&\frac{\lambda_{\max}\left(\sigma_\epsilon^2 I_n + \Sigma_{\bs_0}^{t}\right)}{\lambda_{\min}\left( \sigma_\epsilon^2 I_n + \Sigma_{\bs}\right)}
\leq\frac{\sigma^2_{\epsilon}+ \tilde K_{\bs_0}}{\sigma^2_{\epsilon}}=O(1),
\end{eqnarray*}
showing that ($A2$) holds.

Therefore, we have established the following theorem:
%\begin{theorem}
% \label{thm:RKHS1}
% Consider the nonparametric regression model given by $y=f(\bx_\bs)+\epsilon$, where $f$ %$f \in \mathcal{H}$
%is assumed to follow a Gaussian process with bounded mean function $\mu$ %as in (\ref{mean:rkhs})
%and squared exponential covariance kernel of the form (\ref{var:rkhs}). 
%If there exists a true set of regressors following the same form of the model, then the statement of Theorem \ref{theorem:as_conv} holds true under the assumptions $\underset{n}{\lim\inf}~n^{-1}\|\bmu_{\bs}-\bmu_{\bs_0}\|^2>0,$ and ($A4$). %that $\left(\sigma^2_{\epsilon} I+ \Sigma_{\bs} \right)^{-1}\left(\Sigma_{\bs_0}-\Sigma_{\bs} \right){\bf 1}\leq c_\bs {\bf 1}$ for some $c_\bs \in \mathbb{R}$.
%\end{theorem}

\begin{theorem}
	\label{theorem:gpreg}
	Consider the regression model $y_i=\bbeta_{\bs}^T \bx_{i,\bs}+f(\bx_{i,\bs})+\epsilon_i$, where
	$\epsilon_i\stackrel{iid}{\sim}N\left(0,\sigma^2_{\epsilon}\right)$, for $i=1,\ldots,n$. Let
	$\bbeta_{\bs}\sim N\left(\bbeta_{0,\bs},g_n \sigma^2_{\bbeta}\left( X_{\bs}^T X_{\bs}\right)^{-1} \right)$, where $1\leq |\bs|\leq p$ and $p=O\left(n^r\right)$, $r>0$.
	Let $f(\cdot)$ be a zero-mean Gaussian process with a squared exponential covariance kernel of the form (\ref{var:rkhs}). Assume that the space of covariates is compact,
	%Here $\bS=\{1,2,\ldots,p\}$, and $\bs_0~(\subseteq\bS)$ is the set of indices of the true set of covariates.
and	$\|\bbeta_{0,\bs}\|_{L_1}=O(|\bs|)$.
%	If there exists a true set of regressors following the same form of the model, then 
	If there exists some $\omega_{\bs}\in[0,1]$ such that $|\bs\Delta \bs_{0}|=O\left(p^{\omega_{\bs}}\right)$, and $\underset{n}{\lim\inf}~n^{-1-2r\omega_{\bs}}\|\bmu_{\bs}-\bmu_{\bs_0}\|^2>0$, and further if ($A4$) holds, then for $g_n=O(p^{\omega_{\bs}})$
	the statement of Theorem \ref{theorem:as_conv} 
	holds. %that $\left(\sigma^2_{\epsilon} I+ \Sigma_{\bs} \right)^{-1}\left(\Sigma_{\bs_0}-\Sigma_{\bs} \right){\bf 1}\leq c_\bs {\bf 1}$ for some $c_\bs \in \mathbb{R}$.
	%Then (\ref{eq:bf_conv}) holds.
\end{theorem}

Additionally, consider the following remarks.

\begin{remark}
The condition in (\ref{eq:rkhs_A1})
also implies that 
	${\lim\inf}_n~n^{-1-2r\omega_{\bs}}\mathcal{KL}^n(\bs,\bs_0)$

	\noindent $={\lim\inf}_n~n^{-1-2r\omega_{\bs}} \left\{tr(A_{n,\bs})-\log|A_{n,\bs}|-n+\Delta_{n,\bs} \right\}>0$,
since $tr(A_{n,\bs})-\log|A_{n,\bs}|-n\geq 0$.
	Recall that even in the linear regression setup we had replaced the KL-divergence ${\lim\inf}_n~n^{-1-2r\omega_{\bs}}\mathcal{KL}^n(\bs,\bs_0)>0$
with the above mean divergence condition (\ref{eq:rkhs_A1}) to verify ($A1$), since the eigenvalues of $\Sigma_{\bs}$ in that setup
are also bounded.
\end{remark}

\begin{remark}The linear regression term in the mean function can be replaced by any function $\bmu_{\bs}$ subject to the condition 
	%$\left\| \bmu_{0,\bs} \right\|_{L_1}=O(n^{c_{0}}p^{\omega})=O(n^{c_{0}+r\omega})$, where $c_{0}<1/4$ and $0\leq \omega \leq 1$. 
	$\left\| \bmu_{\bs} \right\|_{L_1}=O(np^{2\omega_{\bs}})=O(n^{1+2r\omega_{\bs}})$, where $0\leq \omega_{\bs} \leq 1$. 
	It is easy to verify assumptions (A1) and (A3) under the aforementioned restriction on $\bmu_{\bs}$.
\end{remark}

\section{The case with unknown error variance}
\label{sec:unknown_error_variance}
So far we have assumed that the error variance $\sigma^2_{\epsilon}$ is known. In reality, this may also be unknown and we need to assign a prior on the same. 
For our purpose, for any $\bx_{i},\bx_{i}\in\mathfrak X$, we now set $Cov(f(\bx_{i}),f(\bx_{j}))=\sigma^2_{\epsilon}c(\bx_{i},\bx_{j})$, where $c({\bf x},{\bf y})$ is some appropriate
correlation function, $i,j=1,\ldots,n$. Thus, we set the process variance of $f(\cdot)$ to be the same as the error variance. Although this might seem somewhat restrictive
from the inference perspective, for Bayes factor based variable selection this is quite appropriate, as we establish almost sure exponential convergence
of the resultant Bayes factor associated with this prior, in favour of the true set of covariates.

With the aforementioned modification, we assign the conjugate inverse-gamma prior on $\sigma^2_{\epsilon}$ with parameters $\alpha,~\beta$ as follows:
\begin{equation}
 \pi \left(\sigma^2_\epsilon\right)=\frac{\beta^\alpha}{\Gamma(\alpha)}\sigma_\epsilon^{-2(\alpha+1)}\exp\left(-\frac{\beta}{\sigma^2_\epsilon }\right), ~~~\alpha>2,~\beta>0. 
\label{eq:invgamma}
\end{equation}
% Additionally, we modify the assumption (A2) in this section as follows:
% \begin{itemize}
% \item[$\left(A2^{\prime}\right)$] $a_j\sim N\left(m_j,\sigma^2_\epsilon \varepsilon^2_j\right)$, where $m_j$'s and $\varepsilon_j$'s satisfy
% \begin{align}
% \sum_{j=1}^{\infty}\left|m_j\right|&<\infty;
% \label{eq:parseval1_modified}\\
% \sum_{j=1}^{\infty} \varepsilon_j&<\infty.
% \label{eq:parseval2_modified}
% \end{align}
% \end{itemize}
% Note that, here also one can show that $\sum_{j=1}^{\infty} |a_j| <\infty$ almost surely. This follows from the fact that given any fixed $\sigma^2_\epsilon$,
% $\sum_{j=1}^{\infty} |a_j| <\infty$ almost surely, and the fact that the inverse-gamma prior is proper.
Under the same prior setup on $f$, the marginal of $\by_n=(y_1,\ldots,y_n)^T$ given $\sigma^2_\epsilon$ is the $n$-variate normal, given by
\begin{equation}
\by_n\sim N_n\left(\bmu_{\bs},\sigma^2_{\epsilon}\left( I_n+\bSigma_{\bs}\right)\right),
\label{eq:marginal_y_modified}
\end{equation}
where $\bSigma_{\bs}$ is as given in (\ref{eq:cov_matrix}).
After marginalizing $\sigma^2_\epsilon$ the marginal of $\by_n$ is
\begin{eqnarray}
 m_{\bs}\left( \by_n \right)\propto \left| I+\Sigma_{\bs}\right|^{-1/2} \left\{ \left(\by_n-\bmu_{\bs}\right)^T\left( I_n+\bSigma_{\bs}\right)^{-1}
\left(\by_n-\bmu_{\bs}\right) +2\beta \right\}^{-(\alpha+n/2)+1},    \notag
\end{eqnarray}
which is proportional to the density of multivariate $t$ distribution with location parameter $\bmu_{\bs}$, covariance matrix $\beta\left( I_n+\bSigma_{\bs}\right)/(\alpha-1)$, and degrees of freedom $2(\alpha-1)$. Thus, $E(\by_n)=\bmu_{\bs}$, and $Var(\by_n)=\beta\left( I_n+\bSigma_{\bs}\right)/(\alpha-2)$, under $\mathcal{M}_{\bs}$.

Here the Bayes factor of any model $\bs$ to the true model $\bs_0$ is
\begin{align}
 BF^{n}_{\bs,\bs_0}=&\frac{\left| I_n+\Sigma_{\bs_0}^{t}\right|^{1/2}}{\left| I_n+\Sigma_{\bs}\right|^{1/2}} \notag \\
 &\times \left[ \frac{\left(\by_n-\bmu_{\bs}\right)^T\left( I_n+\bSigma_{\bs}\right)^{-1}
\left(\by_n-\bmu_{\bs}\right) +2\beta }{\left(\by_n-\bmu_{\bs_0}^{t}\right)^T\left( I_n+\bSigma_{\bs_0}^{t}\right)^{-1}
\left(\by_n-\bmu_{\bs_0}^{t}\right) +2\beta } \right]^{-(\alpha+n/2)+1}.
% &=\frac{\left| I+\Sigma_{n,\bs_0}\right|^{1/2}}{\left| I+\Sigma_{n,\bs}\right|^{1/2}}\left[1-\frac{\Xi_{\bs,\bs_0}}{\bz^T_n\bz_n+2\beta}\right]^{-(\alpha+n/2)+1},
\notag
\end{align}
%
%\begin{eqnarray*}
%1-\frac{\left(\by_n-\bmu_{n,\bs_0}\right)^T\left( I_n+\bSigma_{n,\bs_0}\right)^{-1}
%\left(\by_n-\bmu_{n,\bs_0}\right)-\left(\by_n-\bmu_{n,\bs}\right)^T\left( I_n+\bSigma_{n,\bs}\right)^{-1} \hspace{-.05 in}
%\left(\by_n-\bmu_{n,\bs}\right) }{\left(\by_n-\bmu_{n,\bs_0}\right)^T\left( I_n+\bSigma_{n,\bs_0}\right)^{-1} \hspace{-.05 in}
%\left(\by_n-\bmu_{n,\bs_0}\right) +2\beta }.
%\end{eqnarray*}
As before, define 
$\bz_n\sim N({\bf 0}, I_n)$ such that 

$\bz_n^T\bz_n =\left(\by_n-\bmu_{\bs_0}^{t}\right)^T\left( I_n+\bSigma_{\bs_0}^{t}\right)^{-1}
\left(\by_n-\bmu_{\bs_0}^{t}\right)$, 

 $\Delta_{n,\bs}=(\bmu_{\bs}-\bmu^t_{\bs_0})^T \left( I_n+\bSigma_{\bs}\right)^{-1}
(\bmu_{\bs}-\bmu^t_{\bs_0})$ and

$A_{n,\bs}= \left( I_n+\bSigma_{\bs_0}^{t}\right) \left( I_n+\bSigma_{\bs}\right)^{-1} $,

$C_{n,\bs}= \left( I_n+\bSigma_{n,\bs_0}^{t}\right)^{1/2} \left( I_n+\bSigma_{n,\bs}\right)^{-1} \left( I_n+\bSigma_{n,\bs_0}^{t}\right)^{1/2}$. 
Therefore, 

{\centering
$\left(\by_n-\bmu_{\bs}\right)^T \left( I_n+\bSigma_{\bs}\right)^{-1}
\left(\by_n-\bmu_{\bs}\right) 
= \bz_n^T C_{n,\bs} \bz_n + \Delta_{n,\bs} \hspace{1 in}$

$\hspace{2.2 in}-2\left(\bmu_{\bs}-\bmu_{\bs_0}^{t}\right)^T\left(I_n+\Sigma_{\bs}\right)^{-1}\left(\by_n-\bmu_{\bs_0}^{t}\right).$
\par}

% \\
% -\left(\by_n-\bmu_{n,\bs}\right)^T\hspace{-.05 in} \left( I_n+\bSigma_{n,\bs}\right)^{-1}
% \left(\by_n-\bmu_{n,\bs}\right)\notag\\
% &=\bz^T_n\bz_n-\bz^T_nA_n\bz_n+2\left(\bmu_{n,\bs}-\bmu_{n,\bs_0}\right)^T\left(I_n+\Sigma_{n,\bs}\right)^{-1}\left(\by_n-\bmu_{n,\bs_0}\right),\notag
It follows that
\begin{eqnarray}
	&&\frac{1}{n\log n}\log BF^{n}_{\bs,\bs_0} \notag
	\\&&~= \frac{\log\left|C_{n,\bs}\right|}{2n\log n}- \frac{1}{\log n}\left(\frac{1-\alpha}{n}-\frac{1}{2}\right) 
	\left[ \log\left(\frac{\bz_n^T \bz_n+2\beta}{n^{1+2r\omega_{\bs}}}  \right) \right. \notag\\
	&& \left. - \log\left\{ \frac{ \Delta_{n,\bs}+2\beta+\bz_n^T C_{n,\bs} \bz_n}{n^{1+2r\omega_{\bs}}}   - 2\frac{\left(\bmu_{\bs}-\bmu_{\bs_0}^{t}\right)^T\left(I_n+\Sigma_{\bs}\right)^{-1}\left(\by_n-\bmu_{\bs_0}^{t}\right)}{n^{1+2r\omega_{\bs}}}  \right\} \right] \notag \\
	&& ~\leq  \frac{1}{2\log n}\log \left\{\frac{tr (C_{n,\bs})}{n} \right\}  - \frac{1}{\log n}\left(\frac{1-\alpha}{n}-\frac{1}{2}\right) 
	\left[ \log\left(\frac{ \bz_n^T \bz_n+2\beta}{n} \right)   \right. \notag\\
	&&\qquad \qquad  \left.-2r\omega_{\bs}\log n -\log\left\{ \frac{ \Delta_{n,\bs}+tr(C_{n,\bs})+2\beta}{n^{1+2r\omega_{\bs}}}  + \frac{ \bz_n^T C_{n,\bs} \bz_n  - tr(C_{n,\bs}) }{n^{1+2r\omega_{\bs}}} \right. \right. \notag\\
	&& \hspace{.75 in} \left. \left. 	+\frac{2}{n^{1+2r\omega_{\bs}}}\left(\bmu_{\bs}-\bmu_{\bs_0}^{t}\right)^T\left(I_n+\Sigma_{\bs}\right)^{-1}\left(\by_n-\bmu_{\bs_0}^{t}\right)  \right\} \right] , \label{eq:lnBF2}
%&& \left. \left. +\frac{2}{n}\left(\bmu_{\bs}-\bmu_{\bs_0}\right)^T\left(I_n+\Sigma_{\bs}\right)^{-1}\left(\by_n-\bmu_{\bs_0}\right)  \right\} \right] , \label{eq:lnBF2}
\end{eqnarray}
where the last inequality is due to the log-sum inequality. 
We modify assumption (A1)--(A3) by replacing $\sigma_\epsilon^2$ by $1$, and term them $(A1^\prime)$--$(A3^{\prime})$.

Next observe the following facts:
\begin{enumerate}[(i)]
 %\item $\log |A_{n,\bs}|/2n  \leq \log \left(tr (A_{n,\bs})/n \right)/2$, by log-sum inequality.
%  which is given by
%  $$ \sum_{i=1}^n a_i \log \left( \frac{a_i}{b_i}\right) \geq \left(\sum_{i=1}^{n} a_i \right) \log \left(\frac{\sum_{i} a_i}{\sum_{i} b_i} \right),$$
%  considering $a_i=1$ for all $i$, and $b_i=\lambda_i\left(A_{n,\bs}\right)$.
	\item $E\left[ \left( \bz_n^T C_{n,\bs} \bz_n  - tr(C_{n,\bs})\right)\right]^4= O(n^{2+8r\omega_{\bs}})$ implying that 

		$\left[\bz_n^T C_{n,\bs} \bz_n  - tr(C_{n,\bs}) \right]/n^{1+2r\omega_{\bs}} \xrightarrow{a.s.}0$. 
One can prove this in exactly similar way as done in Result \ref{theorem:power4}, using assumptions $(A1^{\prime})$--$(A3^{\prime})$.
 \item Similarly, it can be shown that $E\left[ \left( \bz_n^T \bz_n  - n\right)\right]^4= O(n^2)$ implying 
 
 $\bz_n^T \bz_n/n \xrightarrow{a.s.} 1$.
 %The proof of this follows in exactly similar way as in Theorem \ref{theorem:power4}, using $(A3^{\prime})$.
\item From the above fact, it follows that $\log \left( \bz_n^\prime \bz_n /n \right) \xrightarrow{a.s.} 0$ by continuous mapping theorem.
\item Applying $(A3^{\prime})$, it can be shown that 
	
	{\centering 
		$E\left[\left(\bmu_{\bs}-\bmu_{\bs_0}^{t}\right)^T\left(I_n+\Sigma_{\bs}\right)^{-1}\left(\by_n-\bmu_{\bs_0}^{t}\right)\right]^4= O(n^{2+8r\omega_{\bs}+2b})$,\par}

 which in turn implies
	 
		{\centering 	$\left(\bmu_{\bs}-\bmu_{\bs_0}^{t}\right)^T\left(I_n+\Sigma_{\bs}\right)^{-1}\left(\by_n-\bmu_{\bs_0}^{t}\right)/n^{1+2r\omega_{\bs}} \xrightarrow{a.s.} 0$.\par}

 %The proof of this follows in exactly similar way as in Theorem \ref{theorem:power4}, using $(A3^{\prime})$.
	\item Finally, $tr(C_{n,\bs})=tr(A_{n,\bs})\leq n\lambda_{\max}(A_{n,\bs})$, and $\lambda_{\max}(A_{n,\bs})=O(p^{2\omega_{\bs}})$.
%by assumption $(A2^{\prime})$,
%we can conclude using $(A1^\prime)$ that $ \limsup_n \frac{1}{n}\log BF^{n}_{\bs,\bs_0} \leq -\delta_\bs$.

 %\item Limit infrimum of the curly braced part of (\ref{eq:lnBF2}) is no less than ~$\xi_\bs>0$~ almost surely by the above facts and assumption $(A1^{\prime})$. Further, limit supremum of the curly braced part of (\ref{eq:lnBF2}) is less than $c_\bs$ almost surely for some constant $c_\bs>0$ by the above facts and assumptions $(A2^{\prime})$--$(A3^{\prime})$. Therefore the third bracketed part in (\ref{eq:lnBF2}) is finite almost surely, by continuous mapping theorem.
 %\item The last fact implies that $(1-\alpha)/n$ multiplied with the third braced part in (\ref{eq:lnBF2}) converges to zero almost surely.

 \end{enumerate}
Using the above facts, it is easy to see that the right hand side of (\ref{eq:lnBF2}) has $\lim\sup$ $-2r\omega_{\bs}$, which is negative for $\omega_{\bs}\in (0,1]$.

When $\omega_{\bs}=0$, similar steps as above would lead to the result 

$
\limsup_{n} \frac{1}{n}\log BF^{n}_{\bs,\bs_0} \stackrel{a.s.}{=}-\delta_{\bs}. 
$
%
%and as $(1-\alpha)/n\rightarrow 0$ and $\beta/n\rightarrow 0$, the RHS of (\ref{eq:lnBF2}) is no bigger than
% \begin{eqnarray*}
%	 \frac{1}{2n^{2r\omega}}\log \left\{\frac{tr (C_{n,\bs})}{n} \right\}  
%	 + \frac{1}{2} \left[-\delta_1 - \frac{1}{n^{2r\omega}}\log\left\{ \frac{ \Delta_{n,\bs}}{n^{1+2r\omega}} +\frac{1}{n^{2r\omega}}\frac{tr(C_{n,\bs})}{n} 
%	 +\delta_2   \right\} \right], 
%%=\frac{1}{2} \left[ \delta_1 - \log\left\{ 1+\frac{ \Delta_{n,\bs}-n\delta_2}{tr(C_{n,\bs})}     \right\} \right] ,
%    \end{eqnarray*}
%for sufficiently large $n$ and any $\delta_1>0$ and $\delta_2>0$.
% 
% 
%Thus by assumption ($A1^\prime$), and choosing suitable small values of $\delta_1$ and $\delta_2$ 
%\begin{eqnarray}
%	\limsup_n \frac{1}{n}\log BF^{n}_{\bs,\bs_0} \leq -\delta_\bs.\notag %+ \frac{n\kappa^1_{n,\bs}}{A_{n,\bs}} \right]+ \kappa^2_{n,\bs}, \notag
%\end{eqnarray}
%%where $\kappa^1_{n,\bs}, \kappa^2_{n,\bs} \xrightarrow{a.s.} 0$.
%
%
%We present this result in the form of the following theorem:
Consequently, the following result holds:
\begin{theorem}
\label{theorem:as_conv2}
Consider the setup of Theorem \ref{theorem:as_conv} except that the error variance $\sigma^2_\epsilon$ is now unknown. 
Let an inverse gamma prior with parameters $\alpha$ and $\beta$ be applied to $\sigma_\epsilon^2$. Assume that 
	($A1^\prime$)--($A3^\prime$) hold for some $\omega_{\bs}\in (0,1]$, and
some positive constant $\delta_{\bs}$ depending upon $\bs$ ($\neq\bs_0$). Then
\begin{equation}
	%\max_{\bs\neq\bs_0} 
	\limsup_n \frac{1}{n\log n}\log\left(BF^{n}_{\bs,\bs_0}\right)\stackrel{a.s.}{=}-\delta_{\bs}.
\notag %label{eq:conv_as2}
\end{equation}
	For $\omega_{\bs}=0$, the following holds:
\begin{equation}
	%\max_{\bs\neq\bs_0} 
	\limsup_n \frac{1}{n}\log\left(BF^{n}_{\bs,\bs_0}\right)\stackrel{a.s.}{=}-\delta_{\bs}.
\notag %label{eq:conv_as2}
\end{equation}
% $\omega_{\bs}=0$ for all $\bs\subseteq\bS$, and
	Moreover, if the number of models is finite then there exists $\delta>0$ such that
\begin{equation}
	\max_{\bs\neq\bs_0} \limsup_n \frac{1}{n}\log\left(BF^{n}_{\bs,\bs_0}\right)\stackrel{a.s.}{=}-\delta.
\notag %label{eq:conv_as2}
\end{equation}
\end{theorem}
%\begin{remark}
%\label{remark:remark4}
%Remark \ref{remark:remark3} on nested models remains valid even in this situation.
%\end{remark}

\section{Convergence of integrated Bayes factor}
%\section{Convergence of Bayes factor when integrated with respect to priors on other parameters and hyperparameters}
\label{sec:bf_int}
Let us suppose, as is usual, that the Bayes factor $BF^n_{\bs,\bs_0}$ depends on a set of parameters and hyperparameters, denoted by $\bta$.
We denote the Bayes factor by $BF^n_{\bs,\bs_0}(\bta)$ instead of $BF^n_{\bs,\bs_0}$ to indicate it's dependence on $\bta$. If $\pi(\bta)$ is the prior for $\bta$, supported on $\bTheta$, then the integrated Bayes factor is given by
\begin{equation}
IBF^n_{\bs,\bs_0}=\int_{\bTheta}BF^n_{\bs,\bs_0}(\bta)\pi(\bta)d\bta.\notag
%\label{eq:ibf1}
\end{equation}
The following convergence result provides conditions under which the integrated Bayes factor converges to zero almost surely.
\begin{theorem}
\label{theorem:ibf_conv}
	Assume ($A1$)--($A3$) (or, ($A1^\prime$)--($A3^\prime$)) hold for some $\omega_{\bs} \in [0,1]$ and $\delta_{\bs}>0$, and for each $\bta\in\bTheta$, and that $\bTheta$ is compact.
	Let $g(n)=n^{1+2r\omega_{\bs}}$ (in the case of ($A1$)--($A3$), for Theorem \ref{theorem:as_conv});
	or $g(n)=n\log n$ and $n$ for $\omega_{\bs}=0$ and $\omega_{\bs}\in (0,1]$, respectively, 
	(in the case of ($A1^\prime$)--($A3^\prime$), for Theorem \ref{theorem:as_conv2}). 
Also assume the following:
\begin{enumerate}[(i)]
	\item $\log\left(BF^{n}_{\bs,\bs_0}(\bta)\right)/g(n)$ is stochastically equicontinuous,
	\item $E\left[\log\left(BF^{n}_{\bs,\bs_0}(\bta)\right) /g(n) \right]$ is equicontinuous with respect to $\bta$ as $n\rightarrow\infty$, and
	\item The $\limsup$ and $\liminf$ of $E\left[\log\left(BF^{n}_{\bs,\bs_0}(\bta)\right)/g(n)\right]$ are upper and lower semicontinuous in $\bta$, respectively.
\end{enumerate}
	Then, there exists $\delta_{\bs}>0$ such that
\begin{equation}
	%\max_{\bs\neq\bs_0}
	\limsup_n \frac{1}{g(n)}\log\left(IBF^{n}_{\bs,\bs_0}\right)\stackrel{a.s.}{=} -\delta_{\bs}.
\label{eq:ibf_conv2}
\end{equation}
\end{theorem}
%{\color{red} Check whether stochastic equicontinuity condition implies continuity of E(log BF/n). If so, then we don't need to assume semicontinuity. }
\begin{proof}
%Since the partial derivatives of $\frac{1}{n}\log\left(BF^{n}_{\bs,\bs_0}(\theta)\right)+\delta_{\bs}(\theta)$ with respect to the components of $\theta$
%exist and are almost surely bounded, it follows that
	%Since $\frac{1}{n}\log\left(BF^{n}_{\bs,\bs_0}(\theta)\right)-E\left[\frac{1}{n}\log\left(BF^{n}_{\bs,\bs_0}(\theta)\right)\right]$ is Lipschitz continuous on
%$\Theta$ as $n\rightarrow\infty$, it is stochastically equicontinuous on $\Theta$.
As assumptions ($A1$)--($A3$) (or ($A1^\prime$)--($A3^\prime$)) hold by hypothesis, Theorem \ref{theorem:power4} (or Theorem \ref{theorem:as_conv2}) holds. 
%Further, as $\log\left(BF^{n}_{\bs,\bs_0}(\theta)\right)/n$ is stochastically equicontinuous and $E\left[\log\left(BF^{n}_{\bs,\bs_0}(\theta)\right)/n \right]$ is equicontinuous w.r.t. $\theta$, the difference of above two functions is stochastically equicontinuous. Thus b
While proving the theorem we have shown 
$$\frac{1}{g(n)}\log\left(BF^{n}_{\bs,\bs_0}(\bta)\right)-E\left[\frac{1}{g(n)}\log\left(BF^{n}_{\bs,\bs_0}(\bta)\right)\right]\stackrel{a.s.}{\longrightarrow}0 \quad \mbox{pointwise in } \bta\in\bTheta.$$
%As $\log\left(BF^{n}_{\bs,\bs_0}(\bta)\right)/g(n)$ is stochastically equicontinuous and $E\left[\log\left(BF^{n}_{\bs,\bs_0}(\bta)\right)/g(n) \right]$ is equicontinuous w.r.t. $\bta$, 
By conditions (i) and (ii) of Theorem \ref{theorem:ibf_conv}, the difference of the above two functions is stochastically equicontinuous.
Further, as $\bTheta$ is compact, by the stochastic Ascoli lemma
(see, e.g., \ctn{Newey91}),
	$$\underset{\bta\in\bTheta}{\sup}~\left|\frac{1}{g(n)}\log\left(BF^{n}_{\bs,\bs_0}(\bta)\right)
	-E\left[\frac{1}{g(n)}\log\left(BF^{n}_{\bs,\bs_0}(\bta)\right)\right]\right|
	\stackrel{a.s.}{\longrightarrow} 0, ~\mbox{as}~ n\rightarrow\infty.$$
In other words, given any data sequence, for any $\epsilon>0$, there exists $n_0(\epsilon)$ such that for $n\geq n_0(\epsilon)$,
\begin{equation}
\left|\frac{1}{g(n)}\log\left(BF^{n}_{\bs,\bs_0}(\bta)\right)
-E\left[\frac{1}{g(n)}\log\left(BF^{n}_{\bs,\bs_0}(\bta)\right)\right]\right|<\epsilon/2, \label{eq:ibf_conv_limsup1}
\end{equation}
for all $\bta\in\bTheta$.
%	That is, it holds that for $n\geq n_0(\epsilon)$, for all $\bta\in\bTheta$,	
%	\begin{equation}
%		E_{\bs_0}\left\{\frac{1}{g(n)}\log\left(BF^{n}_{\bs,\bs_0}(\bta)\right)\right\}-\frac{\epsilon}{2}
%		<\frac{1}{g(n)}\log\left(BF^{n}_{\bs,\bs_0}(\bta)\right)<E_{\bs_0}\left\{\frac{1}{g(n)}\log\left(BF^{n}_{\bs,\bs_0}(\bta)\right)\right\}+\frac{\epsilon}{2}.
%		\label{eq:ibf_conv_limsup1}
%	\end{equation}

	Let us now  define $\overline{\delta_{\bs}(\bta)}$ and $\underline{\delta_{\bs}(\bta)}$ such that
	
	 $-\overline{\delta_{\bs}(\bta)}=\limsup_n~E_{\bs_0}\left\{\log\left(BF^{n}_{\bs,\bs_0}(\bta)\right)/g(n)\right\}$ and
	 
	$-\underline{\delta_{\bs}(\bta)}=\liminf_n~E_{\bs_0}\left\{\log\left(BF^{n}_{\bs,\bs_0}(\bta)/g(n)\right)\right\}$, where
	$\overline{\delta_{\bs}(\bta)},\underline{\delta_{\bs}(\bta)}> 0$ for all $\bta\in\bTheta$. By our assumption,
	$\overline{\delta_{\bs}(\bta)}$ is upper semicontinuous in $\bta$ and $\underline{\delta_{\bs}(\bta)}$ is lower semicontinuous in $\bta$.

	Now, by compactness of $\bTheta$, we have $\bTheta\subset\cup_{i=1}^{m}\tilde\bTheta_i$, for some finite $m>0$, where $\tilde\bTheta_i$ are such that
	$\sup_{\bta_1,\bta_2\in\tilde\bTheta_i}\|\bta_1-\bta_2\|<\delta$. Here $\delta~(>0)$ is such that
	\begin{equation}
		\left|E_{\bs_0}\left\{\frac{1}{g(n)}\log\left(BF^{n}_{\bs,\bs_0}(\bta_1)\right)\right\}-E_{\bs_0}\left\{\frac{1}{g(n)}\log\left(BF^{n}_{\bs,\bs_0}(\bta_2)\right)\right\}\right|
		<\frac{\epsilon}{6},
		\label{eq:equicon1}
	\end{equation}
	for large $n$, due to equicontinuity.
	 Now, for any $\bta\in\bTheta$, $\bta$ must lie in $\tilde\bTheta_i$ for some $i=1,2,\ldots,m$.
	Let $\bta_i\in\tilde\bTheta_i$, for $i=1,\ldots,m$.
	Then, let us write
\begin{align}
&	E_{\bs_0}\left\{\frac{1}{g(n)}\log\left(BF^{n}_{\bs,\bs_0}(\bta)\right)\right\}+\overline{\delta_{\bs}(\bta)}\notag \\
	&=\left[E_{\bs_0}\left\{\frac{1}{g(n)}\log\left(BF^{n}_{\bs,\bs_0}(\bta)\right)\right\}-E_{\bs_0}\left\{\frac{1}{g(n)}\log\left(BF^{n}_{\bs,\bs_0}(\bta_i)\right)\right\}\right]\notag \\	
	&\quad + \left[E_{\bs_0}\left\{\frac{1}{g(n)}\log\left(BF^{n}_{\bs,\bs_0}(\bta_i)\right)\right\}+\overline{\delta_{\bs}(\bta_i)}\right]-\left(\overline{\delta_{\bs}(\bta_i)}-\overline{\delta_{\bs}(\bta)}\right).
\label{eq:ibf_conv_limsup2}
\end{align}
	The first term on the right hand side of of (\ref{eq:ibf_conv_limsup2}) is less than $\epsilon/6$ due to (\ref{eq:equicon1}),
	since both $\bta,\bta_i\in\tilde\bTheta_i$.
	The second term on the right hand side of of (\ref{eq:ibf_conv_limsup2}) is less than $\epsilon/6$ for large enough $n$ by definition of $\limsup$.
	Since $m$ is finite, the requisite $n_1(\epsilon)$ that $n$ needs to exceed, remains finite for all values of $\bta$.
	The third term is less than $\epsilon/6$ by definition of upper semicontinuity, given that $\bta,\bta_i\in\tilde\bTheta_i$. In other words, for all $\bta\in\bTheta$,
	there exists $n_1(\epsilon)$, such that $n\geq n_1(\epsilon)$,
	\begin{equation}
		E_{\bs_0}\left\{\frac{1}{g(n)}\log\left(BF^{n}_{\bs,\bs_0}(\bta)\right)\right\}+\overline{\delta_{\bs}(\bta)}<\frac{\epsilon}{2}.
\notag	%	\label{eq:ibf_conv_limsup3}
	\end{equation}
	Similarly, using the definition of equicontinuity, $\liminf$ and lower semicontinuity, it follows that there exists $n_2(\epsilon)\geq 1$ for all $\bta\in\bTheta$
	such that for $n\geq n_2(\epsilon)$,
	\begin{equation}
		E_{\bs_0}\left\{\frac{1}{g(n)}\log\left(BF^{n}_{\bs,\bs_0}(\bta)\right)\right\}+\underline{\delta_{\bs}(\bta)}>-\frac{\epsilon}{2}.
\notag %\label{eq:ibf_conv_limsup4}
	\end{equation}
From (\ref{eq:ibf_conv_limsup1}) and the above facts, we see that for $n\geq n_3(\epsilon)=\max\{ n_1(\epsilon),n_2(\epsilon)\}$, and all $\bta\in\bTheta$,
	\begin{eqnarray*}
&		-\underline{\delta_{\bs}(\bta)}-\epsilon
		<\frac{1}{g(n)}\log\left(BF^{n}_{\bs,\bs_0}(\bta)\right)<-\overline{\delta_{\bs}(\bta)}+\epsilon,\\
		%\label{eq:ibf_conv_limsup5}
%	\end{equation*}
%or,
%	\begin{eqnarray*}
\implies &		\exp\left\{-g(n)(\epsilon+\underline{\delta_{\bs}(\bta)})\right\}<BF^{n}_{\bs,\bs_0}(\bta)<\exp\left\{g(n)(\epsilon-\overline{\delta_{\bs}(\bta)})\right\}. \notag %\label{eq:ibf_conv_limsup5}
	\end{eqnarray*}
	Integrating the above with respect to $\pi(\bta)d\bta$, and taking $g(n)^{-1}\log$ we obtain,
	\begin{equation}
		-\epsilon+\frac{1}{g(n)}\log\underline{I}_n<\frac{1}{g(n)}\log\left(IBF^{n}_{\bs,\bs_0}\right)<\epsilon+\frac{1}{g(n)}\log\overline{I}_n,
\notag %\label{eq:ibf_conv5}
	\end{equation}
	where $\overline I_n=\int_{\bTheta}\exp\left(-g(n)\overline{\delta_{\bs}(\bta)}\right)\pi(\bta)d\bta$,
	$\underline I_n=\int_{\bTheta}\exp\left(-g(n)\underline{\delta_{\bs}(\bta)}\right)\pi(\bta)d\bta$. Since both $\overline I_n$ and $\underline I_n$ are less than one, the statement in (\ref{eq:ibf_conv2}) holds. \qed
\end{proof}
\begin{remark}
	Note that a sufficient condition for stochastic equicontinuity of $\log\left(BF^{n}_{\bs,\bs_0}(\bta)\right)/g(n)$ is almost sure Lipschitz continuity of the same, with a bounded Lipschitz constant,
	as $n\rightarrow\infty$. Similarly, a sufficient condition of equicontinuity of $E_{\bs_0}\left[\log\left(BF^{n}_{\bs,\bs_0}(\bta)\right)/g(n)\right]$ is Lipschitz continuity. Again, Lipschitz continuity is ensured by boundedness of the partial derivatives. Hence, if
	the partial derivatives of $\log\left(BF^{n}_{\bs,\bs_0}(\bta)\right)/g(n)$ and its expectation with respect to the components of $\bta$
	exist and are almost surely bounded for large $n$, then Lipschitz continuity would follow. This would also imply the semicontinuity assumptions on $E_{\bs_0}\left\{\log\left(BF^{n}_{\bs,\bs_0}(\bta)\right)/g(n)\right\}$. In our applications, we shall often make use of this sufficient condition.
\end{remark}

\begin{remark}
	Note that Theorem \ref{theorem:ibf_conv} is applicable to Gaussian process regression setup where the error variance $\sigma^2_{\epsilon}$, the 
	process variance $\sigma^2_f$, or the diagonal elements of $D_{\bs}$ are unknown. The relevant priors, however, need to have compact supports.
	Although for $\sigma^2_{\epsilon}$ and $\sigma^2_f$ compactly supported prior is not necessary for proving convergence of Bayes factor 
	(as we have shown consistency under an inverse-gamma prior setup with $\sigma^2_{\epsilon}=\sigma^2_f$), but very general priors, 
	albeit with compact supports, can be envisaged for these unknown quantities, without any loss of generality of convergence
	result for the corresponding integrated Bayes factor. In real problems, some other parameters may be assigned compactly supported priors, while
	the inverse-gamma prior may be allotted to the variance parameters.
	
	For illustration of the method for verifying the conditions of Theorem \ref{theorem:ibf_conv},
	in Section \ref{subsec:time_series1} we consider the case of variable selection in an autoregressive regression model with unknown autoregressive parameter.
\end{remark}

\section{Bayes factor asymptotics for correlated errors}
\label{sec:correlated_errors}
So far we assumed $\epsilon_i\stackrel{iid}{\sim}N(0,\sigma^2_{\epsilon})$. However, correlated errors play significant roles in time series models.
Indeed, except some simple cases, \emph{i.i.d.} errors will not be appropriate for such models. For instance, the problem of time-varying covariate selection
in the AR(1) model $y_t=\rho_0y_{t-1}+\sum_{i=0}^{|\bs|}\beta_ix_{it}+\epsilon_t$, $t=1,2,\ldots$, where $\epsilon_t\stackrel{iid}{\sim}N(0,\sigma^2_{\epsilon})$
and $\rho_0$ is known, admits the same treatment as in linear regression considered in Section \ref{sec:linear_regression} by treating
$z_t=y_t-\rho_0y_{t-1}$ as the response. However if $\rho_0$ is unknown, such simple method is untenable.

In general, we must allow correlated errors, that is, for $\bepsilon_n=(\epsilon_1,\ldots,\epsilon_n)^T\sim N_n\left(\bzero,\sigma^2_{\epsilon}\tilde\Sigma_n\right)$,
the zero-mean normal distribution with covariance matrix $\sigma^2_{\epsilon}\tilde\Sigma_n$. Let the correlation matrix under the true model be
$\tilde\Sigma^t_n$. With these, we then replace the previous notions $\sigma^2_{\epsilon}I_n+\Sigma_{\bs}$ and $\sigma^2_{\epsilon}I_n+\Sigma^t_{\bs_0}$
by $\sigma^2_{\epsilon}\tilde\Sigma_n+\Sigma_{\bs}$ and $\sigma^2_{\epsilon}\tilde\Sigma^t_n+\Sigma^t_{\bs_0}$, respectively, and prove similar results with the 
assumptions on $A_{n,\bs}$ and $\Delta_{n,\bs}$, where
%, suppressing the dependence of $\Sigma^t_{n,\bs}$ and $\Sigma^t_{n,\bs_0}$ on $n$.
 $A_{n,\bs}=\left(\sigma^2_{\epsilon}\tilde\Sigma^t_n+\Sigma^t_{\bs_0}\right)\left(\sigma^2_{\epsilon}\tilde\Sigma_n+\Sigma_{\bs}\right)^{-1}$, and $\Delta_{n,\bs}=(\bmu_{\bs}-\bmu^t_{\bs_0})^T
\left(\sigma^2_{\epsilon}\tilde\Sigma_n+\Sigma_{\bs}\right)^{-1}(\bmu_{\bs}-\bmu^t_{\bs_0})$.

\subsection{Illustration 3: Autoregressive model}
\label{subsec:time_series1}

Let us consider the time-varying covariate selection problem in the following $AR(1)$ model.
Let
\begin{align}
y_t=\rho y_{t-1}+\bbeta_{\bs}^{\prime} \bx_{t,\bs}+\epsilon_t,\quad \mbox{and} ~~
\epsilon_t \stackrel{iid}{\sim}N\left(0,\sigma^2_{\epsilon}\right),\quad \mbox{for}~~ t=1,\ldots,n.
\label{eq:ar1_model}
\end{align}
where $y_0\equiv 0$ and $|\rho|<1$.
The above model admits the following representation
\begin{equation*}
y_t=\bbeta_{\bs}^{\prime} \bz_{t,\bs}+\tilde\epsilon_t,
% \label{eq:ar1_1}
% \end{equation}
\quad \mbox{where}~~ \bz_{t,\bs}=\sum_{k=1}^t\rho^{t-k}\bx_{k,\bs} \quad \mbox{ and}~
%\begin{equation}
~\tilde\epsilon_t=\sum_{k=1}^t\rho^{t-k}\epsilon_k.
%\label{eq:eps1}
\end{equation*}
Thus, $\tilde\epsilon_t$ is an asymptotically stationary zero mean Gaussian process with covariance 
\begin{equation}
Cov\left(\tilde\epsilon_{t+h},\tilde\epsilon_t\right)\sim\frac{\sigma^2_{\epsilon}\rho^h}{1-\rho^2},~\mbox{where}~h\geq 0.
\label{eq:cov_ar1}
\end{equation}
Let the true model be of the same form as above but with $\rho$ and $\bs$ replaced by $\rho_0$ and $\bs_0$, respectively, where $|\rho_0|<1$. As in the linear regression case we allow $p=O(n^r)$ covariates, with $r>0$, and $\bs_{0}\subseteq\bS=\{1,\ldots,p\}$.
%We further assume that the effective dimensionality of $\bs$ and $\bs_{0}$ are of order $O(p^{\omega_{\bs}}) $ for some $0\leq \omega_{\bs}\leq 1$.
 
Let
$\bbeta_{\bs}\sim N\left(\bbeta_{0,\bs},g_n \sigma^2_{\bbeta}\left( Z_{\bs}^{\prime} Z_{\bs}\right)^{-1} \right)$, where $Z_{\bs}$ is the design matrix
associated with $\bz_{t,\bs}$; $t=1,\ldots,n$, and $g_n=O\left(1\right)$.
This is again Zellner's $g$ prior, but modified to suit the $AR(1)$ setup. 
 % is the set of indices of the true set of covariates 
%{\color{red}and $\|\bbeta_{0,\bs}\|_{L_1}=\sum_{=1}^{|\bs|}|\beta_{0,j}|=O(p^{\omega}n^{c_{0}})$, for $0\leq \omega\leq 1$ and $c_{0}<1/4$, where $p=O(n^r)$, for $r>0$.} 
As before, $\bbeta_{0,\bs}$ is so chosen that $\|\bbeta_{0,\bs}\|_{L_1}=\sum_{=1}^{|\bs|}|\beta_{0,j}|=O(|\bs|)$.
We also assume compactness of the covariate space and that
the set of covariates $\left\{x_j:j\in\bS\right\}$ is non-zero. Let $\pi(\rho)$ be any prior for $\rho$ supported on $[-1+\gamma,1-\gamma]$ for some small enough $\gamma>0$. The reason for choosing this
support will become clear as we proceed.

The conditional expectation of $\bbeta_{\bs}^{\prime}\bz_{t,\bs}$ given $\rho$, for $t=1,\ldots,n$, is
\begin{equation*}
\mu\left(\bz_{t,\bs}\right)=\bbeta_{0,\bs}^{\prime}\bz_{t,\bs},
% \label{eq:mean_regression}
\end{equation*}
and the covariance between $\bbeta_{\bs}^T\bz_{i,\bs}$ and $\bbeta_{\bs}^T\bz_{j,\bs}$ given $\rho$ is
\begin{align*}
Cov\left(\bbeta_{\bs}^{\prime}\bz_{i,\bs},\bbeta_{\bs}^{\prime}\bz_{j,\bs}\right)
&=\sigma^2_{\bbeta}g_n \bz^{\prime}_{i,\bs}\left(Z_{\bs}^{\prime}Z_{\bs}\right)^{-1}\bz_{j,\bs}.%\notag\\
%&=\left\{\begin{array}{cc}\sigma^2_{\bbeta}\|\bx_{i,\bs}\|^2, & \mbox{if}~~i=j;\\ 0, & \mbox{otherwise.}\end{array}\right.
% \label{eq:cov_regression}
\end{align*}
Let $\Sigma_{\epsilon}$ be the AR(1) correlation matrix $((\rho^{h}))$, $\sigma^2_{\epsilon}\tilde\Sigma_n$ be the covariance matrix of $\tilde{\boldsymbol{\epsilon}}$ as given in (\ref{eq:cov_ar1}), i.e., $\tilde\Sigma_n=(1-\rho^2)^{-1}\Sigma_\epsilon$, $H_{n,\bs}:=\left(\sigma^2_{\epsilon}\tilde\Sigma_n+\sigma^2_{\bbeta}g_nP_{n,\bs}\right)$ and $H_{n,\bs_0}:=\sigma^2_{\epsilon}\tilde\Sigma_n+\sigma^2_{\bbeta}g_nP_{n,\bs_0}$, where $P_{n,\bs}$ is the projection matrix onto the column space of $Z_{\bs}$. Then $A_{n,\bs}=H_{n,\bs_0}H_{n,\bs}^{-1}$.

We first verify (A1)--(A3) in this setup. 
%Note that assumption ($A3$) easily holds as in the case of linear regression since the covariate space is compact, $\max\{|\bs|,|\bs_{0}|\}=O(p^{2\omega})$, and both $|\rho_0|$ and $|\rho|$ are less than one.
For verification of (A3), note that
\begin{eqnarray*}
	\left\|\bmu_{0,\bs}-\bmu^t_{0,\bs_0} \right\|^2 =\left\|Z_{\bs}\bbeta_{0,\bs}- Z_{\bs_0}\bbeta_{0,\bs_0} \right\|^2 
	\leq 2 \left( \left\| Z_{\bs}\bbeta_{0,\bs} \right\|^2+ \left\| Z_{\bs_0}\bbeta_{0,\bs_0} \right\|^2 \right).
%	= 2\sum_{i=1}^n \left\{\left( x_{i,\bs}^\prime \bbeta_{\bs}\right)^2+\left( x_{i,\bs_0}^\prime \bbeta_{\bs_0}\right)^2\right\}.
\end{eqnarray*}
Now, by our assumptions, for any $\bs$, $\|\bbeta_{0,\bs}\|_{L_1}=O(|\bs|)$. We further assume that
$\max\{|\bs|,|\bs_{0}|\}=O(p^{2\omega_{\bs}})$, for $0\leq \omega_{\bs}\leq 1$. 
Also, since the covariates lie on a compact space and  
$|\rho|$ is less than one, it follows that
$\left\| Z_{\bs}\bbeta_{0,\bs} \right\|^2=\sum_{t=1}^n\left(\bz^T_{t,\bs}\bbeta_{0,\bs}\right)^2=O(np^{2\omega_{\bs}})=O(n^{1+2r\omega_{\bs}})$.
Similarly, since $|\rho_0|<1$, $\left\| Z_{\bs_0}\bbeta_{0,\bs} \right\|^2=O(n^{1+2r\omega_{\bs}})$.
Thus ($A3$) holds.

Next we verify ($A2$). Note that by Lemma \ref{lm_3} (in Section \ref{sec:appendix} of the supplement) the eigenvalues of $\Sigma_{\epsilon}/(1-\rho^2)$ have strictly positive lower and upper bounds, independent of
$n$ if $\rho\in [-1+\gamma,1-\gamma]$. Further, the eigenvalues of $P_{n,\bs}$ are either 0 or 1. Thus by Result \ref{eqn:symmetric2}, 
$\lambda_{\max}(A_{n,\bs})=O(p^{\omega_{\bs}})$.

% using Gerschgorin's circle theorem and the fact that $|\rho_0|<1$, it is seen that the maximum
% eigenvalue of $\tilde \Sigma_n^t$ is bounded above by
% $1+2/(1-|\rho_0|)$, and the minimum eigenvalue of $\tilde \Sigma_n$ is bounded below by. It then follows as before that $\lambda_{\max}(A_{n,\bs})=O(1)$.

Assuming, as before, that $\liminf_n n^{-1-2r\omega_{\bs}}\|\bmu_{\bs}-\bmu_{\bs_0}^{t}\|^2>0$, it is seen that ($A1$) also holds.
%
%Let us now verify the conditions of Theorem \ref{theorem:ibf_conv}. Assuming that the limit of $\frac{1}{n}\log \left(BF^n_{\bs,\bs_0}(\rho)\right)$ exists almost surely,
%it follows due to verification of ($A1$)--($A3$) that $\frac{1}{n}\log \left(BF^n_{\bs,\bs_0}(\rho)\right)\stackrel{a.s.}{\longrightarrow}-\delta_{\bs}(\rho)$, so that
%(\ref{eq:ibf_conv1}) holds.
%
Thus, ($A1$)--($A3$) holds.

Next we verify conditions (i)--(iii) of Theorem \ref{theorem:ibf_conv}.
Note that
\begin{align}
	&\frac{\partial}{\partial\rho}\left(\frac{1}{n^{1+2r\omega_{\bs}}}\log BF^n_{\bs,\bs_0}(\rho)\right)\notag \\
	&= -\frac{1}{2n^{1+2r\omega_{\bs}}}tr\left[H_{n,\bs}^{-1}
	\frac{\partial}{\partial\rho}\left(H_{n,\bs}\right)\right] +\frac{1}{n^{1+2r\omega_{\bs}}}\left(\frac{\partial\bmu_{\bs}}{\partial\rho}\right)^TH_{n,\bs}^{-1}
	(\by_n-\bmu_{\bs})	\notag\\
	& \qquad
	+\frac{1}{2n^{1+2r\omega_{\bs}}}(\by_n-\bmu_{\bs})^TH_{n,\bs}^{-1} \frac{\partial}{\partial\rho}\left(H_{n,\bs}\right)H_{n,\bs}^{-1}(\by_n-\bmu_{\bs})	.
	\label{eq:ar1_ibf1}
\end{align}
%In (\ref{eq:ar1_ibf1}), $P_{n,\bs}$ is the projection matrix on the space of $Z_{\bs}$, and $\Sigma_{\epsilon}=\left(( \rho^{|i-j|} )\right)$.
Consider the first term of (\ref{eq:ar1_ibf1}).  
%As before, using Lemma \ref{lm_3} and 
%the eigenvalues of $\Sigma_{\epsilon}/(1-\rho^2)$ are positive and bounded.
%satisfies $|\lambda-1|<\frac{2|\rho|}{1-|\rho|}$, so that $\lambda>\frac{1-3|\rho|}{1-|\rho|}$, which is
%positive if $|\rho|<1/3$. Hence, there exists $\varepsilon>0$ such that
%$\lambda_1\left(\frac{\sigma^2_{\epsilon}}{1-\rho^2}\Sigma_{\epsilon}+\sigma^2_{\bbeta}g_nP_{n,\bs}\right)^{-1}
%=\frac{1}{\lambda_n\left(\frac{\sigma^2_{\epsilon}}{1-\rho^2}\Sigma_{\epsilon}+\sigma^2_{\bbeta}g_nP_{n,\bs}\right)}<\frac{1-\rho^2}{\varepsilon\sigma^2_{\epsilon}}$.
%Using 
%the Result \ref{eqn:symmetric2}
% that for symmetric matrices $A_1$ and $A_2$ of order $n$, %$\lambda_1(A_1+A_2)\leq\lambda_1(A_1)+\lambda_1(A_2)$ and
% \begin{equation}
% \lambda_{\min}(A_1)+\lambda_{\min}(A_2) \leq \lambda_{\min}(A_1+A_2)\leq \lambda_{\max}(A_1+A_2) \leq \lambda_{\max}(A_1)+\lambda_{\max}(A_2), \label{eqn:symmetric}
% \end{equation}
%We have already proved (while verifying (A2)) that 
By Lemma \ref{lm_3} $H_{n,\bs}$ has positive and bounded eigenvalues.
% Also, note that trace and differentiation (w.r.t. $\rho$), being linear operators, are interchangeable, and therefore 
%$tr\left[\frac{\partial }{\partial \rho}\left(\sigma^2_{\epsilon}(1-\rho^2)^{-1}\Sigma_{\epsilon}+\sigma^2_{\bbeta}g_nP_{n,\bs}\right)\right]=O(n)$.
Define $D_{n,\bs}=\displaystyle\frac{\partial}{\partial\rho}H_{n,\bs}$, and
note that
\begin{equation}
D_{n,\bs}=\frac{\sigma^2_{\epsilon}}{1-\rho^2}A_n(\rho)+\frac{2\sigma^2_{\epsilon}\rho}{(1-\rho^2)^2}{\Sigma}_{\epsilon}+\sigma^2_{\bbeta}g_n\frac{\partial}{\partial\rho}P_{n,\bs}, \label{eqn_11} 
\end{equation}
 where $A_n(\rho)=((a_{i,j}))$ is defined by $a_{i,j}=|i-j|\rho^{|i-j|-1}$.
%, and $B_n(\rho)$ is defined by $-\rho^h$, the diagonal elements being $-1$. 
We will show that $D_{n,\bs}$ has finite eigenvalues. 
From Lemma \ref{lm_3} and Lemma \ref{lm_2}, and the fact that $\rho\in [-1+\gamma,1-\gamma]$, it is evident that the 2nd and 3rd matrices in the RHS of (\ref{eqn_11}) have bounded eigenvalues if $g_n$ is bounded. 
As both the matrices are symmetric, it follows from Result \ref{eqn:symmetric2} that the sum of these two matrices have finite eigenvalues.
From Gerschgorin's circle theorem, $\lambda_{\max}\left(A_n(\rho)\right)\leq \max_{j} R_{[j]}$, and $\lambda_{\min}\left(A_n(\rho)\right)\geq -\max_{j} R_{[j]}$ where $R_{[j]}$ is the sum of the absolute values of the non-diagonal entries in the $[j]$-th row of $A_{n}(\rho)$ and $[j]$ is the highest integer less than or equal to $j$. Little algebra shows that $\max_{j} R_{[j]}=R_{[n/2]}=2\left\{1-|\rho|^{[n/2]} -[n/2] |\rho|^{[n/2]}(1-|\rho|)\right\}(1-|\rho|)^{-2}$. As $n$ is large $(1-|\rho|)^{-2}<R_{[n/2]}<2(1-|\rho|)^{-2}$, which implies that the eigenvalues of the 1st matrix of RHS of (\ref{eqn_11}) are bounded. Thus, $D_{n,\bs}$ has bounded eigenvalues by Result \ref{eqn:symmetric2}. 

Let $\alpha_0>0$ be such that $\lambda_{\min}(D_{n,\bs})>-\alpha_0$. Then $D_{n,\bs}+\alpha_0 I$ is a symmetric positive definite matrix.
Hence, the absolute value on first term of (\ref{eq:ar1_ibf1}) is
\begin{eqnarray}
	\left|\frac{1}{2n^{1+2r\omega_{\bs}}}tr\left(H_{n,\bs}^{-1}
		D_{n,\bs}\right)\right|&=& \left|\frac{1}{2n^{1+2r\omega_{\bs}}}tr\left[H_{n,\bs}^{-1}
		(D_{n,\bs}+\alpha_0I)-\alpha_0 H_{n,\bs}^{-1} \right]\right|\notag\\
	&\leq& \frac{1}{2n^{1+2r\omega_{\bs}}}\left[\left|trH_{n,\bs}^{-1}
(D_{n,\bs}+\alpha_0I)\right|+\alpha_0 \left|trH_{n,\bs}^{-1} \right|\right]\notag\\
		&\leq&\frac{1}{2n^{1+2r\omega_{\bs}}}\lambda_1\left(H_{n,\bs}^{-1}\right)
	tr\left(D_{n,\bs}+\alpha_0 I\right) \notag \\
	&& \qquad \qquad +\frac{\alpha_0}{2n^{1+2r\omega_{\bs}}} tr\left(H_{n,\bs}^{-1}\right) =O(1). \notag
%	\label{eq:ar1_ibf2}
\end{eqnarray}
The last equality holds as the eigenvalues of $H_{n,\bs}$ are positive and bounded, that of $D_{n,\bs}$ are bounded, and $\alpha_0$ is finite.

%\vskip10pt

Next consider the third term of (\ref{eq:ar1_ibf1}). Let $H_{n,\bs}^{-1}
	(\by_n-\bmu_{\bs})={\bf u}_{\bs}$, then this term is ${\bf u}_{n,\bs}^T D_{n,\bs}{\bf u}_{n,\bs}/2n^{1+2r\omega_{\bs}}$. 
% Using the notation $A_1\preceq A_2$ for matrices $A_1$ and $A_2$ to denote $A_2-A_1$ is nonnegative definite, we exploit the result that for any symmetric matrix $A$,
% \begin{equation}
% % \lambda_{\min}(A)=\inf_{{\bf u}: \|{\bf u}\|=1} {\bf u}^T A{\bf u}\leq \sup_{{\bf u}: \|{\bf u}\|=1} {\bf u}^T A{\bf u}= \lambda_{\max}(A), \label{eqn:symmetric2}
% 	\lambda_{\min}(A)I\preceq A\preceq\lambda_{\max}(A)I, \label{eqn:symmetric2}
% \end{equation}
Using Result \ref{eqn:symmetric2} we argue that the third term of (\ref{eq:ar1_ibf1}) is lower bounded by $\lambda_{\min}(D_{n,\bs})\|{\bf u}_{n,\bs}\|^2/2n^{1+2r\omega_{\bs}}$ 
and upper bounded by $\lambda_{\max}(D_{n,\bs})\|{\bf u}_{n,\bs}\|^2/2n^{1+2r\omega_{\bs}}$.	
Using Result \ref{eqn:symmetric2}, it can also be shown that $\|{\bf u}_{n,\bs}\|^2$ is bounded by 
$\lambda_{\max}^{-2}( H_{n,\bs})\|\by_n-\bmu_{\bs} \|^2$ and $\lambda_{\min}^{-2}( H_{n,\bs})\|\by_n-\bmu_{\bs} \|^2$.

\noindent Next, we write $\by_n-\bmu_{\bs_0}=H_{n,\bs_0}^{1/2}\tilde\bz_n$,
where $\tilde\bz_n\sim N\left(\bzero,I_n\right)$. It then follows that
\begin{eqnarray*}
&&\|\by_n-\bmu_{\bs}\|^2\\
&&=\|\by_n-\bmu_{\bs_0}^{t}\|^2+\|\bmu_{\bs_0}^{t}-\bmu_{\bs}\|^2+2(\by_n-\bmu_{\bs})^T (\bmu_{\bs_0}^{t}-\bmu_{\bs}) \\
&&\leq \tilde\bz_n^T H_{n,\bs_0} \tilde\bz_n+2\|\bmu_{\bs_0}^{t}-\bmu_{\bs}\|\sqrt{\tilde\bz_n^T H_{n,\bs_0} \tilde\bz_n}+\|\bmu_{\bs_0}^{t}-\bmu_{\bs}\|^2\\
&& \leq \lambda_{\max} (H_{n,\bs_0}) \left\|\tilde\bz_n\right\|^2+2\|\bmu_{\bs_0}^{t}-\bmu_{\bs}\| \lambda^{1/2}_{\max} (H_{n,\bs_0}) \left\|\tilde\bz_n\right\|+\|\bmu_{\bs_0}^{t}-\bmu_{\bs}\|^2.
\end{eqnarray*}
% \begin{align}
% & tr\left[(\by_n-\bmu_{n,\bs})(\by_n-\bmu_{n,\bs})^T\right]\notag\\
% &\leq tr\left(Q_n\bz_n\bz^T_n\right)+2\left|\bz^T_nQ^{1/2}_n(\bmu_{n,\bs}-\bmu_{n,\bs_0})\right|+(\bmu_{n,\bs}-\bmu_{n,\bs_0})^T(\bmu_{n,\bs}-\bmu_{n,\bs_0})\notag\\
% 	&\leq \lambda_1(Q_n)tr(\bz_n\bz^T_n)+\sqrt{\bz^T_nQ^{1/2}_n\bz_n}\times\sqrt{(\bmu_{n,\bs}-\bmu_{n,\bs_0})^TQ^{1/2}_n(\bmu_{n,\bs}-\bmu_{n,\bs_0})}\notag\\
% 	&\qquad+(\bmu_{n,\bs}-\bmu_{n,\bs_0})^T(\bmu_{n,\bs}-\bmu_{n,\bs_0}).
% 	\label{eq:ar1_ibf5}
% \end{align}
% Now, $tr(\bz_n\bz^T_n)=\bz^T_n\bz_n$, $\bz^T_nQ^{1/2}_n\bz_n=tr(Q^{1/2}_n\bz_n\bz^T_n)\leq\lambda_1(Q^{1/2}_n)\bz^T_n\bz_n$, and
% $(\bmu_{n,\bs}-\bmu_{n,\bs_0})^TQ^{1/2}_n(\bmu_{n,\bs}-\bmu_{n,\bs_0})\leq\lambda_1(Q^{1/2}_n)(\bmu_{n,\bs}-\bmu_{n,\bs_0})^T(\bmu_{n,\bs}-\bmu_{n,\bs_0})$.
Combining the facts that $\lambda_{\max}(H_{n,\bs_0})$ is bounded, $\tilde\bz_n^T\tilde\bz_n/n\rightarrow 1$ almost surely as $n\rightarrow\infty$,
$\|\bmu_{\bs}-\bmu_{\bs_0}\|^2/n^{1+2r\omega_{\bs}}=O(1)$, and since $|\rho_0|,|\rho|<1-\gamma$, almost surely,
it follows that $\left\|\by_n-\bmu_{\bs}\right\|^2/n^{1+2r\omega_{\bs}}=O(1)$, almost surely. % because of compactness of the covariate space
% It then follows from (\ref{eq:ar1_ibf4}) that
% \begin{align}
% &\frac{1}{2n}\left|tr\left[\left(\frac{\sigma^2_{\epsilon}}{1-\rho^2}\tilde\Sigma_n+\sigma^2_{\bbeta}g_nP_{n,\bs}\right)^{-1}
% \frac{\tilde A_n(\rho)}{1-\rho}
% \left(\frac{\sigma^2_{\epsilon}}{1-\rho^2}\tilde\Sigma_n+\sigma^2_{\bbeta}g_nP_{n,\bs}\right)^{-1}
% (\by_n-\bmu_{n,\bs})(\by_n-\bmu_{n,\bs})^T\right]\right|\notag\\
% 	&\qquad=O(1),~\mbox{almost surely.}
% \label{eq:ar1_ibf6}
% \end{align}
% In the same way,
% \begin{align}
% 	&\frac{1}{2n}\left|tr\left[\left(\frac{\sigma^2_{\epsilon}}{1-\rho^2}\tilde\Sigma_n+\sigma^2_{\bbeta}g_nP_{n,\bs}\right)^{-1}
% 	\left(\frac{B(\rho)}{(1-\rho)^2}+\sigma^2_{\bbeta}g_n\frac{dP_{n,\bs}}{d\rho}\right)
% 	\left(\frac{\sigma^2_{\epsilon}}{1-\rho^2}\tilde\Sigma_n+\sigma^2_{\bbeta}g_nP_{n,\bs}\right)^{-1}
% 	(\by_n-\bmu_{n,\bs})(\by_n-\bmu_{n,\bs})^T\right]\right|\notag\\
% 	&\qquad=O(1),~\mbox{almost surely.}
% \label{eq:ar1_ibf7}
% \end{align}
In other words, the third term of (\ref{eq:ar1_ibf1}) is $O(1)$ almost surely, as $n\rightarrow\infty$.

\vskip10pt
For the second term of (\ref{eq:ar1_ibf1}), note that
\begin{eqnarray}
	\left|\left(\frac{d\bmu_{\bs}}{d\rho}\right)^TH_{n,\bs}^{-1}~
	(\by_n-\bmu_{\bs})\right|
	 \leq \sqrt{\left(\frac{d\bmu_{\bs}}{d\rho}\right)^T 
	\left(\frac{d\bmu_{\bs}}{d\rho}\right)}\times \left\| {\bf u}_{\bs}\right\|. \notag
%	\sqrt{(\by_n-\bmu_{\bs})^TH_{n,\bs}^{-2}~(\by_n-\bmu_{\bs})}. \notag
%	\label{eq:ar1_ibf8}
\end{eqnarray}
Note that
\begin{eqnarray*}
 \frac{\partial}{\partial\rho}\bmu_{\bs}=\frac{\partial}{\partial\rho} \left( \bbeta_{0,\bs}^T \bz_{1,\bs}, \ldots , \bbeta_{0,\bs}^T\bz_{n,\bs} \right)^T 
\end{eqnarray*}
Thus,
\begin{eqnarray*}
 \left(\frac{\partial}{\partial\rho}\bmu_{\bs}\right)^T \left( \frac{\partial}{\partial\rho}\bmu_{\bs}\right)&=&\sum_{t=1}^{n}\left(\frac{\partial}{\partial\rho}\bbeta_{0,\bs}^T \bz_{t,\bs}\right)^2=\sum_{t=1}^{n}\left(\frac{\partial}{\partial\rho}\bbeta_{0,\bs}^T \sum_{k=1}^{t}\rho^{t-k}\bx_{k,\bs}\right)^2\\
 &=&\sum_{t=1}^{n}\left\{\sum_{k=1}^{t-1}(t-k)\rho^{t-k-1}\bbeta_{0,\bs}^T\bx_{k,\bs}\right\}^2\\ 
 &\leq& M_{n} \sum_{t=1}^{n}\left\{\sum_{k=1}^{t-1}(t-k)\rho^{t-k-1}\right\}^2,\\
 &\leq&  M_{n} \sum_{t=1}^{n}\left\{ \frac{1-\rho^{t-1}(t-t\rho+\rho)}{(1-\rho)^2}\right\}^2.
\end{eqnarray*}
for some appropriate $M_{n}=O(p^{2\omega_{\bs}})$ as $\bx_{k,\bs}$ is uniformly bounded for all $k$, and $\|\bbeta_{0,\bs}\|^2=O(\|\bs\|^{2})=O(p^{2\omega_{\bs}})$. 
As $|\rho|<1-\gamma$, the last expression is $O(n)$.
Moreover, as $\lambda_{\max}\left[H_{n,\bs}^{-2}\right]$ is bounded and $\|\by_n-\bmu_{\bs}\|^2/n^{1+2r\omega_{\bs}}$ is $O(1)$ almost surely, 
% Since $|\rho|<1/3$, $\left(\frac{d\bmu_{n,\bs}}{d\rho}\right)^T\frac{d\bmu_{n,\bs}}{d\rho}$ is $O(n)$, and moreover,
% $\lambda_{\max}\left[\left(\frac{\sigma^2_{\epsilon}}{1-\rho^2}\tilde\Sigma_n+\sigma^2_{\bbeta}g_nP_{n,\bs}\right)^{-2}\right]$ is bounded.
% Then applying the result $tr(AB)\leq \lambda_1(A)tr(B)$ for any symmetric matrix $A$ and non-negative definite matrix $B$ to the right hand side of (\ref{eq:ar1_ibf8}),
it follows that the second term of (\ref{eq:ar1_ibf1}) is $O(1)$ almost surely.
%(\ref{eq:ar1_ibf8})

In other words, all the three terms of (\ref{eq:ar1_ibf1}) are $O(1)$ almost surely, as $n\rightarrow\infty$. That is,
almost surely, as $n\rightarrow\infty$,
\begin{equation}
	\frac{d}{d\rho}\left(\frac{1}{n^{1+2r\omega_{\bs}}}\log BF^n_{\bs,\bs_0}(\rho)\right)=O(1).
	\label{eq:ar1_ibf9}
\end{equation}
Thus, for any given data sequence in the relevant non-null set, the function $\log BF^n_{\bs,\bs_0}(\rho)/n^{1+2r\omega_{\bs}}$ is Lipschitz continuous in $\rho$.
Importantly, (\ref{eq:ar1_ibf9}) shows that there exists $n_0\geq 1$, such that for $n\geq n_0$, the Lipschitz constant for $\log BF^n_{\bs,\bs_0}(\rho)/n^{1+2r\omega_{\bs}}$
remains the same.
In the same way, it can be shown that $E_{\bs_0}\left[\log BF^n_{\bs,\bs_0}(\rho)/n^{1+2r\omega_{\bs}}\right]$ is also Lipschitz in $\rho$, with bounded Lipschitz constant,
as $n\rightarrow\infty$.

%Now note that for $\rho_1,\rho_2$ in the relevant parameter space,
%\begin{equation}
%	\left|-\delta_{\bs}(\rho_1)+\delta_{\bs}(\rho_2)\right|
%	=\lim_{n\rightarrow\infty}\left|\frac{1}{n}\log BF^n_{\bs,\bs_0}(\rho_1)-\frac{1}{n}\log BF^n_{\bs,\bs_0}(\rho_2)\right|
%	\leq L\left|\rho_1-\rho_2\right|,
%	\label{eq:ar1_ibf10}
%\end{equation}
%where $L~(>0)$, is the Lipschitz constant for $\frac{1}{n}\log BF^n_{\bs,\bs_0}(\rho)$, as $n\rightarrow\infty$.
%Thus, (\ref{eq:ar1_ibf10}) shows that the limiting function $-\delta_{\bs}(\rho)$ is also Lipschitz continuous in $\rho$, with the same Lipschitz constant $L$.
%Hence, $\frac{1}{n}\log BF^n_{\bs,\bs_0}(\rho)+\delta_{\bs}(\rho)$ is Lipschitz for large $n$,
%and converges pointwise to zero, almost surely. Additionally, compactness of the parameter space
%$[-1/3,1/3]$ of $\rho$ ensures by stochastic Ascoli lemma that
%$\underset{\rho\in[-\frac{1}{3},\frac{1}{3}]}{\sup}~\left|\frac{1}{n}\log BF^n_{\bs,\bs_0}(\rho)+\delta_{\bs}(\rho)\right|\stackrel{a.s.}{\longrightarrow}0$.

Further, assuming that the $\limsup$ and $\liminf$ of $E_{\bs_0}\left[\log BF^n_{\bs,\bs_0}(\rho)/n^{1+2r\omega_{\bs}}\right]$
are upper and lower semicontinuous, respectively, and appealing to Theorem \ref{theorem:ibf_conv}, we see that
(\ref{eq:ibf_conv2}) holds.
%$\min_{\bs}\limsup_n\frac{1}{n}\log\left(IBF^n_{\bs,\bs_0}\right)\stackrel{a.s.}{=}-\delta$,
%for some $\delta>0$.

We summarize this in the form of the following theorem.

\begin{theorem}
\label{theorem:ar1}
Consider the model selection problem in the $AR(1)$ model (\ref{eq:ar1_model}) with  $p=O\left(n^r\right)$, with $r>0$. Suppose a prior
$\pi$ supported on $[-1+\gamma,1-\gamma]$ is assigned on $\rho$, and 
 $\rho_0$ is the true value of $\rho$, where $|\rho_0|<1-\gamma$, for some $\gamma>0$.
% where $ I_{\left|\bs\right|}$ is the identity matrix of order $\left|\bs\right|$, the cardinality of $\bs$.
%This is the well known Zellner's $g$ prior.
Let $\bs_0, ~\bs~(\subseteq\bS=\{1,2,\ldots,p\})$ be the set of
	indices of the true set of covariates, and a competing model. Assume that $\max\left\{|\bs_{0}|,|\bs|\right\}=O(p^{\omega_{\bs}})$, for some $0\leq \omega_{\bs}\leq 1$. 
Let
	$\bbeta_{\bs}\sim N\left(\bbeta_{0,\bs},g_n \sigma^2_{\bbeta}\left( Z_{\bs}^{\prime} Z_{\bs}\right)^{-1} \right)$, where $g_n=O\left(1\right)$, 
	and $\|\bbeta_{0,\bs}\|_{L_1}=O(p^{\omega_{\bs}})$.
If the space of covariates is compact, and the set of covariates $\left\{x_j:j\in\bS\right\}$ is non-zero,
	then provided that the $\limsup$ and $\liminf$ of $E_{\bs_0}\left[\log BF^n_{\bs,\bs_0}(\rho)/n^{1+2r\omega_{\bs}}\right]$ 
	are upper and lower semicontinuous, respectively,
(\ref{eq:ibf_conv2}) holds.
\end{theorem}

Note that for simplicity we have assumed $\sigma^2_{\epsilon}$ to be known in the proof of Theorem \ref{theorem:ar1}. However, as the following corollary
shows, this is not necessary.
\begin{corollary}
\label{cor:ar1}
	Due to Theorem \ref{theorem:as_conv2}, the result of Theorem \ref{theorem:ar1} continues to hold with $n^{1+2r\omega_{\bs}}$ replaced with $n\log n$ if we set
$\sigma^2_{\beta}=\sigma^2_{\epsilon}$ and assign the conjugate inverse-gamma prior (\ref{eq:invgamma}) to $\sigma^2_{\epsilon}$.
\end{corollary}

\begin{remark}\label{omega}
	Before proceeding further, it is important to understand the role of $\omega_{\bs}$ in the results obtained so far. It is evident that $\omega_{\bs}$ is related to the effective dimensionality of $\mathcal{M}_{\bs}$ and $\mathcal{M}_{\bs_0}^{t}$. When the mean function of the Gaussian process, $\bmu_{\bs}$, is linear (or, a smooth function of the linear combination of covariates in $|\bs|$), and the coefficient of the $j$-th covariate is associated with same prior across different models $\mathcal{M}_{\bs}$ involving it, then $|\bs\Delta \bs_{0}|=p^{\omega_{\bs}}$. We observed this in the linear regression and Gaussian process regression with squared exponential kernel. However, if this simplification is not available, and $\bmu_{\bs}$ is any function satisfying $\|\bmu_{\bs}\|^{2}=O\left(n|\bs|^{2}\right)$, then $\max\{ |\bs|,|\bs_{0}|\}=p^{\omega_{\bs}}$, which is observed in the AR(1) illustration. Finally, if the dimensions of the competing models do not grow with $n$, then $\omega_{\bs}=0$. Although the role of $\omega_{\bs}$ varies
	with the problems setups, existence of an $\omega_{\bs}$ for which (A1)--(A3) hold, is certain. Consequently, strong Bayes factor consistency is achieved at the rate $n^{1+2r\omega_{\bs}}$.  
	\end{remark}

\section{Variable selection using Bayes factors in misspecified situations}
\label{sec:misspecification}
So far we have investigated consistency of the Bayes factor for variable selection when the true model $\mathcal M_{\bs_0}$ is present in the space of models
being compared. However, for a large number of covariates such an assumption need not always be realistic. Indeed, in practice, for a large number available covariates, 
it is usually not feasible to compare all possible models. As the true subset $\bs_0$ is unknown, it is not unlikely to exclude   
it from the set of models being considered for comparison. In such cases of omissions, it makes sense to select the best subset $\bs$ from the available class of subsets
using Bayes factors. Result \ref{theorem:ibf_conv2}, which may be viewed as an adaptation of Theorem \ref{theorem:ibf_conv} for comparing models that are not necessarily
correct, establishes the usefulness of Bayes factors even in the face of such misspecifications.
%, assuming the most general possible setup.  

First consider a simple case. Let $\bs_{1},\bs_{2}\subseteq \bS$ be two competing models of similar order, in the sense that either $\omega_{\bs_{1}},\omega_{\bs_{2}}>0$, or $\omega_{\bs_{1}}=\omega_{\bs_{2}}=0$. The following result holds in this setup.
%Recall that the quantity $\omega$, in all previous applications, depends on the order of the competing models, and hence remains fixed when we consider $\bs_{1}$ and $\bs_{2}$. 
\begin{result}
\label{theorem:ibf_conv2}
 Consider the setup of Section \ref{sec:bf_int} with the error variance $\sigma^2_{\epsilon}$ unknown. Let there exist $\omega_{\bs_{1}},\omega_{\bs_{2}}$, such that ($A1^\prime$)--($A3^\prime$) hold for models $\mathcal{M}_{\bs_{1}}, \mathcal{M}_{\bs_{2}}$, respectively,  for each $\bta\in\bTheta$, where $\bTheta$ is compact.
	Assume that, $g(n)=n$ if $\omega_{\bs_{1}}=\omega_{\bs_{2}}=0$, and $g(n)=n\log(n)$ if $\omega_{\bs_{1}},\omega_{\bs_{2}}\in(0,1]$. Also assume the following:
\begin{enumerate}[(i)]
	\item $\log\left(BF^{n}_{\bs,\bs_0}(\bta)\right)/g(n)$ is stochastically equicontinuous,
	\item $E_{\bs_0}\left[\log\left(BF^{n}_{\bs,\bs_0}(\bta)\right) /g(n) \right]$ is equicontinuous with respect to $\bta$ as $n\rightarrow\infty$, and
	\item The limit of $E_{\bs_0}\left[\log\left(BF^{n}_{\bs,\bs_0}(\bta)\right)/g(n)\right]$ exists and is continuous in $\bta$.
\end{enumerate}
If $\bs_1$ and $\bs_2$ are not equal to $\bs_0$, then there exist $\delta_{\bs_1},\delta_{\bs_2}>0$ associated with models 
$\mathcal M_{\bs_1}$ and $\mathcal M_{\bs_2}$ such that
\begin{equation}
	%\max_{\bs\neq\bs_0}\lim_n \frac{1}{n}\log\left(IBF^{n}_{\bs,\bs_0}\right)\stackrel{a.s.}{=} -\delta.
	\lim_n \frac{1}{g(n)}\log\left(IBF^{n}_{\bs_1,\bs_2}\right)\stackrel{a.s.}{=} -(\delta_{\bs_1}-\delta_{\bs_2}).
\label{eq:ibf_conv3}
\end{equation}
\end{result}
\begin{proof}
Using similar arguments as in the proof of Theorem \ref{theorem:ibf_conv}, under the assumptions (i)--(iii), one can show that for any $\bs$,
\begin{equation}
	\lim_n \frac{1}{g(n)}\log\left(IBF^{n}_{\bs,\bs_0}\right)\stackrel{a.s.}{=} -\delta_{\bs}(\tilde\bta_{\bs}),
\label{eq:ibf_conv4}
\end{equation}
where, due to (iii), $\overline{\delta_{\bs}(\bta)}=\underline{\delta_{\bs}(\bta)}=\delta_{\bs}(\bta)
	=\lim_n E_{\bs_0}\left[\log\left(BF^{n}_{\bs,\bs_0}(\bta)\right)/g(n)\right]$, is continuous for all $\bta\in\bTheta$, and 
	$\tilde\bta_{\bs}\in\bTheta$ such that by the mean value theorem for integrals, 
	\begin{eqnarray*}
		\overline I_n=\int_{\bTheta}\exp\left(-g(n)\overline{\delta_{\bs}(\bta)}\right)\pi(\bta)d\bta=\exp\left(-g(n)\delta_{\bs}(\tilde\bta_{\bs})\right) \\
		=\int_{\bTheta}\exp\left(-g(n)\underline{\delta_{\bs}(\bta)}\right)\pi(\bta)d\bta=\underline I_n.
		\end{eqnarray*}
Noting that
\begin{equation}
	\frac{1}{g(n)}\log\left(IBF^{n}_{\bs_1,\bs_2}\right)=\frac{1}{g(n)}\log\left(IBF^{n}_{\bs_1,\bs_0}\right)-\frac{1}{g(n)}\log\left(IBF^{n}_{\bs_2,\bs_0}\right),
\label{eq:ibf_conv5}
\end{equation}
the proof is completed by taking limits of both sides of (\ref{eq:ibf_conv5}), applying (\ref{eq:ibf_conv4}) on the two terms on the right hand side, and
denoting $\delta_{\bs}(\tilde\bta_{\bs})$ by $\delta_{\bs}$ for all $\bs$.\qed
\end{proof}

\begin{remark}
	From Result \ref{theorem:ibf_conv2} it follows that $\mathcal M_{\bs_1}$ is the better model than $\mathcal M_{\bs_2}$ if $\delta_{\bs_1}<\delta_{\bs_2}$
	and $\mathcal M_{\bs_2}$ is to be preferred over $\mathcal M_{\bs_1}$ if $\delta_{\bs_1}>\delta_{\bs_2}$. The Bayes factor converges exponentially fast
	to infinity and zero, respectively, in these cases.
	Hence, asymptotically with respect to the Bayes factor, the best subset $\bs$ is the one that minimizes $\delta_{\bs}$.
\end{remark}

\begin{remark}
Let $\omega_{\bs_{1}}=0$ and $\omega_{\bs_{2}}>0$. In this case, it is evident that $\mathcal{M}_{\bs_{1}}$ is closer to $\mathcal{M}_{\bs_{0}}$ than $\mathcal{M}_{\bs_{2}}$, in the sense that, either $|\bs_{0}\Delta\bs_{1}|/|\bs_{0}\Delta\bs_{2}|\rightarrow 0$, or  
	$\max\{|\bs_{1}|,|\bs_{0}|\}/\max\{|\bs_{2}|,|\bs_{0}|\}\rightarrow 0$	
	%$|\bs_{1}|\vee|\bs_{0}|/ |\bs_{2}|\vee |\bs_{0}|\rightarrow 0$ 
	(see Remark \ref{omega}), i.e., $\mathcal{M}_{\bs_{2}}$ has significantly large number of covariates than $\mathcal{M}_{\bs_{0}}$, 
	compared to $\mathcal{M}_{\bs_{1}}$. % where $a\vee b=\max\{a,b\}$. 
	Taking $g(n)=n\log(n)$, and following the steps of Result \ref{theorem:ibf_conv2}, one can show that  
$$ \lim_n \frac{1}{g(n)}\log\left(IBF^{n}_{\bs_2,\bs_1}\right)\stackrel{a.s.}{=}-\delta_{\bs_2}.$$
Thus, the Bayes factor favors $\mathcal{M}_{\bs_{1}}$ over $\mathcal{M}_{\bs_{2}}$, and  Bayes factor converges to $0$ at an exponentially fast rate.
	\end{remark}

%where	$\tilde\bta_{\bs}\in\bTheta$ such that by the mean value theorem for integrals, 
%	$$\overline I_n=\int_{\bTheta}\exp\left(-n\overline{\delta_{\bs}(\bta)}\right)\pi(\bta)d\bta=\exp\left(-n\delta_{\bs}(\tilde\bta_{\bs})\right)
%	=\int_{\bTheta}\exp\left(-n\underline{\delta_{\bs}(\bta)}\right)\pi(\bta)d\bta=\underline I_n.$$ 
%Observe that, due to (iii), $\overline{\delta_{\bs}(\bta)}=\underline{\delta_{\bs}(\bta)}=\delta_{\bs}(\bta)=\lim_n E_{\bs_0}\left[\log\left(BF^{n}_{\bs,\bs_0}(\bta)\right)/n\right]$ is continuous for all $\bta\in\bTheta$.

%Noting that
%\begin{equation}
%\frac{1}{n}\log\left(IBF^{n}_{\bs_1,\bs_2}\right)=\frac{1}{n}\log\left(IBF^{n}_{\bs_1,\bs_0}\right)-\frac{1}{n}\log\left(IBF^{n}_{\bs_2,\bs_0}\right),
%\label{eq:ibf_conv5}
%\end{equation}
%the proof is completed by taking limits of both sides of (\ref{eq:ibf_conv5}), applying (\ref{eq:ibf_conv4}) on the two terms on the right hand side, and
%denoting $\delta_{\bs}(\tilde\bta_{\bs})$ by $\delta_{\bs}$ for all $\bs$.

\section{An overview of our simulation and real data experiments}
\label{sec:overview_sim}
We consider two sets of simulation experiments. In the first set, we provide direct validation of our theoretical results by fixing a true set of covariates
and comparing it with specifically chosen incorrect sets of covariates
% (constructed by changing some variables of the true covariate set or by adding more covariates to it), 
using Bayes factor as the sample size is increased. We demonstrate the validity
of our results in the linear regression, Gaussian process regression, as well as in the AR(1) regression context.

In the second simulation scenario, our goal is to identify, using Bayes factors, the true set of data-generating covariates from amongst the set of $2^p-1$ available
subsets of covariates, given any value of $p$ and $n$. 
To this end, we devise a novel and efficient variable-dimensional MCMC algorithm for general-purpose variable selection using Bayes factors, in the framework of 
Transdimensional Transformation based Markov Chain Monte Carlo (TTMCMC) introduced by \ctn{Das19}. 

Not only do we demonstrate the effectiveness of our strategy
with simulation studies involving linear, Gaussian process and AR(1) regressions, but also very successfully apply our procedure to the variable selection problem
in a real riboflavin data consisting of $p=4088$ covariates and $n=71$ data points, using both linear and Gaussian process regression.

\subsection{A briefing on our simulation studies for direct theory validation}
\label{subsec:theorey_validation}
%Validate the three theoretical illustrations, namely, linear, Gaussian process and AR(1) regression, using simulated data. 
In this section $\sigma^2$ is assumed to be unknown, and is assigned an Inverse-Gamma$(1,1)$ prior. The covariates are generated from scaled 
$t_{(3)}$ distribution, with an AR(1) structured scale matrix $\Sigma_{0}$, where $\rho$ varies from $0.1$--$0.25$.  
The total number of covariates $p$ is fixed at $100$, where $n$ varies from $150$ to $600$. 
Three choices of $|\bs_{0}|$ are taken, viz. $|\bs_{0}|=10,40,70$.

As per our result, we expect the Bayes factor of the true model against any other model to converge to zero as $n\rightarrow\infty$. 
We pre-select two competing models which are closest to the true model, in appropriate sense. First, a supermodel having $k$ additional covariates, is considered. 
Second, we choose a model which has the same cardinality as the true model, and exactly $k$ variables are different from the true model. 
%i.e., $(|\bs\smallsetminus\bs_{0}|=|\bs_{0}\smallsetminus\bs|=t)$. 
For illustration 1 (linear model) and 3 (AR(1) model) we choose $k=1$, and for illustration 2 (GP with squared exponential kernel), 
we choose $k=5$. We fix the true $\sigma^2$ at $1$. 

We also consider the case for misspecified models in linear regression and GP regression framework. In both the cases we consider two supermodels of the true model, $\mathcal{M}_{\bs_{1}}$ and $\mathcal{M}_{\bs_{2}}$, having $k_{1}$ and $k_{2}$ extra covariates, and $\bs_{1}\subset \bs_{2}$. Clearly, $\mathcal{M}_{\bs_{1}}$ is closer to the true model than $\mathcal{M}_{\bs_{2}}$.
The simulation set up is kept the same as before. For linear regression we choose $k_{1}=1, ~k_{2}=5$, and for GP regression we choose $k_{1}=5,~k_{2}=15$.

Finally, for each pair $(p,n)$ and each example, data-generation procedure is repeated 100 times to reduce randomness, and the mean Bayes factor is reported.
Very encouraging results are obtained with our strategies in each of the regression scenarios considered.
For misspecified models, it is clearly observed that Bayes factor chooses the better model, i.e., $\mathcal{M}_{\bs_{1}}$, at a growing rate with $n$. 
The complete details are provided in Section \ref{sec:9.1}.

\subsection{Simulation experiments with Bayes factor oriented TTMCMC}
\label{subsec:ttmcmc}

Although a plethora of methods are available for Bayesian variable selection (see, for example, \ctn{OHara_2009} for a review),
including variable-dimensional solutions in the linear and generalized 
linear regression contexts
% have already been
%considered in the literature using reversible jump Markov Chain Monte Carlo (RJMCMC) methods 
(see, for example, \ctn{Sill98}, \ctn{Lunn06}, \ctn{Sill04}, \ctn{Chevalier20}),
%However,
 implementation of variable selection in the nonparametric Gaussian process regression setup, to the best of our knowledge, is nonexistent in the literature. Therefore, it is imperative to develop new methodologies for practical variable selection implementation in this framework.

 Note that when the available number of covariates is even reasonably large, evaluation of the marginal density of the data needed for Bayes factor, even if available in closed form, is infeasible to compute for all possible covariate subsets. Thus, direct comparison of all possible covariate subsets with respect to the marginal density of the data is generally infeasible, and hence suitable MCMC approaches are necessary.
% The key idea is to consider the basic definition of Bayes factor as the ratio of posterior odds to prior odds of the competing models (here, subsets of covariates),
 %and estimate the posterior probabilities of the competing models (covariate subsets) using MCMC realizations associated with the posteriors of the models (covariate subsets). Since the priors for the models are known, this would then enable numerical evaluation of the Bayes factors. 
% The models that do not appear in the MCMC sample, will
 %be considered to have negligible posterior probabilities and need not be further considered as competitors. In other words, Bayes factor based model comparison 
% will be only among models (subsets of covariates) that feature in the MCMC sample. 
 
 The traditional MCMC approaches are not valid in the model selection scenario. Indeed, different competing models may consist of sets of parameters with varying
 cardinalities, which would render the fixed-dimensional MCMC methods invalid. In the variable selection setup, at least the regression coefficients of the 
 competing models associated with different subsets of covariates, are variable-dimensional. 
 %In Gaussian process regression, the smoothness parameters, which are
 %also associated with the covariate subsets, are variable-dimensional. 
 Thus, variable-dimensional MCMC methods are necessary to handle the Bayesian model selection paradigm.
 Although reversible jump MCMC (RJMCMC) \citep{Green95} is a valid model-jumping MCMC method, its effectiveness with respect to practical implementation is often very doubtful,
 with poor mixing properties being the integral part. Thus, considerably more innovative and effective variable-dimensional MCMC procedures are necessary to meet the
 challenges of complex variable-dimensional problems, such as model selection, among many others.

 As such, we shall offer a generic and effective variable-dimensional, Bayes factor oriented solution to any variable selection problem. 
 We employ the novel TTMCMC methodology of \ctn{Das19}
 for general variable-dimensional problems, which is a generalization of the fixed-dimensional
 Transformation based Markov Chain Monte Carlo (TMCMC) of \ctn{Dutta14}. The most important feature of TMCMC is facilitation of updating all the variables
 in question simultaneously using appropriate deterministic transformations of even a singleton random variable. 
 This general strategy leads to remarkable improvement of acceptance rates and mixing properties, even in high dimensions.
 %, by effectively drastically reducing the dimension by the low-dimensional deterministic transformation. 
 These key features are inherited by TTMCMC in the transdimensional context. 
 %Interestingly, the acceptance ratios of neither TMCMC, nor TTMCMC,
 %depend upon the proposal distributions, even if they are not symmetric, and even for dimension-changing moves. Thus, these approaches are novel compared to the traditional
 %fixed-dimensional Metropolis-Hastings and the variable-dimensional RJMCMC approach.
 
 Here we devise a novel TTMCMC algorithm for generic variable selection problems using mixtures of additive and multiplicative transformations of singleton variables,
 further supplemented with another deterministic transformation step to enhance mixing. The algorithm is available as Algorithm \ref{algo:ttmcmc} 
 in Section \ref{sec:generic_ttmcmc} of the 
 supplement. %Supplementary material (SM).
 An important aspect of the algorithm is to propose a new covariate in the
 ``birth move" by Bayes Information Criterion (BIC), given a set of existing covariates. 
 %We compute Bayes factors from the available TTMCMC realizations to compare subsets of the covariates. 
 The method of computation of Bayes factors using TTMCMC samples is detailed in Section \ref{sec:bf_comp} of the supplement.
 In Section \ref{sec:proof_conv} of the supplement we provide the proof of its convergence.
 
 The proposed TTMCMC strategy leads to quite effective variable selection, while exhibiting good mixing properties.   
 We demonstrate this with simulation experiments in linear regression, Gaussian process regression and time series
 regression setups (see Section \ref{sec:simstudy} of the supplement).

\subsection{Overview of our real data experiment}

For real data application of our Bayes factor oriented variable selection procedure, we consider a dataset on riboflavin (vitamin $B_2$) production rate, where
the response variable is the log-transformed riboflavin production rate and the covariates are the logarithms of $4088$ gene expression levels. There are
only $n=71$ data points in the data  (thus, a {\it bona fide} real example of the ``large $p$, small $n$" setup). This data, made publicly available by \ctn{Buhl14},
has been analyzed by various research groups using traditional classical methods in the linear regression framework.  
We model this data as linear regression, as well as Gaussian process regression, and using our Bayes factor based covariate selection, obtain
very interesting and insightful results as compared to the existing results (see Section \ref{sec:realdata} of the supplement).

\section{Summary, conclusion and future direction}
\label{sec:conclusion}
This work is an effort to establish an asymptotic theory of variable selection using Bayes factor in a general Gaussian process framework that encompasses
linear, nonlinear, parametric, nonparametric, independent, as well as dependent setups involving a set of covariates, the size of which is allowed to increase
even at much faster rates than the sample size. The setup also includes the special case where the available number of covariates is considered fixed.   
That even in such a general setup it has been possible to establish almost sure exponential convergence of the Bayes factor in favour of the correct subset of covariates,
seems to be quite encouraging. The illustrations in the case of linear regression, Gaussian process model with squared exponential covariance function containing the
covariates, and a first order autoregressive model with time-varying covariates, vindicate the wide applicability of our asymptotic theory. Besides, it has been possible
to adapt our main results on Bayes factor consistency to misspecified cases, where the true set of covariates is not included in the subsets of covariates to be compared
using Bayes factor. As already explained, misspecification has high likelihood in practice, and from this perspective, the result on almost sure exponential convergence
even for misspecifications, seems to be a pleasant one. Recalling  the predominance of linear or additive model based Bayes factor asymptotics, 
and ``in probability" convergence of the Bayes factor, our efforts in this work attempt to provide a significant advancement. 
% and its various versions involving additivity of functions of the regressors

Furthermore, we have conducted ample simulation experiments to supplement our theoretical investigations. Indeed, not only have we provided direct validation of our theoretical
results; with an eye to variable selection in practical problems, we have devised a generic Bayes factor oriented TTMCMC algorithm for such purpose, demonstrating its efficacy
in detecting the true set of covariates from among a very large pool (size $2^p-1$) of available subsets of covariates, in linear, Gaussian process and AR(1) regression
setups. Our TTMCMC strategy also yielded very interesting (and perhaps quite important) variable selection results in the case of a real riboflavin dataset
consisting of $4088$ covariates and $71$ samples, exemplifying an authentic ``large $p$, small $n$" real-life scenario.

It is easy to discern that our results and the methods of our proofs can be generalized without substantial modifications to situations where parts of the models
are also necessary to select from among a set of possibilities, besides the best set of covariates. For example, in our linear regression example, choice might be necessary
between linear and some specified nonlinear regression functions which also encapsulate the covariates in appropriate forms. In our Gaussian process example with
squared exponential covariance function, the form of the covariance function may itself be questionable, and needs to be chosen from a set of plausible covariance forms,
associated with various stationary and nonstationary Gaussian processes. In the first order autoregressive model example, the order of the autoregression may itself need
to be selected. Our primary calculations confirm that our Bayes factor asymptotics admit extension to simultaneous selection of these model parts and the covariates,
with additional mild assumptions. These findings, with details, will be communicated elsewhere.

\section*{Acknowledgment}
We are sincerely grateful to the Associate Editor and the two referees whose comments have led to significant improvement of our article.

\newpage

\renewcommand\thefigure{S-\arabic{figure}}
\renewcommand\thetable{S-\arabic{table}}
\renewcommand\thesection{S-\arabic{section}}
\renewcommand\theequation{S-\arabic{equation}}
\renewcommand\theresult{S-\arabic{result}}
\renewcommand\thelemma{S-\arabic{lemma}}

\setcounter{figure}{0}
\setcounter{table}{0}
\setcounter{section}{0}
\setcounter{equation}{0}
\setcounter{result}{0}
\setcounter{lemma}{0}

\begin{center}
	        {\bf{\LARGE{Supplementary Material}}}
\end{center}

This document is an addendum to the theory developed in the main manuscript (MB). 
This supplementary material is organized as follows.

%Section \ref{sec:SE} is an appendix to the Section \ref{sec:NE} of MB.
In Section \ref{sec:9.1}, we numerically validate the results of MB, in the contexts of linear regression (LR), Gaussian process regression (GPR)
and AR(1) process regression (AR-1). 

Next we consider the problem of Bayes factor based variable selection from among $2^p-1$ available covariates. 
In this regard, in Section \ref{sec:generic_ttmcmc} we introduce our TTMCMC sampler 
for general Bayesian variable selection problems. The method of computation of Bayes factors using TTMCMC samples is detailed in Section \ref{sec:bf_comp}.
In Section \ref{sec:proof_conv} we provide the proof of convergence of our TTMCMC sampler.

In Section \ref{sec:simstudy} we provide the details of our TTMCMC based variable selection experiments in the contexts of LR, GPR, and AR-1.  

In Section \ref{sec:realdata} we address variable selection among a set of $4088$ covariates in a real, riboflavin dataset, using our Bayes factor oriented TTMCMC 
methodology, considering both linear and Gaussian process regression, and obtain interesting insights with respect to existing results on variable selection
in the same dataset obtained using linear regression and classical methods.

Finally, in Section \ref{sec:appendix}, we provide the proof of the lemmas and results stated in the MB.

\section{Direct validation of the theoretical results using simulation experiments}\label{sec:9.1}
%The three illustrations of the paper are validated using simulated data. Throughout this section, $\sigma^2$ is assumed to be unknown, and is assigned a $\mbox{Inv-Gamma}(1,1)$ prior. The covariates are generated from scaled $t_{(3)}$ distribution, with an AR(1) structured scale matrix $\Sigma_{0}$, where $\rho$ varies from $0.1$-$0.25$.  
%The total number of covariates $p$ is fixed at $100$, where $n$ varies from $150$ to $600$. Three choices of $|\bs_{0}|$ are taken, viz. $|\bs_{0}|=10,40,70$.

%As the theorem states, we expect the Bayes factor of the true model against any other model should converge to zero as $n\rightarrow\infty$. This is true for any wrong model. We pre-select two competing models which are closest to the true model, in appropriate sense. First, a supermodel having $t$ additional covariates. Second, a model ($\bs$) which has same cardinality as the true model, and exactly $t$ variable is different from the true model, i.e., $(|\bs\smallsetminus\bs_{0}|=|\bs_{0}\smallsetminus\bs|=t)$. For illustration 1 (linear model) and 3 (AR(1) model) we choose $t=1$, and for illustration 2 (GP with squared exponential kernel), we choose $t=5$. Finally, we fix $\sigma^2=1$. 
%We repeat the data-generation procedure $100$ times for each pair $(p,n)$ and each example, to reduce randomness, and report the mean Bayes factor.

\subsection{Linear regression} \label{sec:9.1.1}
Here we assume $y_{i}=\bbeta_{\bs_{0}}^{T} \bx_{i,\bs_{0}}+\epsilon_{i}$ where $\epsilon_{i}\stackrel{iid}{\sim} N(0,1)$. As stated above the covariates are generated from scaled $t_{3}$ distribution, where scale matrix $\Sigma_{0}$ is AR(1) structured, with $\rho=0.25$. We assign Zellner's $g$-prior on the regression coefficients $\bbeta_{\bs}$, with $\bbeta_{0,\bs}=\left(1/p, \ldots, 1/p\right)$, $\sigma^{2}_{\bbeta}=1$ and $g=10$. As set of $|\bs_{0}|$ covariates are chosen at random, and the values of the corresponding coefficients are chosen from an $\mathrm{Uniform}(0,1)$ distribution. 

The results are summarized in Figure \ref{fig1}. Note that, the \emph{supermodel} has exactly one extra variable and the \emph{altered model} has exactly one variable different from the true model. Even for such small changes, the Bayes factor identifies the true model efficiently. Further, as the size of the true model increases, Bayes factor becomes more efficient.     
\afterpage{
\begin{figure}
	\begin{center}
		\includegraphics[height=6cm,width=14cm]{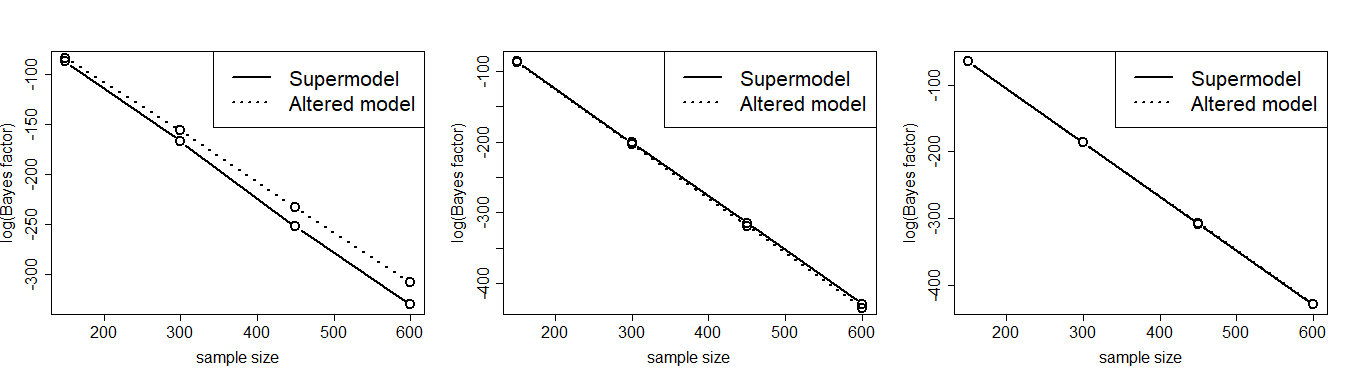}
	\end{center}
	\caption{Line diagram showing how the Bayes factors decrease as sample sizes increase in linear regression when $|\bs_{0}|=10$ (left panel), $|\bs_{0}|=40$ (middle panel), and $|\bs_{0}|=40$ (right panel)}
	\label{fig1}
\end{figure}}

\subsection{Gaussian process with squared exponential kernel}\label{sec:9.1.1.2}
Next we generate data from Gaussian process with squared exponential kernel as given in (\ref{var:rkhs}). We choose 
$D_{\bs}=\mathrm{diag}\{10,\ldots,10\}$ for all $\bs$, and $\sigma_f^{2}=1$. We choose a constant mean function, $\bmu_{\bs}=\mathrm{logistic}\left({\bf x}_{\bs}^{\prime} \bbeta_{\bs}\right)$ for all $\bs$. Note that, the assumptions (A1)-(A3) are satisfied by these choices of the parameters. The coefficients, $\bbeta_{\bs,j}$, are generated randomly from independent Uniform$(-0.5,0.5)$ distributions. 

As before the covariates are generated from scaled $t_{3}$ distribution, where scale matrix $\Sigma_{0}$ is AR(1) structured, with $\rho=0.1$. In this case the \emph{supermodel} has $k=5$ more covariates, and the \emph{altered model} has $k=5$ different covariates than the true model. These covariates are randomly selected from the pool of $p-|\bs_{0}|$ covariates. Figure \ref{fig2} shows the performance of the Bayes factor as $n$ grows. Observe that, unlike both linear regression and AR(1) regression, the Bayes factor detects the true model much faster when some covariates are altered, than a supermodel. 

\afterpage{
\begin{figure}
	\begin{center}
		\includegraphics[height=6cm,width=14cm]{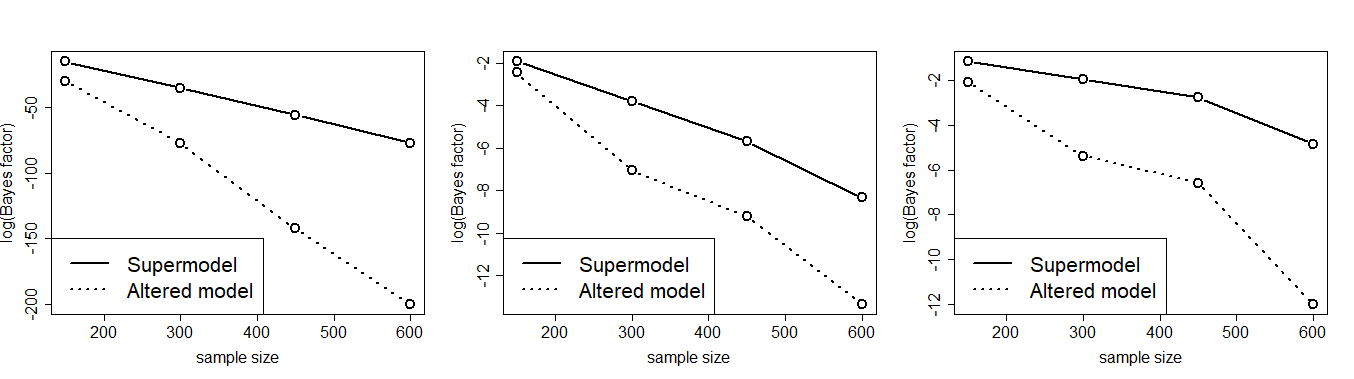}
	\end{center}
	\caption{Line diagram showing how the Bayes factors decrease as sample sizes increase in GP with squared exponential kernel when $|\bs_{0}|=10$ (left panel), $|\bs_{0}|=40$ (middle panel), and $|\bs_{0}|=40$ (right panel)}
	\label{fig2}
\end{figure}}

\subsection{Autoregressive model} The response is now generated from AR(1) model (\ref{eq:ar1_model}) with $\rho=0.25$. The distribution of the covariates, choice of prior on $\bbeta_{\bs}$, and the definition of \emph{supermodel} and \emph{altered model} are same as that in Section \ref{sec:9.1.1}. 

As the true value of $\rho$ is not known, we numerically find the integrated marginal likelihood of the true model and the competing model, considering an $\mathrm{Uniform}(-1,1)$ prior on $\rho$. The integrated Bayes factor is the ratio of the integrated likelihood of the competing and the true model. 

The results are summarized in Figure \ref{fig3}. As in the case of linear model, Bayes factor efficiently captures the true model even when the competing model is the closest one to the truth. 
%However, unlike linear regression, performance of Bayes factor deteriorates as the size of the true model increases.     

\afterpage{
\begin{figure}
	\begin{center}
		\includegraphics[height=6cm,width=14cm]{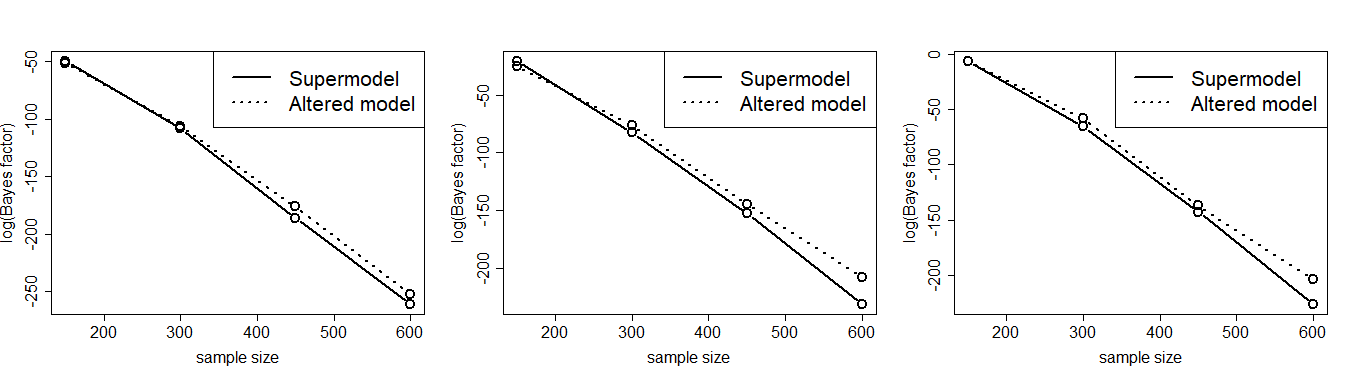}
	\end{center}
	\caption{Line diagram showing how the Bayes factors decrease as sample sizes increase in autoregressive model when $|\bs_{0}|=10$ (left panel), $|\bs_{0}|=40$ (middle panel), and $|\bs_{0}|=40$ (right panel)}
	\label{fig3}
\end{figure}}

\subsection{Misspecified models}
\label{subsec:misspecification}
Now we compare two nested supermodels of the true model, $\mathcal{M}_{\bs_{1}}\subset \mathcal{M}_{\bs_{2}}$, with dimensions $|\bs_{1}|=k_{1}$ and $|\bs_{2}|=k_{2}$, respectively. Clearly, the supermodel with lower dimension, $\mathcal{M}_{\bs_{1}}$, is closer to the true model, and the theory suggests that the Bayes factor $\log BF_{\bs_{2},\bs_{1}}$ decays with growing $n$. 
The linear regression and Gaussian process regression with squared exponential kernel is considered. 

In the linear regression, we choose $k_{1}=1$ and $k_{2}=5$. Everything else is kept same as in Section \ref{sec:9.1.1}, expect here we choose $\bbeta_{0,\bs}=\left(1/2, \ldots, 1/2\right)$.  
In the Gaussian process regression, we choose $k_{1}=5$ and $k_{2}=15$. Everything is kept same as in Section \ref{sec:9.1.1.2}. The results are summarized in Figure \ref{fig4}.

Observe that for both the cases we observe a sharp linear decrease of log Bayes factors as $n$ increases, which validates our theoretical results. 

%\afterpage{
\begin{figure}
	\begin{center}
		\includegraphics[height=12cm,width=15cm]{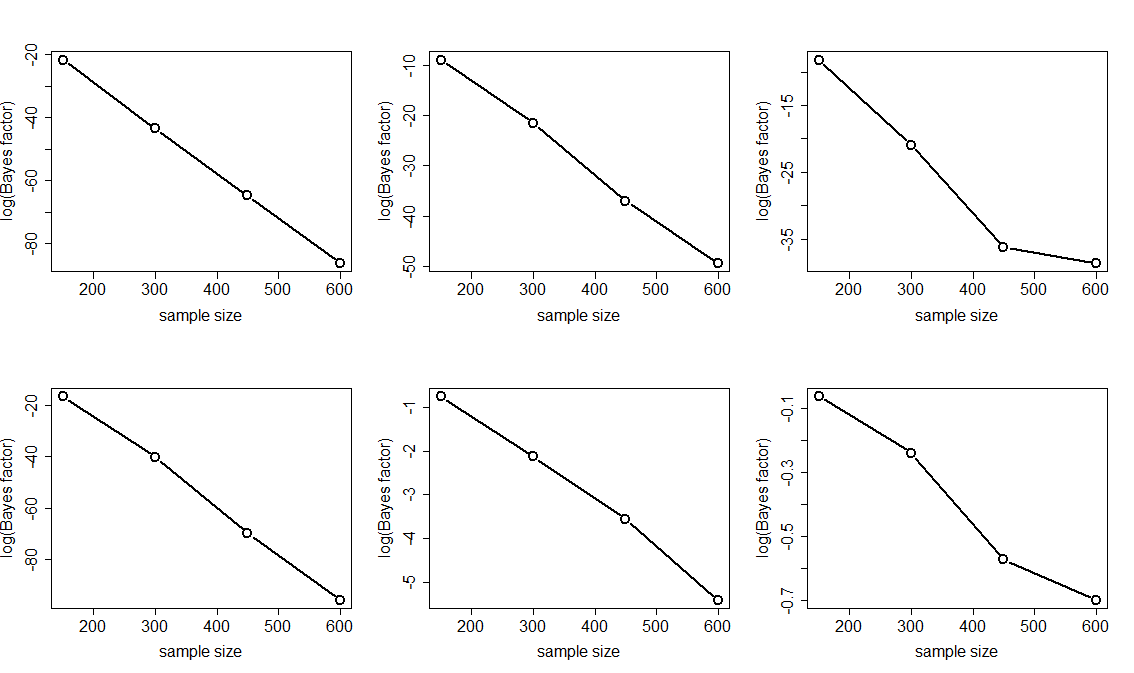}
	\end{center}
	\caption{Line diagram showing how the Bayes factor favors the better model as sample sizes increase in the misspecified models setup, in linear regression (top row), and Gaussian process regression (bottom row), when $|\bs_{0}|=10$ (left panel), $|\bs_{0}|=40$ (middle panel), and $|\bs_{0}|=70$ (right panel).}
	\label{fig4}
\end{figure}%}

\section{A generic TTMCMC sampler for variable selection}
\label{sec:generic_ttmcmc}

Here we devise a novel TTMCMC algorithm for generic variable selection problems using mixtures of additive and multiplicative transformations of singleton variables,
further supplementing with a deterministic transformation step to enhance mixing. Given a set of existing covariates, we propose a new covariate in the
``birth move" by Bayes Information Criterion (BIC). We compute Bayes factors from the available TTMCMC realizations to compare subsets of the covariates. 
Interestingly, the acceptance ratios of neither TMCMC, nor TTMCMC,
depend upon the proposal distributions, even if they are not symmetric, and even for dimension-changing moves. Thus, these approaches are novel compared to the traditional
fixed-dimensional Metropolis-Hastings and the variable-dimensional RJMCMC approach.

We provide our general TTMCMC sampler for variable selection in the form of Algorithm \ref{algo:ttmcmc}.
We assume that $\bta=(\bbeta,\bvartheta)$ is the set of parameters associated with the model, $\bbeta$ being the $k$-dimensional  
regression coefficients associated with the chosen covariates, where $k$ is a random variable.
The parameter vector $\bvartheta$ consists of other sets of parameters, and may even contain several other parameter vectors associated with the covariates, having the
same (variable) dimension $k$ as $\bbeta$. For instance, in a Gaussian process regression, the mean function may be modeled by a linear regression
with regression coefficients $\bbeta$ and the covariance function may be modeled by a squared exponential kernel consisting of smoothness parameters having the
same random dimension $k$ as $\bbeta$. We shall denote by $\pi(\bta,\bs,k)$ as proportional to the product of the prior and the likelihood, where 
$\bs$ and $k$, the random subset of covariate indices and its cardinality, are also considered unknown and suitable priors are envisaged for the same. 
Thus, with abuse of notation for convenience and simplicity, we 
write the posterior $\pi(\bta,k)$ as
\begin{equation}
	\pi(\bta,\bs,k)\propto L(\bta|\bs,k)\pi(\bta|\bs,k)\pi(\bs|k)\pi(k), 
	\label{eq:post2}
\end{equation}
where $\pi(k)$ denotes the prior for $k$, $\pi(\bta|\bs,k)$ stands for the prior for $\bta$ given $\bs$ and $k$, $\pi(\bs|k)$ is the prior for
$\bs $ given $k$ and $L(\bta|\bs, k)$ is the likelihood for $\bta$ given $\bs$ and $k$. 
Given $k$, we set the uniform prior for $\bs$: 
\begin{equation}
	\pi(\bs|k)=\frac{1}{{p\choose k}},~\mbox{for}~k=1,\ldots,p.
	\label{eq:prior_s}
\end{equation}

\begin{algo}\label{algo:ttmcmc} \topline General TTMCMC algorithm for variable selection.
\botline \normalfont \ttfamily
\begin{itemize}
	\item Let the initial value be $\bta^{(0)}=(\bbeta^{(0)},\bvartheta^{(0)})$, where $\bbeta^{(0)}\in \mathbb R^{k^{(0)}}$, are the coefficients of the $k^{(0)}$ covariates
		in the current regression model, and $\bvartheta^{(0)}$ consists of the initial values of the other model parameters, which
		may even include other $k^{(0)}$-dimensional parameters associated with the covariates in the model. 
		Also let $\bs^{(0)}$ denote the initial choice for the subset of indices for the covariates associated with the model. 
 \item For $t=0,1,2,\ldots$
\begin{enumerate}
	\item Generate $u=(u_1,u_2,u_3)\sim Multinomial (1;w_{b,k^{(t)}},w_{d,k^{(t)}},w_{nc,k^{(t)}})$, where $w_{b,k^{(t)}},w_{d,k^{(t)}},w_{nc,k^{(t)}}$ are birth, death
		and no-change probabilities, given $k^{(t)}$. Hence, 
		$w_{b,k^{(t)}},w_{d,k^{(t)}},w_{nc,k^{(t)}}$ are non-negative and
		$w_{b,k^{(t)}}+w_{d,k^{(t)}}+w_{nc,k^{(t)}}=1$. Also, $w_{b,k^{(t)}}=0$ if $k^{(t)}=|\bS|$ and $w_{d,k^{(t)}}=0$ if $k^{(t)}=1$. 
 \item If $u_1=1$ (increase dimension by selecting a new covariate), generate $U\sim U(0,1)$ and do the following: 
	 \begin{enumerate}
		 \item  If $U\leq \tilde p$, where $\tilde p\in[0,1]$ (use additive transformation for dimension change),
 \begin{enumerate}
 \item Given $\bs^{(t)}$, the current subset of covariates and the current set of parameters $\bta^{(t)}$, select a new covariate $\{x_{ir}:i=1,\ldots,n\}$, where 
 $r\in\bS\backslash\bs^{(t)}$, by minimizing $BIC(u)$, for $u\in\bS\backslash\bs^{(t)}$. Here $BIC(u)$ stands for the BIC when the model consists
		 of the covariates indexed by $\{\bs^{(t)},u\}$. Let $\bs'=\{\bs^{(t)},r\}$.
 \item Randomly select a co-ordinate from $\bbeta^{(t)}=(\beta^{(t)}_1,\ldots,\beta^{(t)}_{k^{(t)}})$ assuming uniform 
	 probability $1/k^{(t)}$ for each co-ordinate.
 Let $j$ denote the chosen co-ordinate.
 \item Generate $\e_1 \sim N(0,1)$ and 
%	 for $i=1,\ldots,k;~i\neq j$ simulate  
% $$z_i \sim Multinomial(1;p_i,q_i,1-p_i-q_i)$$ independently. 
 %For $i=1,\ldots,k$, the probabilities $p_i$ and $q_i$ need to be specified;
 %see \ctn{Dutta14} for a detailed discussion regarding these choices.
% \item Propose 
 propose the following birth move: 
 $$ \bbeta'=(\beta^{(t)}_1,\ldots,\beta^{(t)}_{j-1},\beta^{(t)}_j+a_{\beta,j}|\e_1|,\beta^{(t)}_{j}-a_{\beta,j}|\e_1|,\beta^{(t)}_{j+1},\ldots,\beta^{(t)}_{k^{(t)}}).$$	 
 Here $a_{\beta,j}$ is the appropriate positive scaling constant associated with the $j$-th co-ordinate of $\bbeta$. In general, $a_{\theta,j}$ will
 stand for the appropriate positive scaling constant associated with the $j$-th co-ordinate of $\bta$.
% \begin{align} 
% \bm x' &= T_{b,\bz}(\supr{\bm x}{t}, \e)=(g_{1,z_1}(x^{(t)}_1,\e),\ldots,g_{j-1,z_{j-1}}(x^{(t)}_{j-1},\e),\notag\\
%& g_{j,{z_j=1}}(x^{(t)}_j,\e),g_{j,{z^c_j=-1}}(x^{(t)}_j,\e),g_{j+1,z_{j+1}}(x^{(t)}_{j+1},\e),\ldots,
%g_{k,z_k}(x^{(t)}_k,\e)).\notag
%\end{align}
%Re-label the elements of $\bm x'$ as $(x'_1,x'_2,\ldots,x'_{k+1})$.
	 \item Re-label the elements of $\bbeta'$ as $(\beta'_1,\beta'_2,\ldots,\beta'_{k^{(t)}+1})$.
 \begin{enumerate}
 \item If there is another set of real-valued variable-dimensional parameters, say, $\bgamma$, associated with the covariates, then also generate 
 $\e_2 \sim N(0,1)$ and propose
 $$ \bgamma'=(\gamma^{(t)}_1,\ldots,\gamma^{(t)}_{j-1},\gamma^{(t)}_j+a_{\gamma,j}|\e_2|,\gamma^{(t)}_{j}-a_{\gamma,j}|\e_2|,\gamma^{(t)}_{j+1},\ldots,\gamma^{(t)}_{k^{(t)}}).$$	 
 \item Re-label the elements of $\bgamma'$ as $(\gamma'_1,\gamma'_2,\ldots,\gamma'_{k^{(t)}+1})$.
 \item Repeat the procedure for further sets of variable-dimensional parameters related to the covariates.
 \item Keep all other elements of $\bta$ unchanged, and refer to the entire set of proposed parameter values as $\bta'$.
 \end{enumerate}
\item If $\bbeta$ is the only variable-dimensional parameter related to the covariates, then the acceptance probability of the birth move is:
 \begin{align}
 a_b &= 
	 \min\left\{1, \frac{1}{k^{(t)}+1}\times\frac{w_{d,k^{(t)}+1}}{w_{b,k^{(t)}}} 
	 ~\dfrac{\pi\left(\bta',\bs',k^{(t)}+1\right)}{\pi\left(\bta^{(t)},\bs^{(t)},k^{(t)}\right)}\times 2a_{\beta,j}\right\}.\notag 
% ~\left|\frac{\partial (T_{b,\bz}(\supr{\bm x}{t}, \e))}{\partial(\supr{\bm x}{t}, \e)}\right| \right\},\notag
 \end{align}
\begin{enumerate}
	\item If $\bgamma$ is another real-valued variable-dimensional parameter related to the covariates, then the acceptance probability of the birth move is:
 \begin{align}
 a_b &= 
	 \min\left\{1, \frac{1}{k^{(t)}+1}\times\frac{w_{d,k^{(t)}+1}}{w_{b,k^{(t)}}} 
	 ~\dfrac{\pi\left(\bta',\bs',k^{(t)}+1\right)}{\pi\left(\bta^{(t)},\bs^{(t)},k^{(t)}\right)}\times 2a_{\beta,j}\times 2a_{\gamma,j}\right\},\notag 
% ~\left|\frac{\partial (T_{b,\bz}(\supr{\bm x}{t}, \e))}{\partial(\supr{\bm x}{t}, \e)}\right| \right\},\notag
 \end{align}
that is, $2a_{\gamma,j}$ must also be multiplied to the acceptance ratio.
\item For further real-valued variable-dimensional parameter associated with the covariates, the process must be continued by further multiplying 
	twice the scaling constant of the relevant parameter to the acceptance ratio.
\end{enumerate}
%\begin{align}
% a_b(\supr{\bm x}{t}, \e) &= 
% \min\left\{1, \frac{1}{k+1}\times\frac{w_{d,k+1}}{w_{b,k}}\times\dfrac{P_{(j)}(\bz^c)}{P_{(j)}(\bz)} 
% ~\dfrac{\pi(\bm x')}{\pi(\supr{\bm x}{t})} 
% ~\left|\frac{\partial (T_{b,\bz}(\supr{\bm x}{t}, \e))}{\partial(\supr{\bm x}{t}, \e)}\right| \right\},\notag
% \end{align}
% where
% \[ 
%P_{(j)}(\bz)=\prod_{i\neq j=1}^kp^{I_{\{1\}}(z_i)}_iq^{I_{\{-1\}}(z_i)}_i,
% \]
% and
% \[ 
%P_{(j)}(\bz^c)=\prod_{i\neq j=1}^kp^{I_{\{1\}(z^c_i)}}_iq^{I_{\{-1\}}(z^c_i)}_i.
% \]

\item Set \[ (\bta^{(t+1)},\bs^{(t+1)},k^{(t+1)})= \left\{\begin{array}{ccc}
		(\bta',\bs',k^{(t)}+1) & \mbox{ with probability } & a_b \\
	(\bta^{(t)},\bs^{(t)},k^{(t)}) & \mbox{ with probability } & 1 - a_b.
\end{array}\right.\]
\end{enumerate}
	 \end{enumerate}

	 \begin{enumerate}
		 \item[(b)]  If $U> \tilde p$ (use multiplicative transformation for dimension change),
 \begin{enumerate}
 \item Given $\bs^{(t)}$, the current subset of covariates and the current set of parameters $\bta^{(t)}$, select a new covariate $\{x_{ir}:i=1,\ldots,n\}$, where 
 $r\in\bS\backslash\bs^{(t)}$, by minimizing $BIC(u)$, for $u\in\bS\backslash\bs^{(t)}$. 
		 Let $\bs'=\{\bs^{(t)},r\}$.
 % Here $BIC(u)$ stands for the BIC when the model consists of the covariates indexed by $\{\bs^{(t)},u\}$.
 \item Randomly select a co-ordinate from $\bbeta^{(t)}=(\beta^{(t)}_1,\ldots,\beta^{(t)}_{k^{(t)}})$ assuming uniform 
	 probability $1/k^{(t)}$ for each co-ordinate.
 Let $j$ denote the chosen co-ordinate.
 \item Generate $\e_1 \sim U(-1,1)$ and 
%	 for $i=1,\ldots,k;~i\neq j$ simulate  
% $$z_i \sim Multinomial(1;p_i,q_i,1-p_i-q_i)$$ independently. 
 %For $i=1,\ldots,k$, the probabilities $p_i$ and $q_i$ need to be specified;
 %see \ctn{Dutta14} for a detailed discussion regarding these choices.
% \item Propose 
	 propose the following birth move: 
		 $$ \bbeta'=(\beta^{(t)}_1,\ldots,\beta^{(t)}_{j-1},\beta^{(t)}_j\e_1,\beta^{(t)}_{j}/\e_1,\beta^{(t)}_{j+1},\ldots,\beta^{(t)}_{k^{(t)}}).$$	 
% \begin{align} 
% \bm x' &= T_{b,\bz}(\supr{\bm x}{t}, \e)=(g_{1,z_1}(x^{(t)}_1,\e),\ldots,g_{j-1,z_{j-1}}(x^{(t)}_{j-1},\e),\notag\\
%& g_{j,{z_j=1}}(x^{(t)}_j,\e),g_{j,{z^c_j=-1}}(x^{(t)}_j,\e),g_{j+1,z_{j+1}}(x^{(t)}_{j+1},\e),\ldots,
%g_{k,z_k}(x^{(t)}_k,\e)).\notag
%\end{align}
%Re-label the elements of $\bm x'$ as $(x'_1,x'_2,\ldots,x'_{k+1})$.
	 \item Re-label the elements of $\bbeta'$ as $(\beta'_1,\beta'_2,\ldots,\beta'_{k^{(t)}+1})$.
 \begin{enumerate}
	 \item If there is another set of real-valued variable-dimensional parameters, say, $\bgamma$, associated with the covariates, 
		 then also generate $\e_2 \sim U(-1,1)$ and propose
		 $$ \bgamma'=(\gamma^{(t)}_1,\ldots,\gamma^{(t)}_{j-1},\gamma^{(t)}_j\e_2,\gamma^{(t)}_{j}/\e_2,\gamma^{(t)}_{j+1},\ldots,\gamma^{(t)}_{k^{(t)}}).$$	 
	 \item Re-label the elements of $\bgamma'$ as $(\gamma'_1,\gamma'_2,\ldots,\gamma'_{k^{(t)}+1})$.
 \item Repeat the procedure for further sets of variable-dimensional parameters related to the covariates.
 \item Keep all other elements of $\bta$ unchanged, and refer to the entire set of proposed parameter values as $\bta'$.
 \end{enumerate}
\item If $\bbeta$ is the only variable-dimensional parameter related to the covariates, then the acceptance probability of the birth move is:
 \begin{align}
 a_b &= 
	 \min\left\{1, \frac{1}{k^{(t)}+1}\times\frac{w_{d,k^{(t)}+1}}{w_{b,k^{(t)}}}
	 %\times\frac{1}{2}\times \dfrac{\pi\left(\bta',\bs',k^{(t)}+1\right)}{\pi\left(\bta^{(t)},\bs^{(t)},k^{(t)}\right)}\times \frac{2|\beta^{(t)}_j|}{|\e_1|}\right\}.\notag 
	 \times \dfrac{\pi\left(\bta',\bs',k^{(t)}+1\right)}{\pi\left(\bta^{(t)},\bs^{(t)},k^{(t)}\right)}\times \frac{|\beta^{(t)}_j|}{|\e_1|}\right\}.\notag 
% ~\left|\frac{\partial (T_{b,\bz}(\supr{\bm x}{t}, \e))}{\partial(\supr{\bm x}{t}, \e)}\right| \right\},\notag
 \end{align}
\begin{enumerate}
	\item If $\bgamma$ is another real-valued variable-dimensional parameter related to the covariates, then the acceptance probability of the birth move is:
 \begin{align}
 a_b &= 
	 \min\left\{1, \frac{1}{k^{(t)}+1}\times\frac{w_{d,k^{(t)}+1}}{w_{b,k^{(t)}}}
	 %\times\frac{1}{2^2}\times \dfrac{\pi\left(\bta',\bs',k^{(t)}+1\right)}{\pi\left(\bta^{(t)},\bs^{(t)},k^{(t)}\right)}\times \frac{2|\beta^{(t)}_j|}{|\e_1|}\times\frac{2|\gamma^{(t)}_j|}{|\e_2|}\right\}.\notag 
	 \dfrac{\pi\left(\bta',\bs',k^{(t)}+1\right)}{\pi\left(\bta^{(t)},\bs^{(t)},k^{(t)}\right)}\times \frac{|\beta^{(t)}_j|}{|\e_1|}\times\frac{|\gamma^{(t)}_j|}{|\e_2|}\right\}.\notag 
% ~\left|\frac{\partial (T_{b,\bz}(\supr{\bm x}{t}, \e))}{\partial(\supr{\bm x}{t}, \e)}\right| \right\},\notag
 \end{align}
\item For further variable-dimensional parameter associated with the covariates, noting that the process must be continued by further multiplying 
	the ratio of the absolute value of the current parameter value and the relevant $\e$, to the acceptance ratio.
\end{enumerate}
%\begin{align}
% a_b(\supr{\bm x}{t}, \e) &= 
% \min\left\{1, \frac{1}{k+1}\times\frac{w_{d,k+1}}{w_{b,k}}\times\dfrac{P_{(j)}(\bz^c)}{P_{(j)}(\bz)} 
% ~\dfrac{\pi(\bm x')}{\pi(\supr{\bm x}{t})} 
% ~\left|\frac{\partial (T_{b,\bz}(\supr{\bm x}{t}, \e))}{\partial(\supr{\bm x}{t}, \e)}\right| \right\},\notag
% \end{align}
% where
% \[ 
%P_{(j)}(\bz)=\prod_{i\neq j=1}^kp^{I_{\{1\}}(z_i)}_iq^{I_{\{-1\}}(z_i)}_i,
% \]
% and
% \[ 
%P_{(j)}(\bz^c)=\prod_{i\neq j=1}^kp^{I_{\{1\}(z^c_i)}}_iq^{I_{\{-1\}}(z^c_i)}_i.
% \]

\item Set \[ (\bta^{(t+1)},\bs^{(t+1)},k^{(t+1)})= \left\{\begin{array}{ccc}
		(\bta',\bs',k^{(t)}+1) & \mbox{ with probability } & a_b \\
	(\bta^{(t)},\bs^{(t)},k^{(t)}) & \mbox{ with probability } & 1 - a_b.
\end{array}\right.\]
\end{enumerate}
	 \end{enumerate}

 \item If $u_2=1$ (decrease dimension by deleting an existing covariate), generate $U\sim U(0,1)$ and do the following: 
	 \begin{enumerate}
		 \item  If $U\leq \tilde p$ (use additive transformation for dimension change),
 \begin{enumerate}
	 \item Randomly select a co-ordinate $j$ from $\{1,\ldots,k^{(t)}\}$ %$\bbeta^{(t)}=(\beta^{(t)}_1,\ldots,\beta^{(t)}_k)$ 
		 assuming uniform probability $1/k^{(t)}$ for each co-ordinate, and randomly select a co-ordinate $j'$ from $\{1,\ldots,k^{(t)}\}\backslash\{j\}$ 
		 with probability $1/(k^{(t)}-1)$.
		 Assuming $j<j'$, let $\beta^*_j=(\beta^{(t)}_j+\beta^{(t)}_{j'})/2$. Replace $\beta^{(t)}_j$ with $\beta^*_j$ and delete $\beta^{(t)}_{j'}$.
	 \item Delete $\{x_{ij'}:i=1,\ldots,n\}$. Let $\bs'=\bs^{(t)}\backslash\{j'\}$. 
%	 for $i=1,\ldots,k;~i\neq j$ simulate  
% $$z_i \sim Multinomial(1;p_i,q_i,1-p_i-q_i)$$ independently. 
 %For $i=1,\ldots,k$, the probabilities $p_i$ and $q_i$ need to be specified;
 %see \ctn{Dutta14} for a detailed discussion regarding these choices.
 \item Propose the following death move: 
	 $$ \bbeta'=(\beta^{(t)}_1,\ldots,\beta^{(t)}_{j-1},\beta^*_j,\beta^{(t)}_{j+1},\ldots,\beta^{(t)}_{j'-1},\beta^{(t)}_{j'+1},\ldots,\beta^{(t)}_{k^{(t)}}).$$	 
% \begin{align} 
% \bm x' &= T_{b,\bz}(\supr{\bm x}{t}, \e)=(g_{1,z_1}(x^{(t)}_1,\e),\ldots,g_{j-1,z_{j-1}}(x^{(t)}_{j-1},\e),\notag\\
%& g_{j,{z_j=1}}(x^{(t)}_j,\e),g_{j,{z^c_j=-1}}(x^{(t)}_j,\e),g_{j+1,z_{j+1}}(x^{(t)}_{j+1},\e),\ldots,
%g_{k,z_k}(x^{(t)}_k,\e)).\notag
%\end{align}
%Re-label the elements of $\bm x'$ as $(x'_1,x'_2,\ldots,x'_{k+1})$.
 \item Re-label the elements of $\bbeta'$ as $(\beta'_1,\beta'_2,\ldots,\beta'_{k^{(t)}-1})$.
 \begin{enumerate}
 \item If there is another set of real-valued variable-dimensional parameters, say, $\bgamma$, associated with the covariates, then propose 
	 $$ \bgamma'=(\gamma^{(t)}_1,\ldots,\gamma^{(t)}_{j-1},\gamma^*_j,\gamma^{(t)}_{j+1},\ldots,\gamma^{(t)}_{j'-1},\gamma^{(t)}_{j'+1},\ldots,\gamma^{(t)}_{k^{(t)}}),$$	 
 where $\gamma^*_j=(\gamma^{(t)}_j+\gamma^{(t)}_{j'})/2$.
 \item Re-label the elements of $\bgamma'$ as $(\gamma'_1,\gamma'_2,\ldots,\gamma'_{k^{(t)}-1})$.
 \item Repeat the procedure for further sets of variable-dimensional parameters related to the covariates.
 \item Keep all other elements of $\bta$ unchanged, and refer to the entire set of proposed parameter values as $\bta'$.
 \end{enumerate}
\item If $\bbeta$ is the only variable-dimensional parameter related to the covariates, then the acceptance probability of the death move is:
 \begin{align}
 a_d &= 
	 \min\left\{1, k^{(t)}\times\frac{w_{b,k^{(t)}-1}}{w_{d,k^{(t)}}} 
	 ~\dfrac{\pi\left(\bta',\bs',k^{(t)}-1\right)}{\pi\left(\bta^{(t)},\bs^{(t)},k^{(t)}\right)}\times\frac{1}{2a_{\beta,j}}\right\}.\notag 
% ~\left|\frac{\partial (T_{b,\bz}(\supr{\bm x}{t}, \e))}{\partial(\supr{\bm x}{t}, \e)}\right| \right\},\notag
 \end{align}
\begin{enumerate}
	\item If $\bgamma$ is another real-valued variable-dimensional parameter related to the covariates, then the acceptance probability of the death move is:
 \begin{align}
 a_d &= 
	 \min\left\{1, k^{(t)}\times\frac{w_{b,k^{(t)}-1}}{w_{d,k^{(t)}}} 
	 ~\dfrac{\pi\left(\bta',\bs',k^{(t)}-1\right)}{\pi\left(\bta^{(t)},\bs^{(t)},k^{(t)}\right)}\times \frac{1}{2a_{\beta,j}}\times \frac{1}{2a_{\gamma,j}}\right\},\notag 
% ~\left|\frac{\partial (T_{b,\bz}(\supr{\bm x}{t}, \e))}{\partial(\supr{\bm x}{t}, \e)}\right| \right\},\notag
 \end{align}
that is, $1/(2a_{\gamma,j})$ must also be multiplied to the acceptance ratio.
\item For further real-valued variable-dimensional parameter associated with the covariates, the process must be continued in the above manner.
\end{enumerate}
%\begin{align}
% a_b(\supr{\bm x}{t}, \e) &= 
% \min\left\{1, \frac{1}{k+1}\times\frac{w_{d,k+1}}{w_{b,k}}\times\dfrac{P_{(j)}(\bz^c)}{P_{(j)}(\bz)} 
% ~\dfrac{\pi(\bm x')}{\pi(\supr{\bm x}{t})} 
% ~\left|\frac{\partial (T_{b,\bz}(\supr{\bm x}{t}, \e))}{\partial(\supr{\bm x}{t}, \e)}\right| \right\},\notag
% \end{align}
% where
% \[ 
%P_{(j)}(\bz)=\prod_{i\neq j=1}^kp^{I_{\{1\}}(z_i)}_iq^{I_{\{-1\}}(z_i)}_i,
% \]
% and
% \[ 
%P_{(j)}(\bz^c)=\prod_{i\neq j=1}^kp^{I_{\{1\}(z^c_i)}}_iq^{I_{\{-1\}}(z^c_i)}_i.
% \]

\item Set \[ (\bta^{(t+1)},\bs^{(t+1)},k^{(t+1)})= \left\{\begin{array}{ccc}
		(\bta',\bs',k^{(t)}-1) & \mbox{ with probability } & a_d \\
	(\bta^{(t)},\bs^{(t)},k^{(t)}) & \mbox{ with probability } & 1 - a_d.
\end{array}\right.\]
\end{enumerate}
	 \end{enumerate}

	 \begin{enumerate}
		 \item[(b)]  If $U>\tilde p$ (use multiplicative transformation for dimension change),
 \begin{enumerate}
	 \item Randomly select a co-ordinate $j$ from $\{1,\ldots,k^{(t)}\}$ %$\bbeta^{(t)}=(\beta^{(t)}_1,\ldots,\beta^{(t)}_k)$ 
		 assuming uniform probability $1/k^{(t)}$ for each co-ordinate, and randomly select a co-ordinate $j'$ from $\{1,\ldots,k^{(t)}\}\backslash\{j\}$ 
		 with probability $1/(k^{(t)}-1)$.
 Assuming $j<j'$, let $\beta^*_j=\sqrt{|\beta^{(t)}_j\beta^{(t)}_{j'}|}$ with probability $1/2$ and set 
 $\beta^*_j=-\sqrt{|\beta^{(t)}_j\beta^{(t)}_{j'}|}$ with the remaining probability. Replace $\beta^{(t)}_j$ with $\beta^*_j$ and delete $\beta^{(t)}_{j'}$.
 \item Delete $\{x_{ij'}:i=1,\ldots,n\}$. Let $\bs'=\bs^{(t)}\backslash\{j'\}$. 
%	 for $i=1,\ldots,k;~i\neq j$ simulate  
% $$z_i \sim Multinomial(1;p_i,q_i,1-p_i-q_i)$$ independently. 
 %For $i=1,\ldots,k$, the probabilities $p_i$ and $q_i$ need to be specified;
 %see \ctn{Dutta14} for a detailed discussion regarding these choices.
 \item Propose the following death move: 
	 $$ \bbeta'=(\beta^{(t)}_1,\ldots,\beta^{(t)}_{j-1},\beta^*_j,\beta^{(t)}_{j+1},\ldots,\beta^{(t)}_{j'-1},\beta^{(t)}_{j'+1},\ldots,\beta^{(t)}_{k^{(t)}}).$$	 
% \begin{align} 
% \bm x' &= T_{b,\bz}(\supr{\bm x}{t}, \e)=(g_{1,z_1}(x^{(t)}_1,\e),\ldots,g_{j-1,z_{j-1}}(x^{(t)}_{j-1},\e),\notag\\
%& g_{j,{z_j=1}}(x^{(t)}_j,\e),g_{j,{z^c_j=-1}}(x^{(t)}_j,\e),g_{j+1,z_{j+1}}(x^{(t)}_{j+1},\e),\ldots,
%g_{k,z_k}(x^{(t)}_k,\e)).\notag
%\end{align}
%Re-label the elements of $\bm x'$ as $(x'_1,x'_2,\ldots,x'_{k+1})$.
 \item Re-label the elements of $\bbeta'$ as $(\beta'_1,\beta'_2,\ldots,\beta'_{k^{(t)}-1})$.
 \begin{enumerate}
 \item If there is another set of real-valued variable-dimensional parameters, say, $\bgamma$, associated with the covariates, then propose 
	 $$ \bgamma'=(\gamma^{(t)}_1,\ldots,\gamma^{(t)}_{j-1},\gamma^*_j,\gamma^{(t)}_{j+1},\ldots,\gamma^{(t)}_{j'-1},\gamma^{(t)}_{j'+1},\ldots,\gamma^{(t)}_{k^{(t)}}),$$	 
 where $\gamma^*_j=\sqrt{|\gamma^{(t)}_j\gamma^{(t)}_{j'}|}$ or $-\sqrt{|\gamma^{(t)}_j\gamma^{(t)}_{j'}|}$ with equal probabilities.
 \item Re-label the elements of $\bgamma'$ as $(\gamma'_1,\gamma'_2,\ldots,\gamma'_{k^{(t)}-1})$.
 \item Repeat the procedure for further sets of variable-dimensional parameters related to the covariates.
 \item Keep all other elements of $\bta$ unchanged, and refer to the entire set of proposed parameter values as $\bta'$.
 \end{enumerate}
\item If $\bbeta$ is the only variable-dimensional parameter related to the covariates, then the acceptance probability of the death move is:
 \begin{align}
 a_d &= 
	 \min\left\{1, k^{(t)}\times\frac{w_{b,k^{(t)}-1}}{w_{d,k^{(t)}}} 
	 ~\dfrac{\pi\left(\bta',\bs',k^{(t)}-1\right)}{\pi\left(\bta^{(t)},\bs^{(t)},k^{(t)}\right)}\times \frac{1}{|\beta^{(t)}_{j'}|}\right\}.\notag 
% ~\left|\frac{\partial (T_{b,\bz}(\supr{\bm x}{t}, \e))}{\partial(\supr{\bm x}{t}, \e)}\right| \right\},\notag
 \end{align}
\begin{enumerate}
	\item If $\bgamma$ is another real-valued variable-dimensional parameter related to the covariates, then the acceptance probability of the death move is:
 \begin{align}
 a_d &= 
	 \min\left\{1, k^{(t)}\times\frac{w_{b,k^{(t)}-1}}{w_{d,k^{(t)}}} 
	 ~\dfrac{\pi\left(\bta',\bs',k^{(t)}-1\right)}{\pi\left(\bta^{(t)},\bs^{(t)},k^{(t)}\right)}\times \frac{1}{|\beta^{(t)}_{j'}|}\times \frac{1}{|\gamma^{(t)}_{j'}|}\right\},\notag 
% ~\left|\frac{\partial (T_{b,\bz}(\supr{\bm x}{t}, \e))}{\partial(\supr{\bm x}{t}, \e)}\right| \right\},\notag
 \end{align}
that is, $1/|\gamma^{(t)}_{j'}|$ must also be multiplied to the acceptance ratio.
\item For further real-valued variable-dimensional parameter associated with the covariates, the process must be continued in the above manner.
\end{enumerate}
%\begin{align}
% a_b(\supr{\bm x}{t}, \e) &= 
% \min\left\{1, \frac{1}{k+1}\times\frac{w_{d,k+1}}{w_{b,k}}\times\dfrac{P_{(j)}(\bz^c)}{P_{(j)}(\bz)} 
% ~\dfrac{\pi(\bm x')}{\pi(\supr{\bm x}{t})} 
% ~\left|\frac{\partial (T_{b,\bz}(\supr{\bm x}{t}, \e))}{\partial(\supr{\bm x}{t}, \e)}\right| \right\},\notag
% \end{align}
% where
% \[ 
%P_{(j)}(\bz)=\prod_{i\neq j=1}^kp^{I_{\{1\}}(z_i)}_iq^{I_{\{-1\}}(z_i)}_i,
% \]
% and
% \[ 
%P_{(j)}(\bz^c)=\prod_{i\neq j=1}^kp^{I_{\{1\}(z^c_i)}}_iq^{I_{\{-1\}}(z^c_i)}_i.
% \]

\item Set \[ (\bta^{(t+1)},\bs^{(t+1)},k^{(t+1)})= \left\{\begin{array}{ccc}
		(\bta',\bs',k^{(t)}-1) & \mbox{ with probability } & a_d \\
	(\bta^{(t)},\bs^{(t)},k^{(t)}) & \mbox{ with probability } & 1 - a_d.
\end{array}\right.\]
\end{enumerate}
	 \end{enumerate}

\item If $u_3=1$ (dimension remains unchanged), then given that there are $d$ dimensions in the current iteration, 
		generate $U\sim U(0,1)$. 
	\begin{enumerate}
	\item If $U\leq \tilde p$, then do the following:
	\begin{enumerate}
		\item[(i)] For parameters $\bbeta$, $\bgamma$, etc. associated with the covariates, for $j=1,\ldots,k^{(t)}$, set 
	$\tilde a_{\beta,j}=ca_{\beta,j}$, $\tilde a_{\gamma,j}=ca_{\gamma,j}$, etc. 
	where $c\in (0,1)$ is some appropriate constant. %We shall set $a_{\theta,j}$ reasonably large and $c\in (0,1)$ in our applications.
        For all other parameter co-ordinates $\theta_j$, let $\tilde a_{\theta,j}=a_{\theta,j}$.
		\item[(ii)] Generate $\varepsilon\sim N(0,1)$, $b_j\stackrel{iid}{\sim}U(\{-1,1\})$ for $j=1,\ldots,d$, and set
			$\theta'_j=\theta^{(t)}_j+b_j\tilde a_{\theta,j}|\varepsilon|$, for $j=1,\ldots,d$.  
		\item[(iii)] Evaluate 
                        \begin{equation*}
				\alpha_1=\min\left\{1,\frac{\pi(\bta',\bs^{(t)},k^{(t)})}{\pi\left(\bta^{(t)},\bs^{(t)},k^{(t)}\right)}\right\}.
	                  %\label{eq:acc1}
                        \end{equation*}
		\item[(iv)] Set $(\bta^{(t+1)},\bs^{(t+1)},k^{(t+1)})=(\bta',\bs^{(t)},k^{(t)})$ with probability $\alpha_1$, 
			else set $(\bta^{(t+1)},\bs^{(t+1)},k^{(t+1)})=(\bta^{(t)},\bs^{(t)},k^{(t)})$. 
	\end{enumerate}
	\end{enumerate}

      \begin{enumerate}
	      \item[(b)] If $U>\tilde p$, then do the following:
	\begin{enumerate}
		\item[(i)] Generate $\varepsilon\sim U(-1,1)$, $b_j\stackrel{iid}{\sim}U(\{-1,0,1\})$ for $j=1,\ldots,d$, and set
			$\theta'_j=\theta^{(t)}_j\varepsilon$ if $b_j=1$, $\theta'_j=\theta^{(t)}_j/\varepsilon$ if $b_j=-1$ and 
			$\theta'_j=\theta^{(t)}_j$ if $b_j=0$, for $j=1,\ldots,d$. 
			Calculate $|J|=|\varepsilon|^{\sum_{j=1}^db_j}$.
		\item[(ii)] Evaluate 
                        \begin{equation*}
				\alpha_2=\min\left\{1,\frac{\pi(\bta',\bs^{(t)},k^{(t)})}{\pi(\bta^{(t)},\bs^{(t)},k^{(t)})}\times|J|\right\}.
	                  %\label{eq:acc1}
                        \end{equation*}
		\item[(iii)] Set $(\bta^{(t+1)},\bs^{(t+1)},k^{(t+1)})=(\bta',\bs^{(t)},k^{(t)})$ with probability $\alpha_2$, 
			else set $(\bta^{(t+1)},\bs^{(t+1)},k^{(t+1)})=(\bta^{(t)},\bs^{(t)},k^{(t)})$. 
	\end{enumerate}
	\end{enumerate}
		
\item ({\it Mixing-enhancement step}) 
	Assume that there are $d$ dimensions in the current iteration after implementing either of the birth, death and no-change steps. Generate $U\sim U(0,1)$. 
	\begin{enumerate}
		\item If $U\leq \tilde q$, where $\tilde q\in (0,1)$, then do the following
	\begin{enumerate}
		\item[(i)] For parameters $\bbeta$, $\bgamma$, etc. associated with the covariates, for $j=1,\ldots,k^{(t+1)}$, set 
	$\tilde a_{\beta,j}=ca_{\beta,j}$, $\tilde a_{\gamma,j}=ca_{\gamma,j}$, etc. 
	where $c\in (0,1)$ is some appropriate constant. %We shall set $a_{\theta,j}$ reasonably large and $c\in (0,1)$ in our applications.
        For all other parameter co-ordinates $\theta_j$, let $\tilde a_{\theta,j}=a_{\theta,j}$.
		\item[(ii)] Generate $\tilde U\sim U(0,1)$ and $\varepsilon\sim N(0,1)$. If $\tilde U<1/2$, set
			$\theta''_j=\theta^{(t+1)}_j+\tilde a_{\theta,j}|\varepsilon|$, for $j=1,\ldots,d$; else, 
			set $\theta''_j=\theta^{(t+1)}_j-\tilde a_{\theta,j}|\varepsilon|$, for $j=1,\ldots,d$. 
		\item[(iii)] Letting $\bta''=(\theta''_1,\ldots,\theta''_d)$, evaluate 
                        \begin{equation*}
				\alpha_3=\min\left\{1,\frac{\pi(\bta'',\bs^{(t+1)},k^{(t+1)})}{\pi(\bta^{(t+1)},\bs^{(t+1)},k^{(t+1)})}\right\}.
	                  %\label{eq:acc1}
                        \end{equation*}
		\item[(iv)] Set $(\tilde\bta^{(t+1)},\bs^{(t+1)},k^{(t+1)})=(\bta'',\bs^{(t+1)},k^{(t+1)})$ with probability $\alpha_3$, 
			else set $(\tilde\bta^{(t+1)},\bs^{(t+1)},k^{(t+1)})=(\bta^{(t+1)},\bs^{(t+1)},k^{(t+1)})$. 
	\end{enumerate}
	\end{enumerate}

	\begin{enumerate}
        \item[(b)] If $U> \tilde q$, then
	\begin{enumerate}
		\item[(i)] Generate $\varepsilon\sim U(-1,1)$ and $\tilde U\sim U(0,1)$. If $\tilde U<1/2$, set
			$\theta''_j=\theta^{(t+1)}_j\varepsilon$ for $j=1,\ldots,d$ and $|J|=|\varepsilon|^d$, else set 
			$\theta''_j=\theta^{(t+1)}_j/\varepsilon$ for $j=1,\ldots,d$ and 
			$|J|=|\varepsilon|^{-d}$.
		\item[(ii)] Evaluate 
               \begin{equation*}
		       \alpha_4=\min\left\{1,\frac{\pi(\bta'',\bs^{(t+1)},k^{(t+1)})}{\pi(\bta^{(t+1)},\bs^{(t+1)},k^{(t+1)})}\times|J|\right\}.
	       %\label{eq:acc1}
               \end{equation*}
       \item[(iii)] Set $(\tilde\bta^{(t+1)},\bs^{(t+1)},k^{(t+1)})=(\bta'',\bs^{(t+1)},k^{(t+1)})$ with probability $\alpha_4$, 
	       else set $(\tilde\bta^{(t+1)},\bs^{(t+1)},k^{(t+1)})=(\bta^{(t+1)},\bs^{(t+1)},k^{(t+1)})$. 
	\end{enumerate}
	\end{enumerate}
\end{enumerate}
\item End for
\item Store $\{(\tilde\bta^{(0)},\bs^{(0)},k^{(0)}),(\tilde\bta^{(1)},\bs^{(1)},k^{(1)}),\ldots\}$ for Bayesian inference.
\end{itemize}
\botline \rmfamily
\end{algo}

The main strategies proposed in the general TTMCMC Algorithm \ref{algo:ttmcmc} for variable selection require some elucidation. In this regard,
a few remarks are in order.

First, we propose a mixture of additive and multiplicative transformations in all the steps of the algorithm, since it has been observed in \ctn{Dey16}
that such mixture proposal induces better mixing that either additive or multiplicative transformations using the localised moves of the additive transformation
and the non-localised (``random dive") moves of the multiplicative transformation (see also \ctn{Dutta12} for some theoretical details on random dive).

In the dimension-changing steps 2. and 3. of Algorithm \ref{algo:ttmcmc}, except for the parameters associated with increase or decrease of the dimension, 
we have proposed to keep all the remaining parameters fixed. Fixing the other parameters is not necessary for the validity of TTMCMC; indeed, \ctn{Das19}
proposed to update all the parameters even in the dimension-changing steps. However, in our variable selection experiments, fixing the remaining parameters
led to significantly improved acceptance rates of the birth and death steps compared to the strategy of updating all the unknowns simultaneously.
The choice of the positive scales $a_{\theta,j}$ in the additive transformation part plays important role here. To elucidate, note that it is natural to
expect high acceptance rates with sufficiently small scales in fixed-dimensional problems, but in our variable-dimensional setup, observe that the acceptance
ratios for the birth and death steps depend upon the scales of the parameters selected for birth and death. If the scales are generally chosen to be small, then
the acceptance rate for the birth move would be small as well. On the other hand, if the scales are generally chosen to be relatively large, then the acceptance rate
for the entire dimension-changing move would be small, for a relatively large number of parameters. With these small or large scale choices, the acceptance ratios
in the no-change (fixed-dimensional) step 4. and the mixing-enhancement step 5. would also be small.

We attempt to solve all the above problems with the strategy of choosing somewhat large scales $a_{\theta,j}$ 
and by fixing the parameters in the birth and death steps that are not involved in dimension-change. The relatively large scales would ensure adequate acceptance
rate for the birth move; note that the scales should not be so large as to reduce the death rate significantly. Now, these large scales would also diminish the
acceptance rates in the no-change and the mixing-enhancement steps. To counter this, we multiply the scales of the parameters associated with the covariates 
by $c\in (0,1)$ in those steps, which is a valid
mathematical strategy in the sense of satisfying detailed balance. Further discussion regarding these will be provided in course of the applications of  
Algorithm \ref{algo:ttmcmc}.  

The fixed-dimensional mixing-enhancement step has parallels with \ctn{Liu00} 
(see also the supplement of \ctn{Dutta14} and Algorithm 2 of \ctn{Roy20}). Indeed, it has been observed that 
the strategy can often drastically improve the mixing properties in fixed-dimensional setups.

Finally, note that $\bs$ and $k$ are not updated in the no-change and mixing enhancing steps, so that $\pi(\bs|k)\pi(k)$ gets cancelled in the corresponding acceptance ratios.

\section{Bayes factor computation using TTMCMC realizations}
\label{sec:bf_comp}
Assuming that there are $N$ realizations of TTMCMC %$\{(\tilde\bta^{(1)},k^{(1)}),(\tilde\bta^{(2)},k^{(2)}),\ldots,\tilde\bta^{(N)},k^{(N)})\}$ 
stored for Bayesian inference after discarding a suitable burn-in period, the Bayes factors associated with the distinct subsets of the covariates featuring in the TTMCMC
samples can be calculated as follows.

Let there be $\tilde N~(<N)$ distinct subsets $\left\{\bs^*_1,\bs^*_2,\ldots,\bs^*_{\tilde N}\right\}$ in the TTMCMC sample, each subset consisting of distinct indices
of a set of covariates which is a subset of the entire pool of covariates indexed by $\bS$. Thus, the TTMCMC sample consists of $\tilde N$ distinct subsets of covariates
out of a total $2^p-1$ possibilities, $p=|\bS|$ being the total available number of covariates. The subsets of covariates that did not feature in the TTMCMC sample
will be interpreted as having negligible posterior probabilities and will be not be considered any further for our Bayesian analyses.

For $i=1,\ldots,\tilde N$, assuming that $\bs^*_i$ is repeated $N_i$ times in the TTMCMC sample, so that $\sum_{i=1}^{\tilde N}N_i=N$, we estimate its posterior probability by
$\tilde\pi(\bs^*_i)=N_i/N$. Let $k^*_i=|\bs^*_i|$ be the cardinality of $\bs^*_i$. Note that the prior for the model associated with any subset $\bs$ consisting
of $k$ covariates is uniform over all ${p\choose k}$ possibilities, given by (\ref{eq:prior_s}). 
%\begin{equation*}
%	\pi(\bs|k)=\frac{1}{{p\choose k}}.
%	%\label{eq:modelprob1}
%\end{equation*}
Hence, the marginal prior probability of $\bs$ with $|\bs|=k$ is
\begin{equation}
\pi(\bs)=\sum_{j=1}^p\pi(\bs|j)\pi(j)=\pi(\bs|k)\pi(k)=\frac{\pi(k)}{{p\choose k}},
\label{eq:marginal_modelprob1}
\end{equation}
since $\pi(\bs|j)=0$ if $j\neq k$. In the above, $\pi(k)$ denotes the prior for $k$.

Using (\ref{eq:marginal_modelprob1}), we compute for each $i=1,\ldots,\tilde N$,
\begin{equation}
B_i=\frac{\tilde\pi(\bs^*_i)}{\pi(\bs^*_i)}=\frac{N_i}{N}\times\frac{{p\choose k^*_i}}{\pi(k^*_i)}.
\label{eq:B_i}
\end{equation}
For any $i,j\in\{1,\ldots,\tilde N\}$, the (approximate) Bayes factor of the model associated with $\bs^*_i$ against that associated with $\bs^*_j$ is given by
\begin{equation}
BF_{ij}=B_i/B_j.
\label{eq:bf_ij}
\end{equation}
Thus, the best model is the one with the largest $B_i$; $i=1,\ldots,\tilde N$.
Note that $B_i$ is proportional to the marginal density of the data, given the $i$-th model, where the proportionality constant is the same for all the competing models.

\section{Proof of convergence of the TTMCMC algorithm}
\label{sec:proof_conv}

To prove convergence of Algorithm \ref{algo:ttmcmc} it is sufficient to establish detailed balance, irreducibility and aperiodicity of the algorithm, which we 
undertake step-by-step in this section.
For simplicity, let us assume that $\bbeta$ is the only parameter vector associated with the covariates. The extension is trivial for other parameter vectors
associated with the covariates.

\subsection{Proof of detailed balance}

\subsubsection{Additive transformation}

Let us first consider the case of the additive transformation, which we select with probability $\tilde p$.
To see that detailed balance is satisfied for the birth and death moves, note that associated with the birth move, 
the probability (essentially) of transition $(\bbeta^{(t)},\bs^{(t)},k)\mapsto (\bbeta',\bs',k+1)$, with $k=|\bs^{(t)}|$ and
$k+1=|\bs'|$ (so that $\bbeta^{(t)}\in\mathbb R^k$ and $\bbeta'\in\mathbb R^{k+1}$), 
while the other elements of $\bta$ are held fixed, is given by:
\begin{align}
	&\pi(\bta^{(t)},\bs^{(t)},k)\times \tilde p\times\frac{1}{k}\times w_{b,k}\times N(\e:0,1)\notag\\
%\times\prod_{i\neq j=1}^kp^{I_{\{1\}}(z_i)}_iq^{I_{\{-1\}}(z_i)}_i\notag\\
	&\qquad\times\min\left\{1, \frac{1}{k+1}\times\frac{w_{d,k+1}}{w_{b,k}}
%\frac{\prod_{i\neq j=1}^k p^{I_{\{1\}}(z^c_i)}_iq^{I_{\{-1\}}(z^c_i)}_i}
%{\prod_{i\neq j=1}^kp^{I_{\{1\}}(z_i)}_iq^{I_{\{-1\}}(z_i)}_i}%\right.\notag\\
%&\quad\quad\quad\quad\left.
	\times\frac{\pi(\bta',\bs',k+1)}{\pi(\bta^{(t)},\bs^{(t)},k)}\times 
\left|\frac{\partial\bbeta'}{\partial(\bbeta, \e)}\right|\right\}\notag\\
	&=\tilde p\times N(\e:0,1)\times\min\left\{\pi(\bta^{(t)},\bs^{(t)},k+1)\times\frac{1}{k}\times w_{b,k},\right.\notag\\
%\prod_{i\neq j=1}^kp^{I_{\{1\}}(z_i)}_iq^{I_{\{-1\}}(z_i)}_i,\right.\notag\\
%&\quad\quad\quad\quad\left.
	&\qquad\left.\frac{1}{k(k+1)}\times w_{d,k+1}\times
%\prod_{i\neq j=1}^k p^{I_{\{1\}}(z^c_i)}_iq^{I_{\{-1\}}(z^c_i)}_i
	\pi(\bta',\bs',k+1)\times \left|\frac{\partial\bbeta'}{\partial(\bbeta^{(t)}, \e)}\right|
\right\},
\label{eq:db_birth}
\end{align}
where $N(\e:0,1)$ is the density of the normal distribution with mean $0$ and variance $1$, evaluated at $\e$.
Assuming that $\beta^{(t)}_j$ was selected, and was split into $\beta^{(t)}_j+a_{\beta,j}\e$ and $\beta^{(t)}_j-a_{\beta,j}\e$, 
$\left|\frac{\partial\bbeta'}{\partial(\bbeta, \e)}\right|=2a_{\beta,j}$.

At the reverse death move we must be able to return to $(\bbeta^{(t)},\bs^{(t)},k)$ from 
$(\bbeta',\bs',k+1)$, while the other elements of $\bta$ are held fixed.
We select $\beta'_j$ with probability $1/(k+1)$, then select $\beta'_{j+1}$ without
replacement with probability $1/k$, and take the resultant average.

Let $\e^*$ be such that $\beta^{(t)}_j+a_{\beta,j}\e^*=\beta'_j$ and $\beta^{(t)}_j-a_{\beta,j}\e^*=\beta'_{j+1}$, so that $\e^*=(\beta'_j-\beta'_{j+1})/2$.
The transition probability of the death move is hence given by:
\begin{align}
	&\pi(\bta',\bs',k+1)\times \tilde p\times w_{d,k+1}\times N(\e,0,1)
\times\frac{1}{k+1}\times\frac{1}{k}
%\times \left|\frac{\partial (T^{-1}_{d,\bz}({\bm x}, \e),\e^*)}{\partial({\bm x}, \e)}\right|
	\times \left|\frac{\partial (\bbeta',\e)}{\partial(\bbeta^{(t)},\e^*,\e)}\right|
\notag\\
&\qquad\qquad\times\min\left\{1,(k+1)\times\frac{w_{b,k}}{w_{d,k+1}}
%\times\frac{\prod_{i\neq j=1}^kp^{I_{\{1\}}(z_i)}_iq^{I_{\{-1\}}(z_i)}_i}
%{\prod_{i\neq j=1}^k p^{I_{\{1\}}(z^c_i)}_iq^{I_{\{-1\}}(z^c_i)}_i}%\right.\notag\\
%&\quad\quad\quad\quad\left.
	\times\frac{\pi(\bta^{(t)},\bs^{(t)},k)}{\pi(\bta',\bs',k+1)}\times
%\left|\frac{\partial (T_{d,\bz}({\bm x}, \e),\e^*)}{\partial({\bm x}, \e)}\right|
	\left|\frac{\partial(\bbeta^{(t)},\e^*,\e)}{\partial (\bbeta',\e)}\right|\right\}
\notag\\
&=\tilde p\times N(\e:0,1)\times\min\left\{\pi(\bta',\bs',k+1)\times w_{d,k+1}
%\prod_{i\neq j=1}^k p^{I_{\{1\}}(z^c_i)}_iq^{I_{\{-1\}}(z^c_i)}_i
\times\frac{1}{k(k+1)}
\times \left|\frac{\partial (\bbeta',\e)}{\partial(\bbeta^{(t)},\e^*,\e)}\right|,\right.\notag\\
	&\qquad\qquad\qquad\qquad\left. 
	\frac{1}{k}\times w_{b,k}\times\pi(\bta^{(t)},\bs^{(t)},k) \right\}\notag\\
&=\tilde p\times N(\e:0,1)\times\min\left\{\pi(\bta',\bs',k+1)\times w_{d,k+1}
\times\frac{1}{k(k+1)}
	\times 2a_{\beta,j},\right.\notag\\
	&\qquad\qquad\qquad\qquad\left.\frac{1}{k}\times w_{b,k}\times\pi(\bta^{(t)},\bs^{(t)},k) \right\}.
\label{eq:db_death}
\end{align}
Thus, (\ref{eq:db_birth}) = (\ref{eq:db_death}), showing that detailed balance holds for the birth and the death moves.
The proof of detailed balance for the no-change move type where the dimension remains unchanged is the same as that of
TMCMC, and has been been proved in the supplement of \ctn{Dutta14}.

\subsubsection{Multiplicative transformation}

Now let us consider the multiplicative transformation, which we select with probability $1-\tilde p$.
For the birth move, the probability (essentially) of the transition $(\bbeta^{(t)},\bs^{(t)},k)\mapsto (\bbeta',\bs,k+1))$, 
while the other elements of $\bta$ are held fixed, is given by:
\begin{align}
	&\pi(\bta^{(t)},\bs^{(t)},k)\times (1-\tilde p)\times\frac{1}{k}\times w_{b,k}\times U(\e:-1,1)\notag\\
%\times\prod_{i\neq j=1}^kp^{I_{\{1\}}(z_i)}_iq^{I_{\{-1\}}(z_i)}_i\notag\\
	&\qquad	\times\min\left\{1, \frac{1}{k+1}\times\frac{w_{d,k+1}}{w_{b,k}}\times\frac{1}{2}
%\frac{\prod_{i\neq j=1}^k p^{I_{\{1\}}(z^c_i)}_iq^{I_{\{-1\}}(z^c_i)}_i}
%{\prod_{i\neq j=1}^kp^{I_{\{1\}}(z_i)}_iq^{I_{\{-1\}}(z_i)}_i}%\right.\notag\\
%&\quad\quad\quad\quad\left.
	\times\frac{\pi(\bta',\bs',k+1)}{\pi(\bta^{(t)},\bs^{(t)},k)}\times 
	\left|\frac{\partial\bbeta'}{\partial(\bbeta^{(t)}, \e)}\right|\right\}\notag\\
	&=(1-\tilde p)\times U(\e:-1,1)
	\times\min\left\{\pi(\bta^{(t)},\bs^{(t)},k)\times\frac{1}{k}\times w_{b,k},\right.\notag\\
%\prod_{i\neq j=1}^kp^{I_{\{1\}}(z_i)}_iq^{I_{\{-1\}}(z_i)}_i,\right.\notag\\
%&\quad\quad\quad\quad\left.
	&\qquad\left.\frac{1}{k(k+1)}\times w_{d,k+1}\times
%\prod_{i\neq j=1}^k p^{I_{\{1\}}(z^c_i)}_iq^{I_{\{-1\}}(z^c_i)}_i
	\pi(\bta',\bs',k+1)\times\frac{1}{2}\times \left|\frac{\partial\bbeta'}{\partial(\bbeta^{(t)}, \e)}\right|
\right\},
\label{eq:db_birth2}
\end{align}
where $U(\e:-1,1)$ is the density of the uniform distribution on $[-1,1]$, evaluated at $\e$.
Assuming that $\beta^{(t)}_j$ was selected, and was split into $\beta^{(t)}_j\e$ and $\beta^{(t)}_j/ \e$, 
$\left|\frac{\partial\bbeta'}{\partial(\bbeta^{(t)}, \e)}\right|=2|\beta^{(t)}_j|/|\e|$.

At the reverse death move we must be able to return to $(\bbeta^{(t)},\bs^{(t)},k)$ from 
$(\bbeta',\bs',k+1)$, while the other elements of $\bta$ are held fixed.
We select $\beta'_j$ with probability $1/(k+1)$, then select $\beta'_{j+1}$ without
replacement with probability $1/k$, and take $\sqrt{|\beta'_j\beta'_{j+1}|)}$ or $-\sqrt{|\beta'_j\beta'_{j+1}|)}$ with equal probabilities.

Let $\e^*$ be such that $\beta^{(t)}_j\e^*=\beta'_j$ and $\beta^{(t)}_j/\e^*=\beta'_{j+1}$, so that $\e^*=\pm\sqrt{|\beta'_j\beta'_{j+1}|)}$.
The transition probability of the death move is hence given by:
\begin{align}
	&\pi(\bta',\bs',k+1)\times (1-\tilde p)\times w_{d,k+1}\times U(\e,-1,1)
\times\frac{1}{k+1}\times\frac{1}{k}
%\times \left|\frac{\partial (T^{-1}_{d,\bz}({\bm x}, \e),\e^*)}{\partial({\bm x}, \e)}\right|
	\times \frac{1}{2}\times\left|\frac{\partial (\bbeta',\e)}{\partial(\bbeta^{(t)},\e^*,\e)}\right|
\notag\\
&\qquad\qquad\times\min\left\{1,(k+1)\times\frac{w_{b,k}}{w_{d,k+1}}
%\times\frac{\prod_{i\neq j=1}^kp^{I_{\{1\}}(z_i)}_iq^{I_{\{-1\}}(z_i)}_i}
%{\prod_{i\neq j=1}^k p^{I_{\{1\}}(z^c_i)}_iq^{I_{\{-1\}}(z^c_i)}_i}%\right.\notag\\
%&\quad\quad\quad\quad\left.
	\times\frac{\pi(\bta^{(t)},\bs^{(t)},k)}{\pi(\bta',\bs',k+1)}\times 2\times
%\left|\frac{\partial (T_{d,\bz}({\bm x}, \e),\e^*)}{\partial({\bm x}, \e)}\right|
	\left|\frac{\partial(\bbeta^{(t)},\e^*,\e)}{\partial (\bbeta',\e)}\right|\right\}
\notag\\
	&=(1-\tilde p)\times U(\e:-1,1)\times\min\left\{\pi(\bta',\bs',k+1)\times w_{d,k+1}
%\prod_{i\neq j=1}^k p^{I_{\{1\}}(z^c_i)}_iq^{I_{\{-1\}}(z^c_i)}_i
	\times\frac{1}{k(k+1)}\times\frac{1}{2}\right.\notag\\
	&\qquad\qquad\left.\times \left|\frac{\partial (\bbeta',\e)}{\partial(\bbeta^{(t)},\e^*,\e)}\right|,%\right.\notag\\
%	&\quad\quad\quad\quad \left. 
	\frac{1}{k}\times w_{b,k}\times\pi(\bta^{(t)},\bs^{(t)},k) \right\}\notag\\
	&=(1-\tilde p)\times U(\e:-1,1)\times\min\left\{\pi(\bta',\bs',k+1)\times w_{d,k+1}
\times\frac{1}{k(k+1)}\right.\notag\\
	&\qquad\qquad\left.	\times |\beta'_{j'}|,\frac{1}{k}\times w_{b,k}\times\pi(\bta^{(t)},\bs^{(t)},k) \right\}.
\label{eq:db_death2}
\end{align}
Noting that $|\beta'_{j'}|=|\beta^{(t)}_j|/|\e|$, it is seen that (\ref{eq:db_birth2}) = (\ref{eq:db_death2}); that is, detailed balance holds for the birth and the death moves
with respect to the multiplicative transformation.
Again, the proof of detailed balance for the no-change move type where the dimension remains unchanged is the same as that of
TMCMC.

Also, the proof of detailed balance of the mixing-enhancement step (Step 5. of Algorithm \ref{algo:ttmcmc}) is the same as that of TMCMC.

\subsection{Irreducibility and aperiodicity}
\label{subsec:irr_ap}
The proof of irreducibility and aperiodicity of Algorithm \ref{algo:ttmcmc} follows easily from the general arguments provided in the supplements of \ctn{Das19}
and \ctn{Dutta14}.

%It is easy to see that our TTMCMC algorithm is irreducible and aperiodic. Assume that 
%$\bx\in\mathbb R^k$, with $k\geq 1$. For $k'>0$ with $k'\neq k$, let $(k',A_{k'})$ have positive probability under
%the target distribution, that is, $\pi(k',A_{k'})>0$; here $A_{k'}$ is a Borel set associated with $\mathbb R^{k'}$. 
%Then $\mathbb R^{k'}$ can be reached from $\bx\in\mathbb R^k$ in a finite number of steps using the birth
%and the death moves, accordingly as $k'>k$ or $k'<k$. 
%Thus, if $k'>k$, $\mathbb R^{k'}$ can be reached in $(k'-k)$ steps by applying the birth move, and if $k'<k$,
%then $\mathbb R^{k'}$ can be reached in $(k-k')$ steps using the death move.
%Once $\mathbb R^{k'}$ is reached 
%the no-change move-type and the transformations can be used to reach $A_{k'}$ in $k'$ steps. For the
%proof of the latter see \ctn{Dutta14} and \ctn{Dey14}. 
%Thus, $(k',A_{k'})$ can be reached from $\bx\in\mathbb R^k$ in $(|k'-k|+k')$ steps with positive probability.
%Since the set $(k',A_{k'})$ is arbitrary, aperiodicity also follows.

\section{Bayes factor based variable selection experiments with TTMCMC}\label{sec:simstudy}
%\label{sec:9.2}

%\subsection{Analysis of riboﬂavin data set}\label{sec:9.3}

%\section{Simulation experiments}

We now provide details of our simulation studies with respect to variable selection. We consider linear regression (Section \ref{sec:linear_regression} of MB), 
Gaussian process regression with squared exponential covariance kernel (Section \ref{sec:gp_illustration} of MB)
as well as autoregressive regression (Section \ref{subsec:time_series1} of MB) for our purpose.

\subsection{Linear regression}
\label{subsec:linreg_simstudy}

\subsubsection{Data generation with random sets of covariates}
As in Section \ref{sec:linear_regression} of MB, we consider the model of the form 
$y_i=\bbeta'_{\bs}\bx_{i,\bs}+\epsilon_i$, where $\epsilon_i\stackrel{iid}{\sim}N(0,\sigma^2_{\e})$.
For the true, data-generating model, we set $\sigma^2_{\epsilon}=0.1$, and set, for $i=1,\ldots,n$ and $j=1,\ldots,p=|\bS|$, 
$x_{ij}=5/j+\eta_{ij}$, where $\eta_{ij}\stackrel{iid}{\sim}N(0,\sigma^2_{\eta})$, with $\sigma^2_{\eta}=0.1$.
For generating the data, we randomly select a subset $\bs$ from the set $\bS$ associated with $p$ covariates, construct $\bx_{i,\bs}$ and simulate the elements of the regression 
coefficient vector $\bbeta_{\bs}$ independently from $N(0,\sigma^2_{b})$, with $\sigma^2_{b}=5$. We also consider an intercept $\alpha$ in our data-generating model,
which we simulate as $\alpha\sim N(\mu_{\alpha},\sigma^2_{\alpha})$, with $\mu_{\alpha}=1$ and $\sigma^2_{\alpha}=0.1$. 
Abusing notation for convenience, we shall assume that $\alpha$ is the first element of $\bbeta_{\bs}$ and that the vector of ones is the first column of 
the design matrix $X_{\bs}$. With this setup, we then generate the data
from the resulting true regression model.

For data generation, we consider three scenarios. Setting $p=10,20,30$, we generate $n=25,25,35$ data-points for the respective values of $p$. We repeat the data-generation
procedure $1000$ times for each pair $(p,n)$, so that for every $(p,n)$, we have $1000$ datasets, each consisting of $n$ data-points and a random subset of covariates
selected from the possible $p$ covariates. For each of the $1000$ simulated datasets, we attempt to select the best subset of covariates using Bayes factor obtained
through TTMCMC. The Bayesian model and prior specifications that we used for the purpose is detailed next. 

\subsubsection{Bayesian linear regression model and prior specification for variable selection using TTMCMC and Bayes factors}
Then assuming that the model for the simulated data $y_i$ is normal linear regression (with intercept) on an unknown subset of covariates of the complete
set of $p$ covariates, and with all parameters unknown, we attempt to select the best subset of covariates, using our TTMCMC algorithm (Algorithm \ref{algo:ttmcmc}) 
and Bayes factors resulting from TTMCMC, as detailed in Section \ref{sec:bf_comp}. 
For the prior on $\bbeta_{\bs}$, we consider the same form of Zellner's $g$ prior considered in Section \ref{sec:linear_regression} of MB; 
here we assume the following equivalent form: 
\begin{equation}
	\bbeta_{\bs}\sim N\left(\bzero,\exp(\phi-g)\left( X_{\bs}^{\prime} X_{\bs}\right)^{-1} \right),
\label{eq:g_prior2}
\end{equation}
where $g$ and $\phi$ are real-valued parameters. %Note that in (\ref{eq:g_prior2}),
Rather than fixing $g$ and $\phi$, we consider them as random variables, to be updated in TTMCMC. Thus, priors are needed on these parameters.
As in the case of Zellner-Siow prior (\ctn{Zellner80}; see also \ctn{Liang08} for further discussion), we assume that {\it a priori}, $\exp(g)\sim Gamma(1/2,n/2)$, 
so that the log-prior for $g$ is given, after ignoring an additive constant, by
\begin{equation}
	\log\pi(g)=-\frac{n}{2}\exp(g)+\frac{g}{2}.
	\label{eq:logprior_g}
\end{equation}
We also assume that 
\begin{equation}
	\pi(\phi)\propto 1.
	\label{eq:prior_phi}
\end{equation}
As regards the prior for $\sigma^2_{\e}$, we re-parameterize this as $\exp(-\tau)$, and assume that $\exp(\tau)\sim Gamma(a_{\tau},b_{\tau})$, so that the log-prior,
after ignoring an additive constant, is given by
\begin{equation}
	\log\pi(\tau)=-b_{\tau}\exp(\tau)+a_{\tau}\tau.
	\label{eq:logprior_tau}
\end{equation}
We set $a_{\tau}=b_{\tau}=0.01$.

We put a discrete normal prior on $k=|\bs|$, given by
\begin{equation}
	\pi(k)\propto\exp\left\{-\frac{1}{2\sigma^2_k}(k-\mu_k)^2\right\};~k=1,2,\ldots,p.
	\label{eq:prior_k}
\end{equation}
Note that although the Poisson distribution is commonly used
for specifying priors on the dimension in variable-dimensional problems, the above discrete normal prior is more flexible, since it can control both the mean and variance
of the dimensionality, unlike the Poisson prior which has the same mean and variance.

In (\ref{eq:prior_k}) we set $\mu_k=8,16,24$, respectively, when $p=10,20,30$, and fix $\sigma^2_k=1$ for all the chosen values of $p$. 
These relatively large values of $\mu_k$ with respect to $p$ are chosen to avoid the Lindley's paradox which creates the tendency among Bayes factors 
to select parsimonious models, irrespective of the truth. The variance $\sigma^2_k=1$ is expected to disallow significant drift of the dimension towards small values,
unless the data dictates so.

\subsubsection{TTMCMC implementation for Bayesian linear regression}
For TTMCMC implementation, we set $w_{b,k}=w_{d,k}=w_{nc,k}=1/3$ for all $k=2,\ldots,p-1$; for $k=1$ and $k=p$, we set $w_{d,k}=0$ and $w_{b,k}=0$, respectively.
For the latter two cases, we set $w_{b,k}=w_{nc,k}=1/2$ and $w_{d,k}=w_{nc,k}=1/2$, respectively.

We also set $\tilde p=\tilde q=1/2$, so that we select additive and multiplicative transformations with equal probabilities.
We set the scales $a_{\beta,j}=0.5$, for $j=1,\ldots,p$, and $a_{\theta,j}=0.05$ for the remaining parameters. 
However, when $p$ is as large as $20$ and $30$, we set $a_{\theta,j}=0.005$ for the remaining parameters to make the acceptance rates reasonably large.
For the no-change and 
mixing-enhancement steps, we set $c=0.01$. Recall from the discussion in Section \ref{sec:generic_ttmcmc} that the goal of this strategy 
is to improve acceptance rates of the birth moves as well as of the
no-change and mixing-enhancing moves, induced by the additive transformation. Indeed, note that with the additive transformation, 
the acceptance ratio of the birth move depends significantly
on twice $a_{\beta,j}$, so that relatively large value of $a_{\beta,j}$ would lead to higher acceptance probability. However, too large $a_{\beta,j}$ would of course
lead to increased rejection rate, since $\beta^{(t)}_j+a_{\beta,j}|\e_1|$ and $\beta^{(t)}_j-a_{\beta,j}|\e_1|$ may take the new $\bbeta$-vector too far from
the current $\bbeta^{(t)}$-vector. Thus, relatively large, but adequate choices of $a_{\beta,j}$s are necessary. This also ensures that the acceptance rate of the death
move, which depends upon inverse of $a_{\beta,j}$, is not too small.

Now, relatively large choice of $a_{\beta,j}$s would make the acceptance rates associated with the no-change and the mixing-enhancing steps induced
by the additive transformation too small, since in those steps, all the unknown quantities are updated simultaneously. To avoid this undesirable situation, 
we multiply $a_{\beta,j}$s by $c=0.01$, so that they are rendered adequately small in these steps. Detailed balance is easily seen to hold with respect to this
multiplication by $c$, in the same way as in fixed-dimensional TMCMC.

We standardize all the available covariates so that their empirical means and variances are $0$ and $1$, respectively. 
Now note that minimizing $BIC(u)$ in the linear regression case reduces to minimizing the residual sum of squares 
$\sum_{i=1}^n(y_i-\hat\bbeta'_{\bs^{(t)}}\bx_{i,\bs^{(t)}}-\hat\beta_u x_{i,u})^2$ with respect to $u\in\bS\backslash\bs^{(t)}$, 
where $(\hat\bbeta_{\bs^{(t)}},\hat\beta_u)$ is the least squares estimator associated with the current covariate index subset $\bs^{(t)}$.

For our TTMCMC implementation, we discard the first $10^4\times 150$ iterations as burn-in, and store one in every $150$ iterations in the next $5\times 10^4\times 150$
iterations, to obtain $5\times 10^4$ iterations for our Bayesian inference. We initialise our TTMCMC algorithm with only one covariate, $\{x_{i1};~=1,\ldots,n\}$.

\subsubsection{Parallelization}
Recall that for every pair $(p,n)$, $1000$ datasets are generated and TTMCMC must be implemented for variable selection in each of the $1000$ datasets. Thus,
$1000$ TTMCMC implementations are necessary for each pair $(p,n)$. For Bayesian linear regression, a single typical TTMCMC run in our $C$ code implementation 
on each core (with $2.8$ GHz CPU speed) of our VMWare (about $2$ TB memory) takes about $2$ minutes, $4$ minutes and $11$ minutes, respectively, 
for $(p=10,n=25)$, $(p=20,n=25)$ and $(p=30,n=35)$. 
Hence, for completing our simulation experiments in reasonable times, parallelization of our computations is indispensable.

Although our VMWare that we use for our current research consists of $80$ single-threaded cores, using only the best $50$ of them yields the optimum performance.
As such, using shell scripting language, we parallelise the $1000$ $C$ code based TTMCMC runs for each $(p,n)$ combination into $50$ cores, so that
$50$ TTMCMC runs are simultaneously implemented for each $(p,n)$; each core implementing only $20$ TTMCMC runs. 
This parallelization strategy allowed us to obtain the results for all our simulation experiments in very reasonable times, as is obvious
from the aforementioned timings for the single TTMCMC runs.

A typical TTMCMC run for the $(p=10,n=25)$ case yields the overall acceptance rate $0.193$, birth rate $0.064$, death rate $0.081$ and no-change rate $0.410$.
For $(p=20,n=25)$, these rates are $0.191$, $0.039$, $0.039$ and $0.496$, respectively, and for $(p=30,n=35)$, these are 
$0.234$, $0.101$, $0.101$ and $0.498$, respectively. These rates are computed on the basis of the entire TTMCMC run, not just on the stored samples.
That is, these rate computations are based on $10^4\times 150+5\times 10^4\times150=9\times 10^6$ TTMCMC realizations.

\subsubsection{Results of the linear regression simulation experiments}

After every TTMCMC run in each processor of our VMWare, we implement an $R$ code that computes $B_i$ given by (\ref{eq:B_i}), for $i=1,\ldots\tilde N$.
The $R$ code selects that set of covariates indexed by $\bs_{best}$ which corresponds to $B_{\max}=\max\{B_1,\ldots,B_{\tilde N}\}$. 
We also consider a binary vector $V=(v_1,\ldots,v_p)$, 
where, for $j=1,\ldots,p$, $v_j=1$ or $0$ accordingly as $j\in\bs_{best}$ or $j\notin\bs_{best}$. Also, let $V_0$ denote the binary vector associated with $\bs_0$,
the set of indices of the data-generating covariates. The $R$ code also computes the Hamming distance between the binary vectors $V$ and $V_0$, which, simply put, is
the total number of position-wise mismatches in the two vectors consisting of $p$ positions. Thus, the Hamming distance is zero if and only if $V=V_0$, that is, when
the best model obtained is the same as the true model. The Bayes factor of the best model against the true, data-generating model is also computed in the $R$ code
using the formula (\ref{eq:bf_ij}), provided that the true model appears in the stored TTMCMC sample. Furthermore, we also compute the rank of the true model based
on the $B_i$ values, again provided that the true model features in the stored TTMCMC sample.

For each $(p,n)$, these results for all the $1000$ TTMCMC runs are combined to yield the proportions of times the Hamming distance takes the values 
$0,1,\ldots,p-1$, among the $1000$ runs.
We also compute the average log-Bayes factor of the best model against the true model and the average rank of the true model, 
the averaging done over those TTMCMC samples which consist of the true model and Hamming distance value $r$, for $r=0,1,\ldots,p-1$. 
These results are depicted in Figure \ref{fig:linreg}. Panel (a) of the figure shows that for $p=10,n=25$, the Hamming distance gives the highest probability 
(about $0.527$) to $0$, that is, the true set of covariates is selected with the highest probability, which is also significantly higher compared to the other 
values of the Hamming distance. The average log-Bayes factor, as shown in panel (b), is the highest when the Hamming distance is $5$, while for Hamming distance
$7$, $8$ and $9$, the average log-Bayes factor is not available since the true set of covariates did not appear in the TTMCMC samples in those cases.
Note that the average log-Bayes factor is not increasing with the Hamming distance, which is indeed not to be expected in general.
The average true model rank, displayed in panel (c), is increasing with the Hamming distance, but again, is unavailable for the values $7$, $8$ and $9$
of the Hamming distance since the true set of covariates has probability zero with respect to the respective TTMCMC samples.

The scenario when $p=20,n=25$, is not significantly different from the $p=10,n=25$ case. Panel (d) shows that the Hamming distance gives the highest probability $0.143$ to both
$0$ and $1$, which is again significantly higher than those for the other values. 
The highest probability is of course much less than in the corresponding $(p=10,n=25)$ scenario, which is expected, since the number of covariate subsets
to search for the true covariate subset is far greater than in the previous case. Since $n=25$ is also the same as before, the information about the true covariate set 
is not increased either. But that in spite of these issues the true covariate sets are found with the highest probability, vindicates the efficacy of our variable selection
theory and the TTMCMC based methodology.
The average log-Bayes factor is the highest when the Hamming distance is $9$
and the average true model rank is the highest for Hamming distance $7$. Note that unlike the case of $(p=10,n=25)$, the true model rank is not increasing with
the Hamming distance in this case, and this is to be expected in general

The case of $p=30,n=35$ is the most challenging situation among all the $(p,n)$ pairs considered, as searching for the true set of covariates from among a set of
$2^{30}-1=1073741823$ possible subsets is akin to looking for a needle in a haystack! Yet, as panel (g) of Figure \ref{fig:linreg} shows, 
the Hamming distance gives significant probability to $0$,
while the value $2$ gets the highest probability. This performance of our TTMCMC based Bayes factor should not be considered unsatisfactory at all. Note that the
average log-Bayes factor is the highest when the Hamming distance is $10$, but the average model rank is the worst when the Hamming distance is as small as $5$, relative to
$p-1=29$. In other words, even for Hamming distance $5$, there are many models that perform better than the true model on a average, in terms of Bayes factor.
Since for most of the larger values of the Hamming distance the true sets of covariates have probabilities zero with respect to TTMCMC, it is clear that for most values
of the Hamming distance, the true, data-generating model is outperformed by the other models. It must also be remarked that for a limited TTMCMC sample size, 
reliably measuring the performances of many important models among a set of such a huge number of models, is infeasible.
\begin{figure}
	\centering
	\subfigure [$p=10,n=25$.]{ \label{fig:a10}
	\includegraphics[width=4.5cm,height=4.5cm]{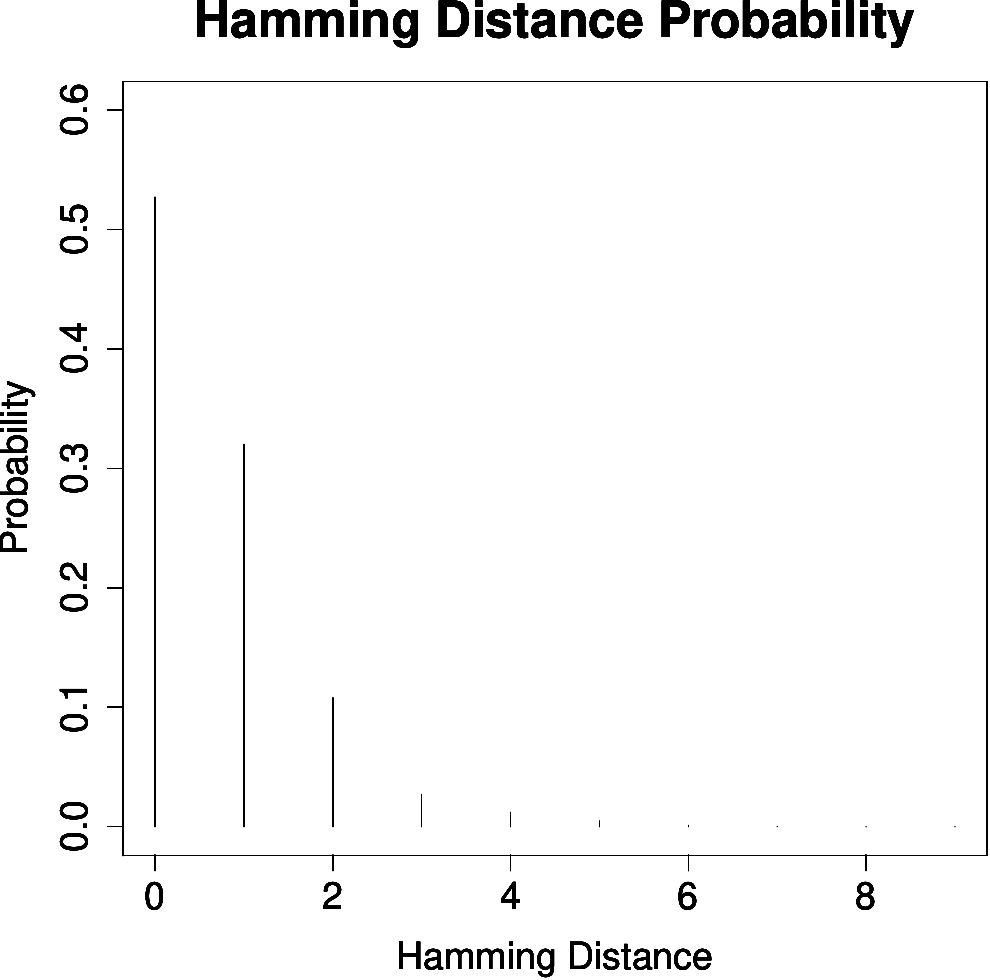}}
	\hspace{2mm}
	\subfigure [$p=10,n=25$.]{ \label{fig:b10}
	\includegraphics[width=4.5cm,height=4.5cm]{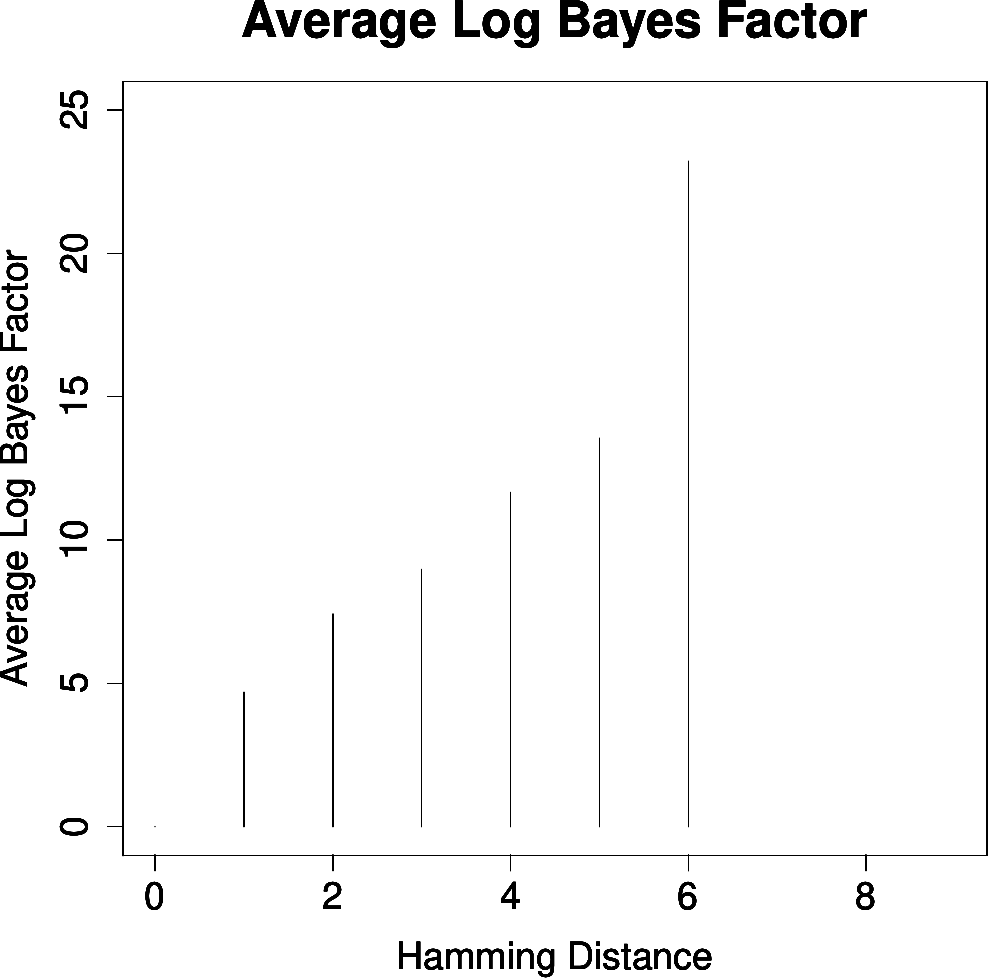}}
	\hspace{2mm}
	\subfigure [$p=10,n=25$.]{ \label{fig:c10}
	\includegraphics[width=4.5cm,height=4.5cm]{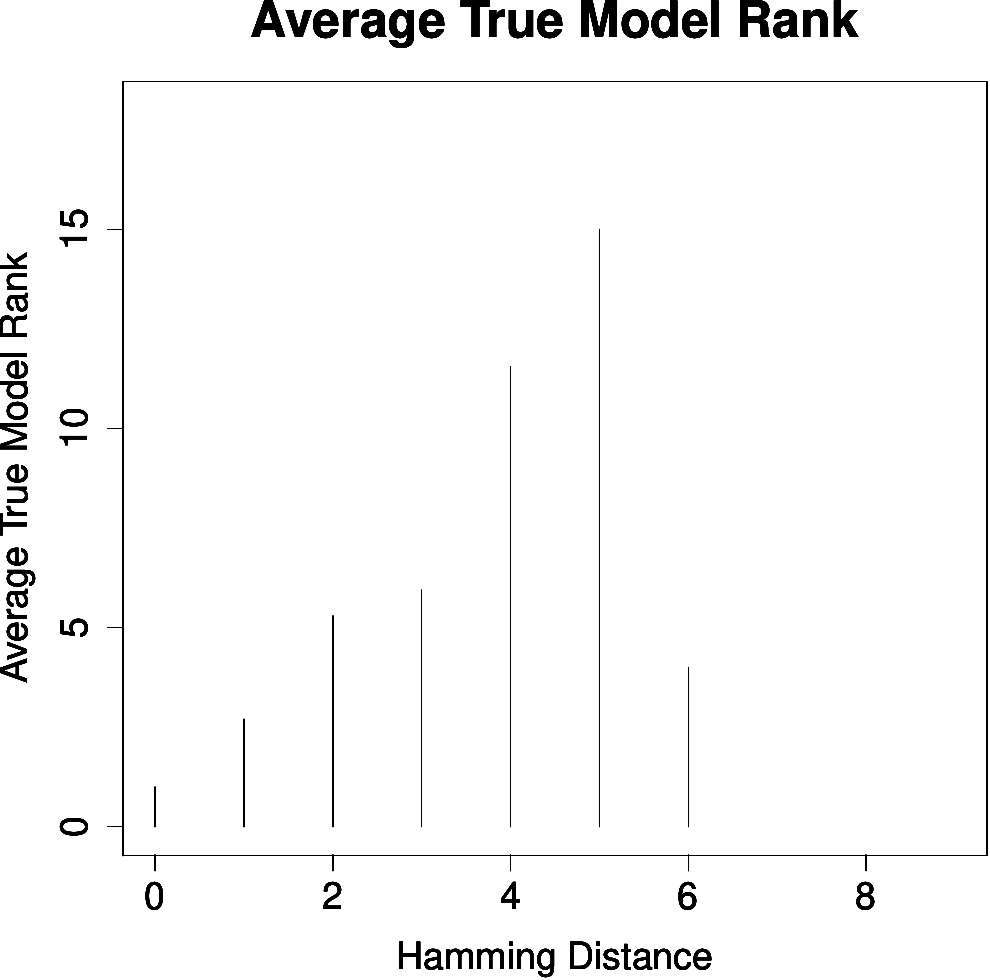}}\\
	\vspace{2mm}
	\subfigure [$p=20,n=25$.]{ \label{fig:a20}
	\includegraphics[width=4.5cm,height=4.5cm]{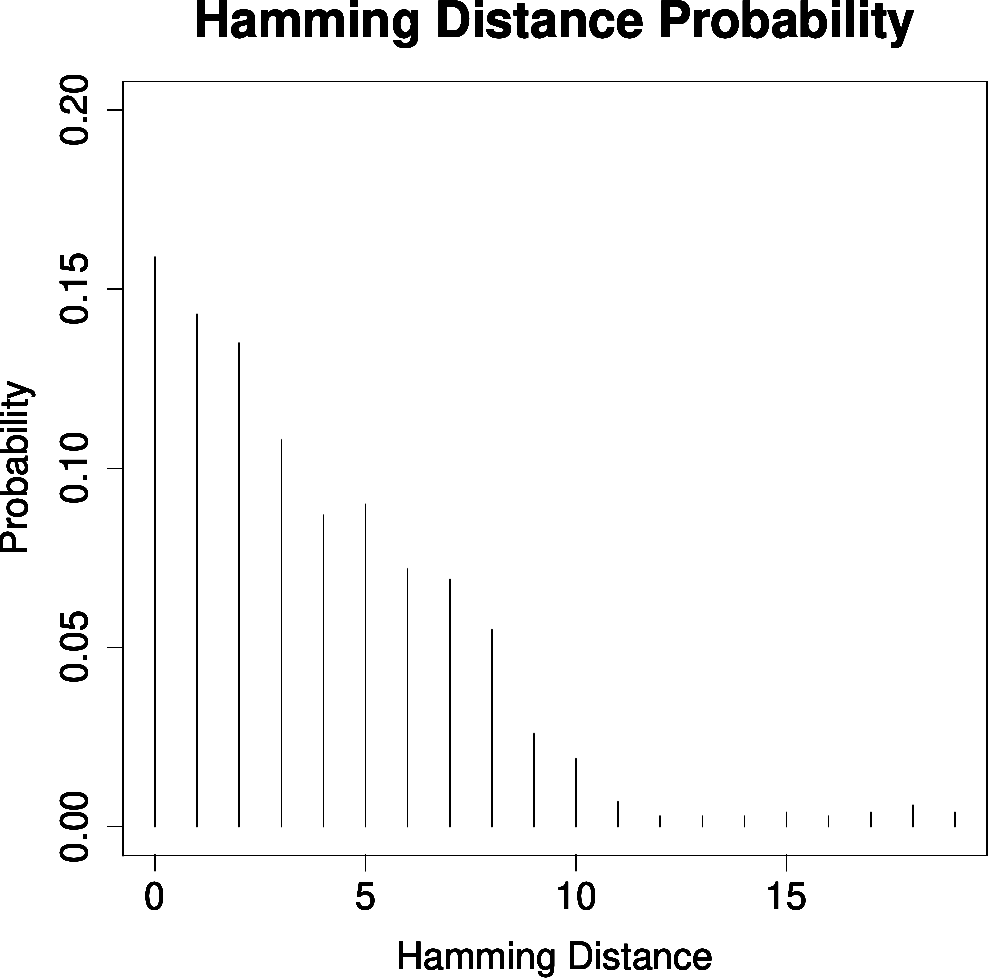}}
	\hspace{2mm}
	\subfigure [$p=20,n=25$.]{ \label{fig:b20}
	\includegraphics[width=4.5cm,height=4.5cm]{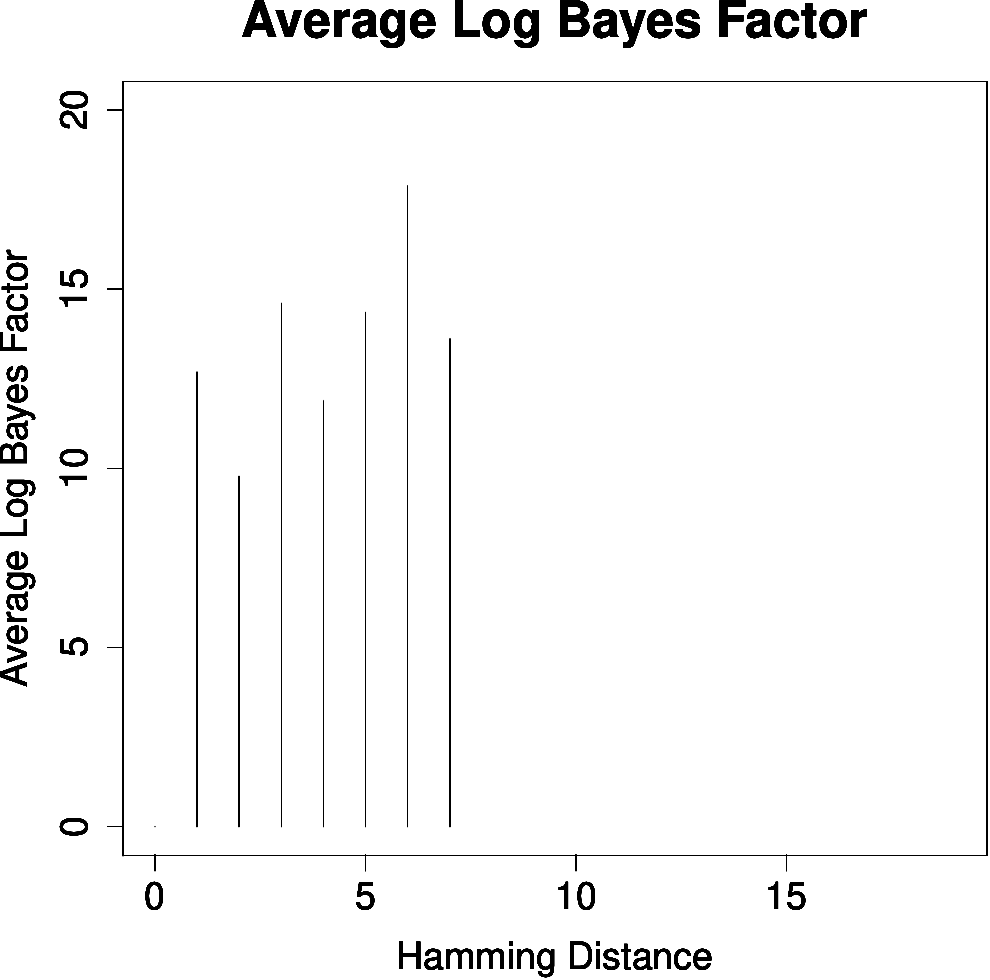}}
	\hspace{2mm}
	\subfigure [$p=20,n=25$.]{ \label{fig:c20}
	\includegraphics[width=4.5cm,height=4.5cm]{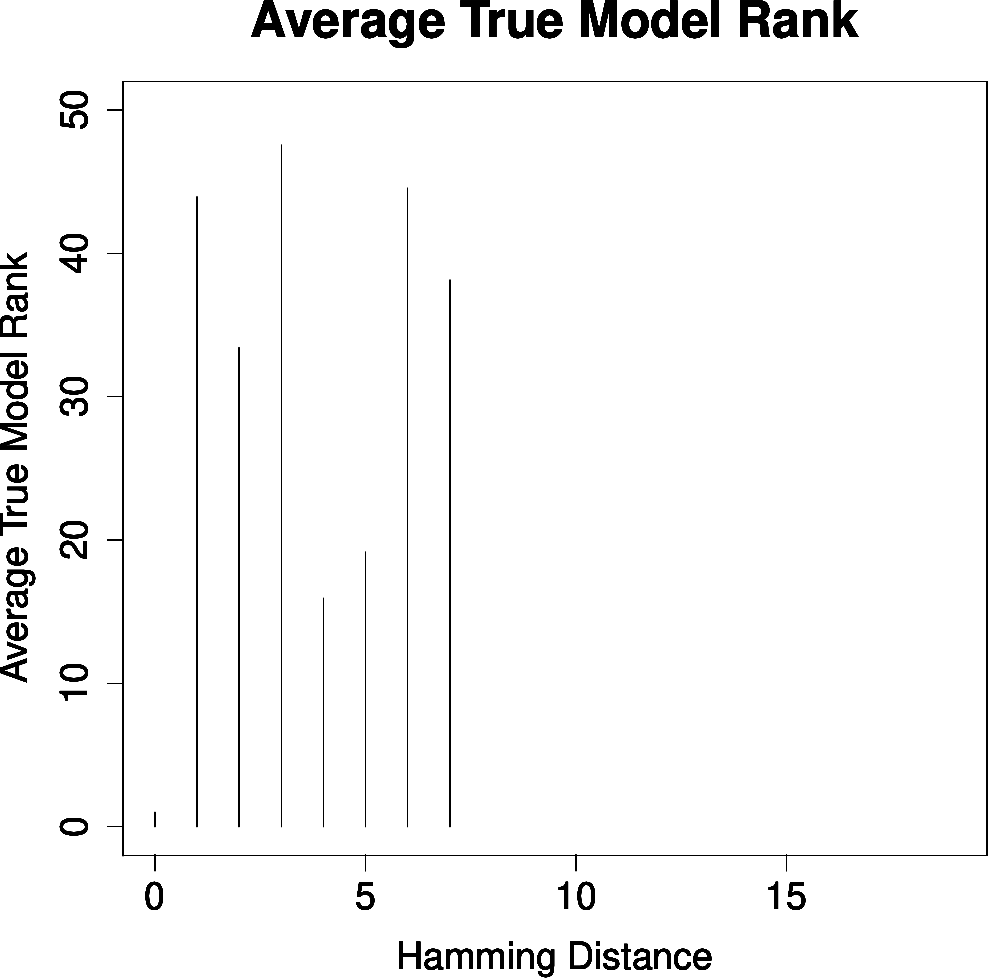}}\\
	\vspace{2mm}
	\subfigure [$p=30,n=35$.]{ \label{fig:a30}
	\includegraphics[width=4.5cm,height=4.5cm]{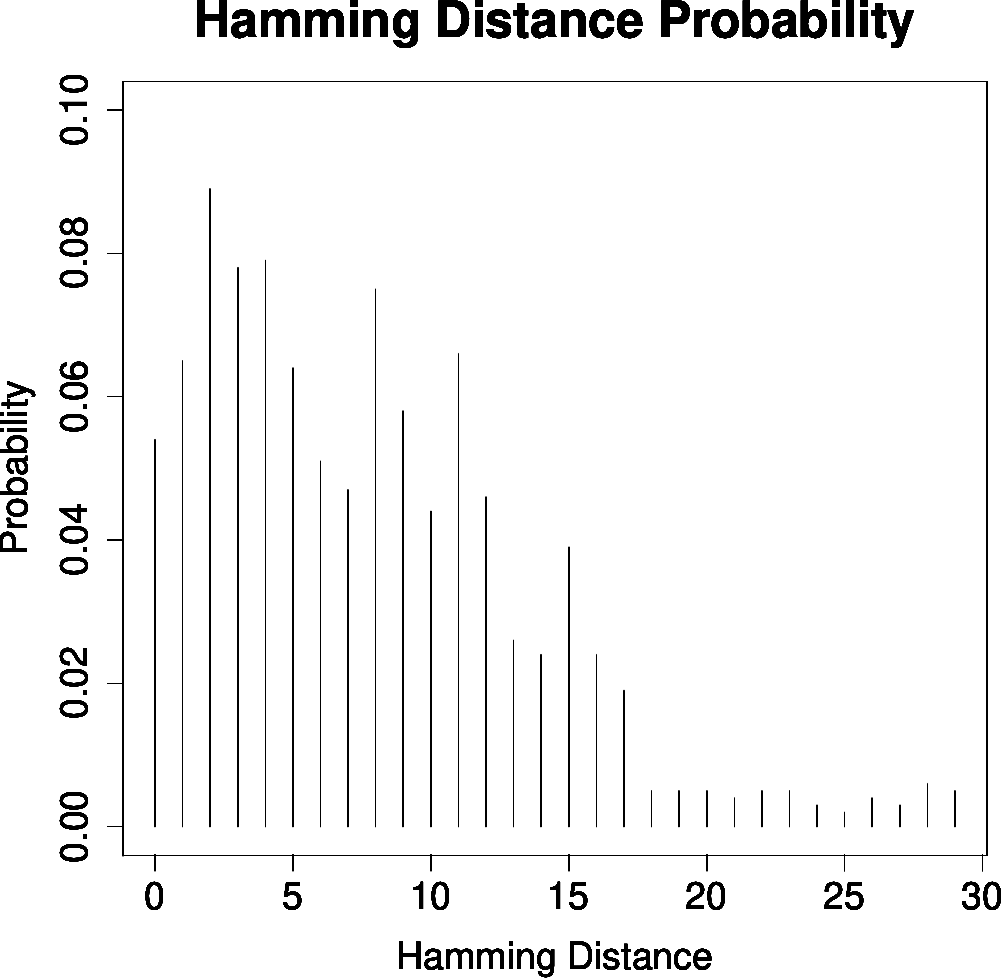}}
	\hspace{2mm}
	\subfigure [$p=30,n=35$.]{ \label{fig:b30}
	\includegraphics[width=4.5cm,height=4.5cm]{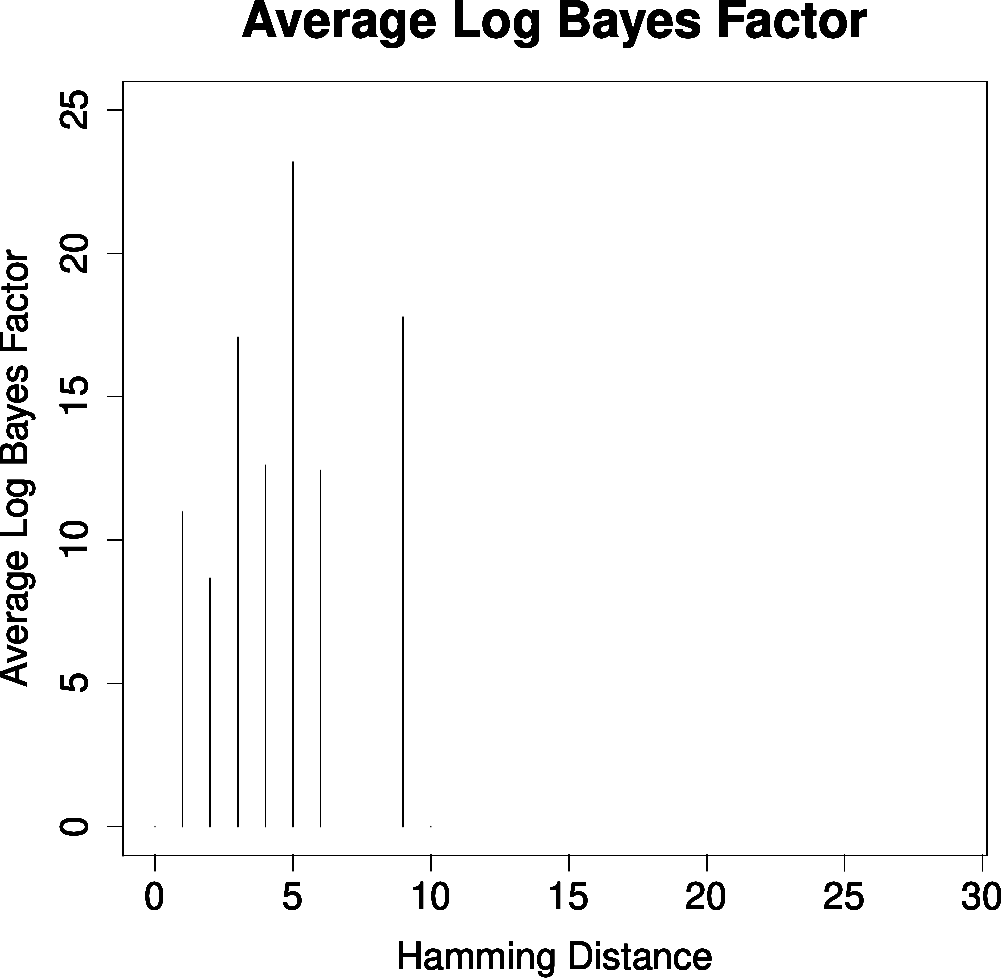}}
	\hspace{2mm}
	\subfigure [$p=30,n=35$.]{ \label{fig:c30}
	\includegraphics[width=4.5cm,height=4.5cm]{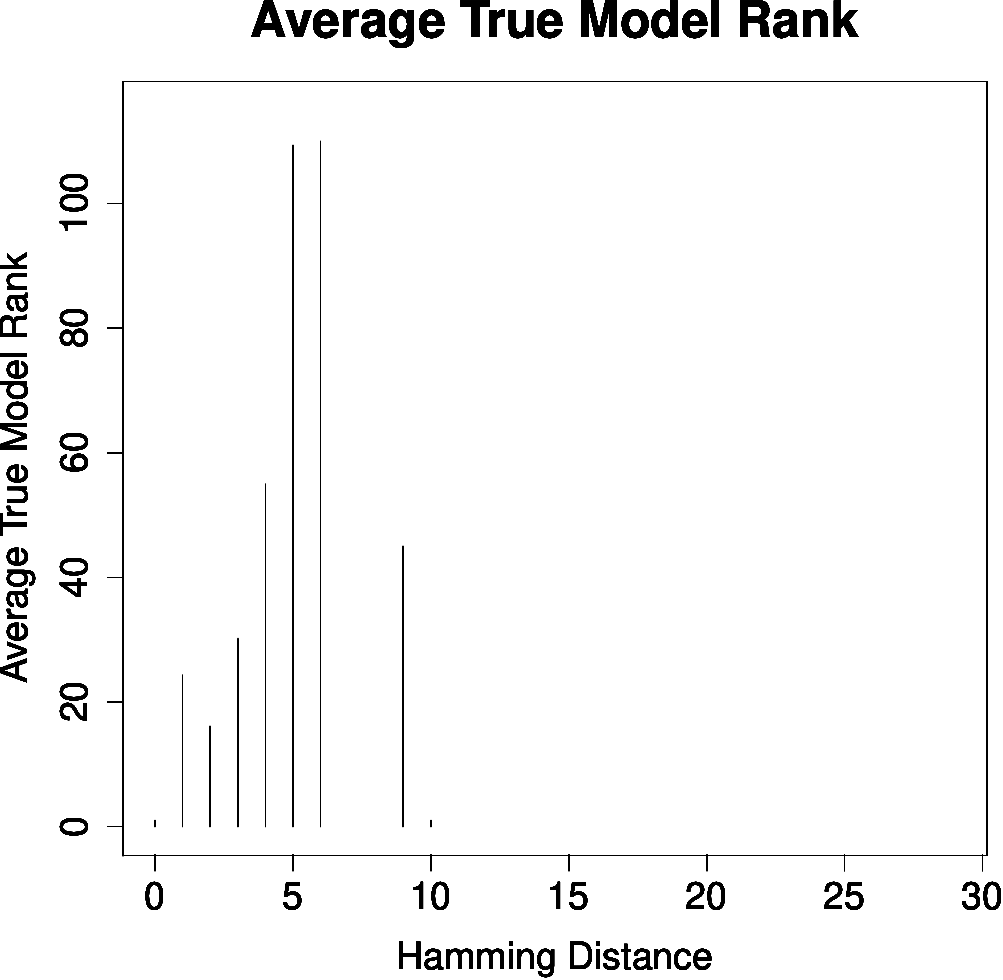}}
	\caption{Simulation study: Bayesian linear regression variable selection results.}
	\label{fig:linreg}
\end{figure}

\subsection{Gaussian process regression}
\label{subsec:gpreg_simstudy}

We now consider simulation experiments with Gaussian process regression as described in Section \ref{sec:gp_illustration} of MB.
That is, now the model that we consider is of the form $y=f(\bx_{\bs})+\epsilon$, where $f(\cdot)$ is modeled by a Gaussian process with
mean function $\mu\left(\bx_{\bs}\right)$ and squared exponential covariance kernel of the form %(\ref{eq:var:rkhs}) given in MB.
\begin{eqnarray*}
 Cov\left(f(\bx_\bs), f(\bx_\bs^{\prime}) \right) = \sigma_f^2 \exp \left\{-\frac{1}{2} \left(\bx_\bs-\bx^{\prime}_\bs\right)^T 
	D_\bs \left(\bx_\bs-\bx^{\prime}_\bs\right)  \right\},
%\label{var:rkhs}
 \end{eqnarray*}
where $\sigma_f^2$ is the process variance and the diagonal elements of $D_{\bs}$
are the smoothness parameters.

Here the data are modeled as $y_i=f(\bx_{i,\bs})+\epsilon_i$, where, for $i=1,\ldots,n$, $\epsilon_i\stackrel{iid}{\sim}N(0,\sigma^2_{\e})$.
As before, we reparameterize $\sigma^2_\e$ as $\exp(-\tau)$; we also reparameterize $\sigma^2_f$ as $\exp(-\tau_f)$, where $\tau$ and $\tau_f$ are real parameters.
But unlike the linear regression case, here we assume that {\it a priori}, $\tau\sim N(\mu_{\tau},\sigma^2_{\tau})$ and $\tau_f\sim N(\mu_{\tau_f},\sigma^2_{\tau_f})$,
with $\mu_{\tau}=\mu_{\tau_f}=0$, $\sigma^2_{\tau}=0.5$ and $\sigma^2_{\tau_f}=0.1$, the variances reflecting the belief that uncertainty about the process variance
is less than that of the noise variance.

For $i=1,\ldots,|\bs|$, we reparameterize the $i$-th diagonal element of $D_{\bs}$ as $\exp(-\gamma_i)$, where we assume {\it a priori} that 
$\gamma_i\stackrel{iid}{\sim}N(0,\sigma^2_{\gamma})$, with $\sigma^2_{\gamma}=2$.
%We denote the $n$-dimensional mean vector of model $\bs$ by $\bmu_{\bs}$ and the covariance matrix by $\Sigma_{\bs}$, as before.

We model the mean function $\mu\left(\bx_{\bs}\right)$ as linear regression containing the intercept, that is, we set $\mu\left(\bx_{\bs}\right)=\bbeta'_{\bs}\bx_{\bs}$,
assuming that the first element of $\bx_{\bs}$ is $1$.
We consider the same Zellner-Siow prior form for $\bbeta_{\bs}$ as in Section \ref{subsec:linreg_simstudy}. 
As before, we standardize all the available covariates for model implementation with TTMCMC. 

Letting $\bgamma_{\bs}$ denote the vector of smoothness parameters, note that $\bbeta_{\bs}$ and $\bgamma_{\bs}$ are both variable-dimensional vectors, the dimensions of which
must be increased or decreased simultaneously. Recall that such updating provision is of course considered in our TTMCMC algorithm (Algorithm \ref{algo:ttmcmc}).
The prior for $k$ remains the same as in the linear regression setup.

The data simulation principle from the true model consisting of random sets of covariates and the formation of the covariates 
remain the same as in the linear regression case; here  
$(y_1,\ldots,y_n)$ is generated from the joint multivariate normal model dictated by the above Gaussian process setup, given $\bbeta_{\bs}$, $\bgamma_{\bs}$,
$\tau$ and $\tau_f$. We set $\tau=-\log(0.1)$ and $\tau_f=-\log(0.2)$, and simulate the elements of $\bbeta_{\bs}$ and $\bgamma_{\bs}$ from the zero mean normal 
distribution with variance $5$. 

As before, we consider the settings $(p=10,n=25)$, $(p=20,n=25)$ and $(p=30,n=35)$ for evaluating our Bayes factor based variable selection obtained via TTMCMC.
The TTMCMC algorithm in this Gaussian process setup is similar to that for linear regression, with the extra variable-dimensional parameter $\bgamma_{\bs}$
and the fixed-dimensional variable $\tau_f$ being accounted for. The procedure for updating these remain the same as before, in accordance with the details
provided in Algorithm \ref{algo:ttmcmc}.

As regards computation of $BIC(u)$ in this Gaussian process setup, we first obtain the least squares estimates $\hat\bbeta_{\bs}$ corresponding to $\bbeta_{\bs}$ 
pretending a linear regression context, and substitute $\hat\bbeta_{\bs}$ in the Gaussian process likelihood. Also, in the Gaussian process likelihood, we set
$\hat\gamma_i=0$, for $i=1,\ldots,|\bs|$, corresponding to $\bgamma_{\bs}$. Finally, we substitute the current values $\tau^{(t)}$ and $\tau^{(t)}_f$ 
for $\tau$ and $\tau_f$ in the likelihood, and proceed to compute the BIC version with these substitutions. 
As we shall demonstrate, our experiments reveal that this method yields
quite reliable propositions for the new covariates. However, computation of $BIC(u)$ is quite demanding in this setup due to the requirement of $n\times n$-order 
matrix inversions for every $u$. Thus, in this setup, for $p$ even moderately large, our TTMCMC takes considerably more implementation time than for linear regression.
Indeed, for $(p=10,n=25)$ the time taken is about $26$ minutes for a typical run, and for $(p=20,n=25)$ and $(p=30,n=35)$, the respective run times are
about $46$ minutes and $2$ hours $22$ minutes. We parallelise our simulation experiments consisting of $1000$ TTMCMC runs for each $(p,n)$ pair in the same way as in the linear
regression setup.

The overall acceptance rate, birth rate, death rate and the no-change rate for a typical TTMCMC run in the $(p=10,n=25)$ case are about
$0.182$, $0.014$, $0.014$ and $0.518$, respectively. For $(p=20,n=25)$, these numbers are $0.217$, $0.065$, $0.065$ and $0.522$, while
in the $(p=30,n=35)$ scenario, the respective rates are $0.201$, $0.046$, $0.046$ and $0.511$.
Again, these rates are computed on the basis of $9\times 10^6$ TTMCMC realizations.

\subsubsection{Results of the Gaussian process regression simulation experiments}
The results of our variable selection method in the Gaussian process setup are encapsulated in Figure \ref{fig:gpreg}. Note that although for $(p=10,n=25)$
the Hamming distance assigns significantly higher probability to $0$ compared to all the other values, in the other two more challenging scenarios $(p=20,n=25)$
and $(p=30,n=35)$, this good performance is not kept up. This observation is in line with the average log-Bayes factors and the average true model ranks, as
displayed in Figure \ref{fig:gpreg}. Thus, compared to the linear regression setup, our variable selection methods in 
the Gaussian process regression setup seems to be less robust with respect to increasing dimensions. This is, however, not unexpected due to the structured
dependence in the Gaussian process regression datasets.
\begin{figure}
	\centering
	\subfigure [$p=10,n=25$.]{ \label{fig:a10_gp}
	\includegraphics[width=4.5cm,height=4.5cm]{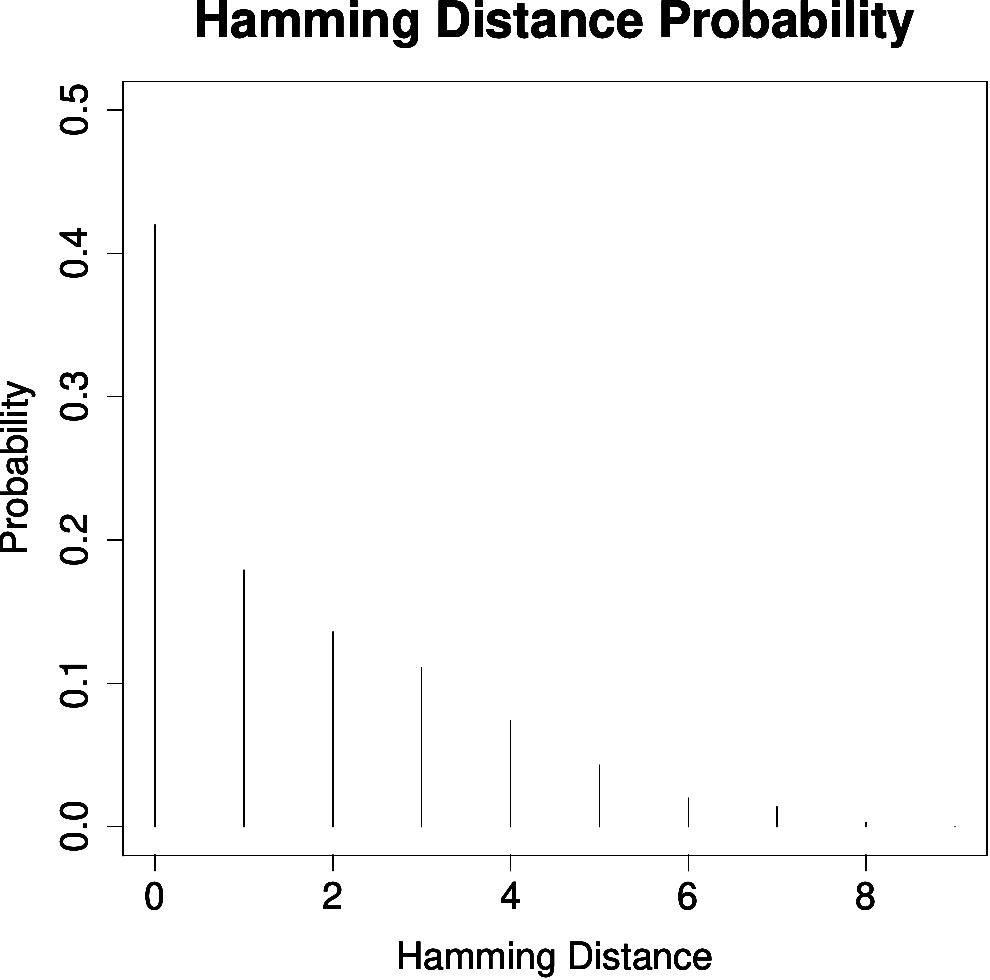}}
	\hspace{2mm}
	\subfigure [$p=10,n=25$.]{ \label{fig:b10_gp}
	\includegraphics[width=4.5cm,height=4.5cm]{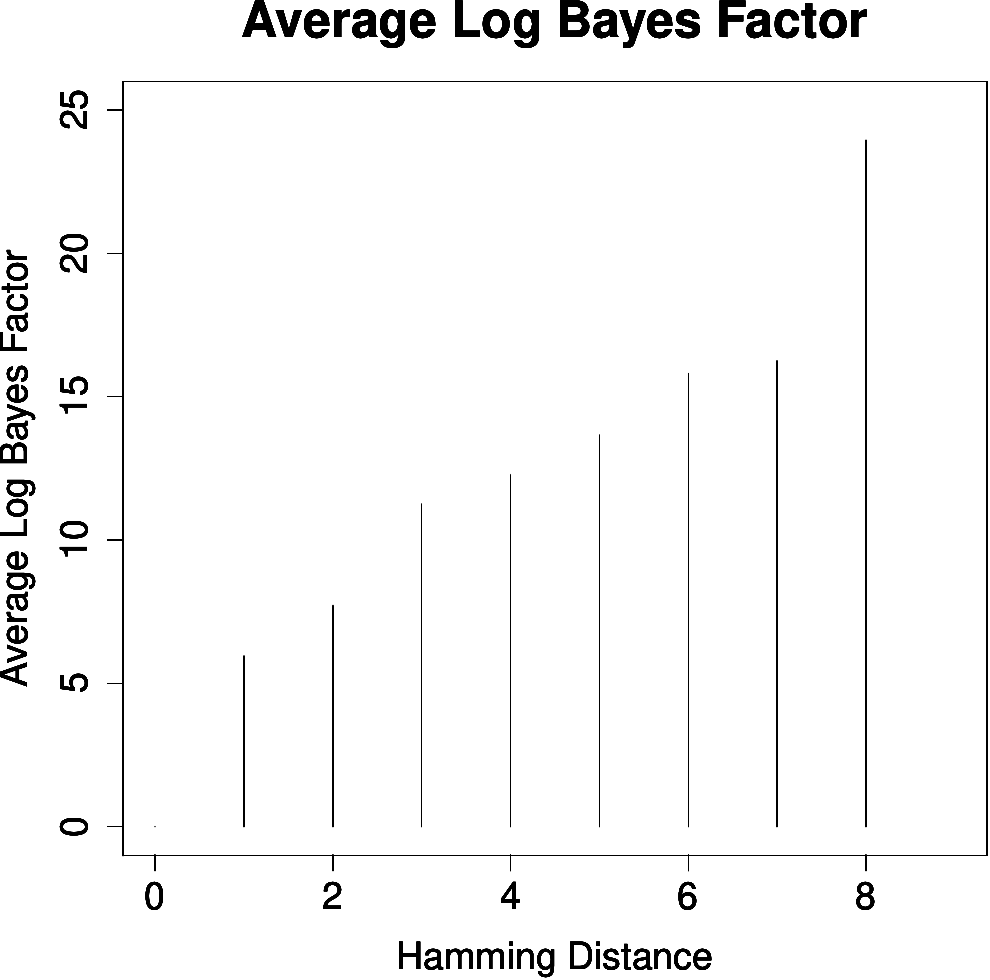}}
	\hspace{2mm}
	\subfigure [$p=10,n=25$.]{ \label{fig:c10_gp}
	\includegraphics[width=4.5cm,height=4.5cm]{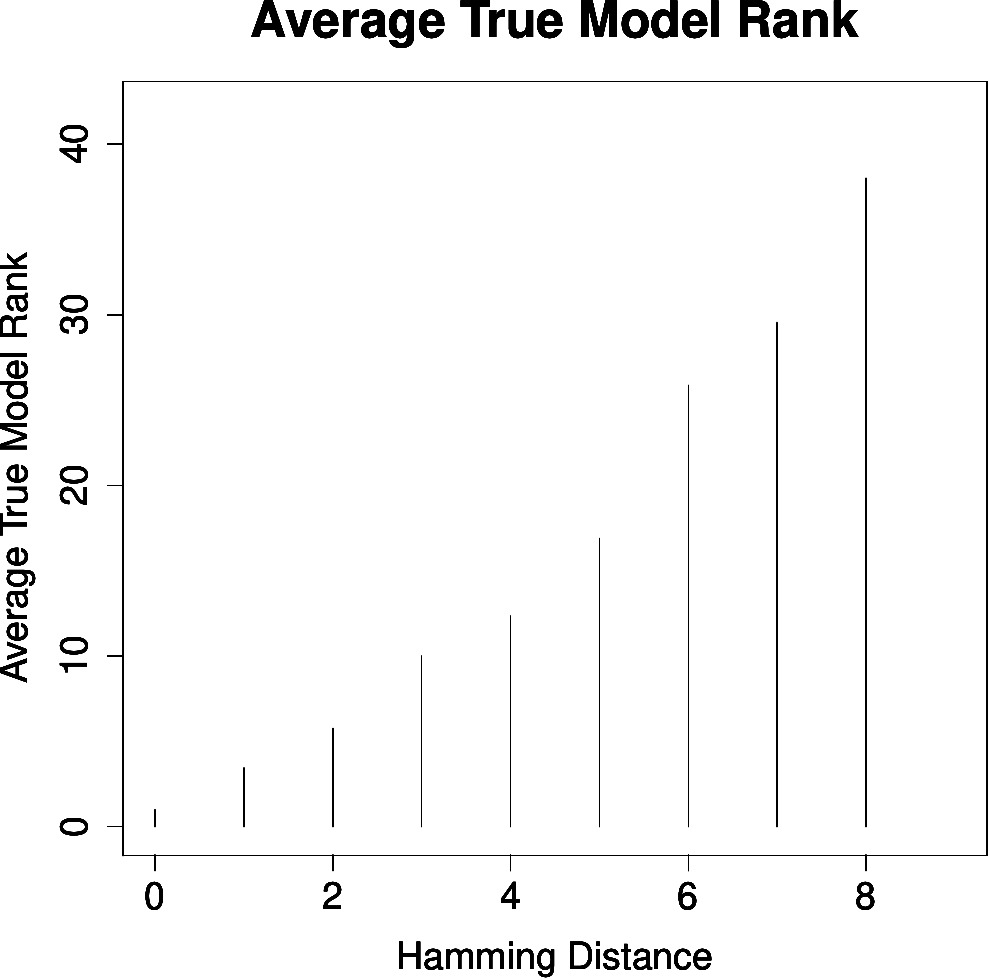}}\\
	\vspace{2mm}
	\subfigure [$p=20,n=25$.]{ \label{fig:a20_gp}
	\includegraphics[width=4.5cm,height=4.5cm]{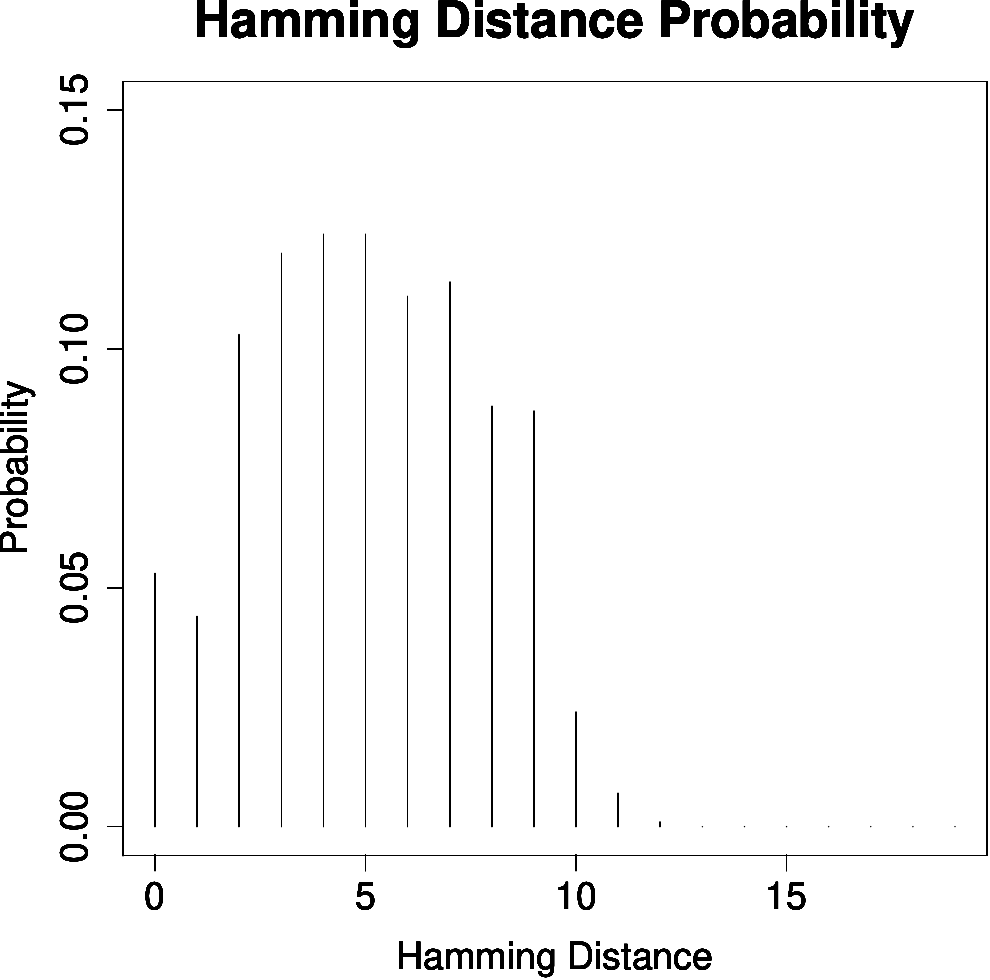}}
	\hspace{2mm}
	\subfigure [$p=20,n=25$.]{ \label{fig:b20_gp}
	\includegraphics[width=4.5cm,height=4.5cm]{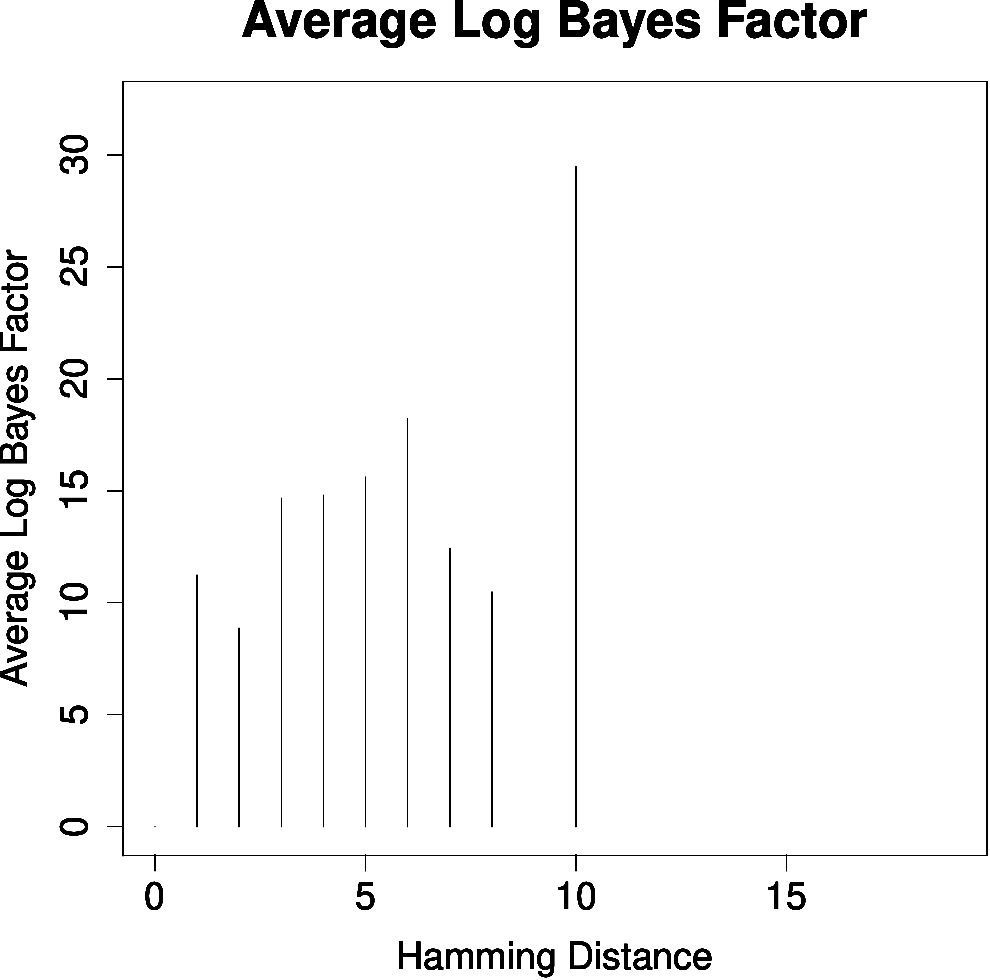}}
	\hspace{2mm}
	\subfigure [$p=20,n=25$.]{ \label{fig:c20_gp}
	\includegraphics[width=4.5cm,height=4.5cm]{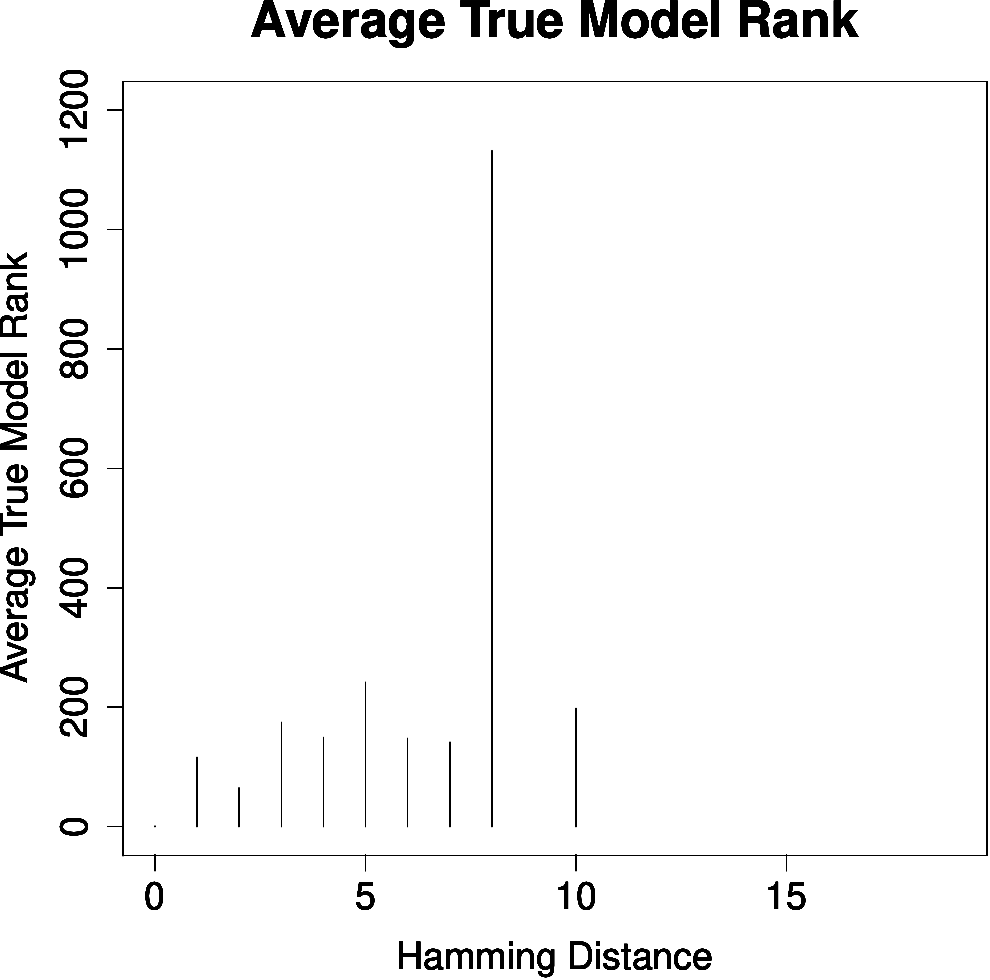}}\\
	\vspace{2mm}
	\subfigure [$p=30,n=35$.]{ \label{fig:a30_gp}
	\includegraphics[width=4.5cm,height=4.5cm]{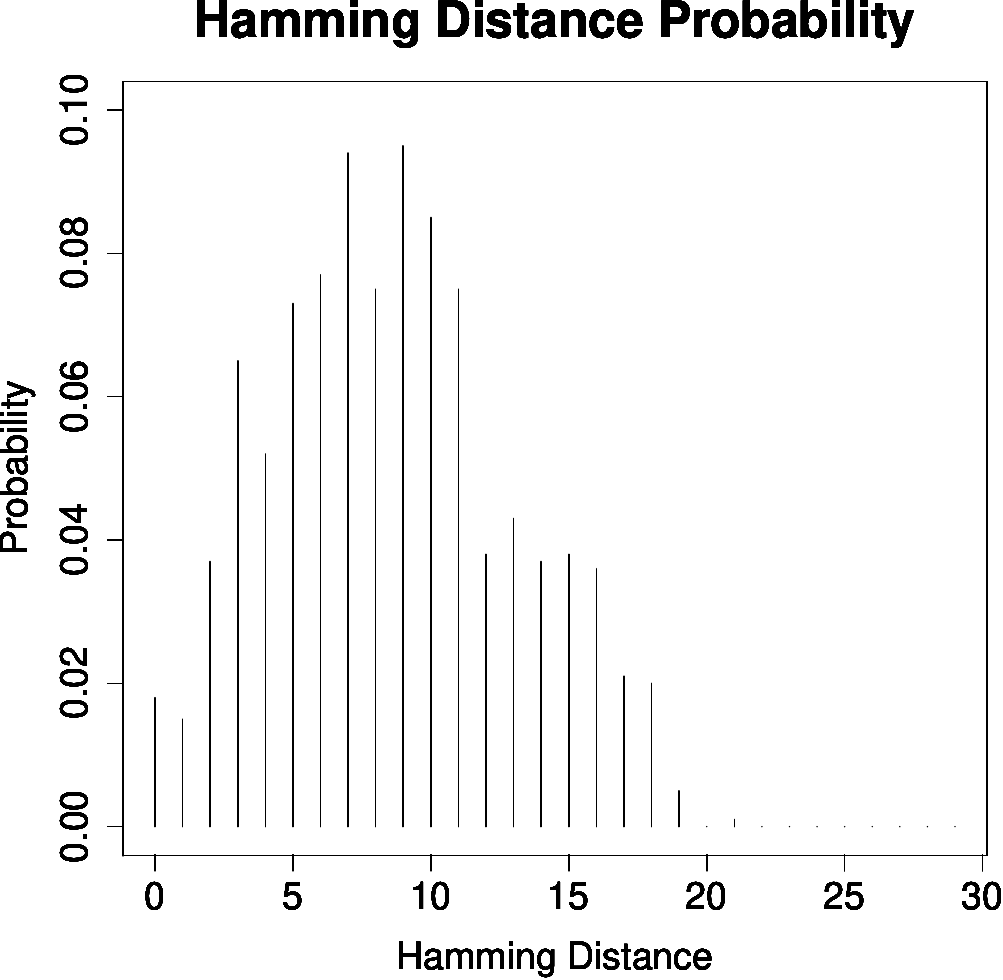}}
	\hspace{2mm}
	\subfigure [$p=30,n=35$.]{ \label{fig:b30_gp}
	\includegraphics[width=4.5cm,height=4.5cm]{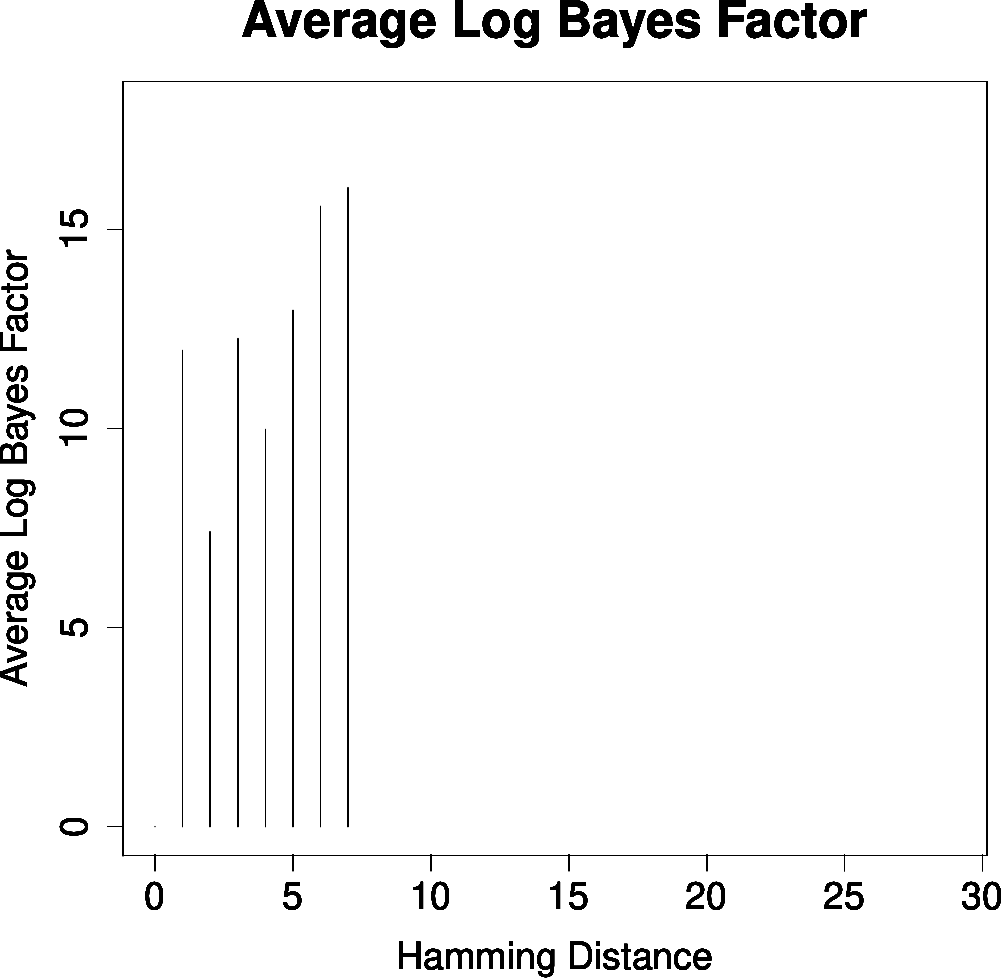}}
	\hspace{2mm}
	\subfigure [$p=30,n=35$.]{ \label{fig:c30_gp}
	\includegraphics[width=4.5cm,height=4.5cm]{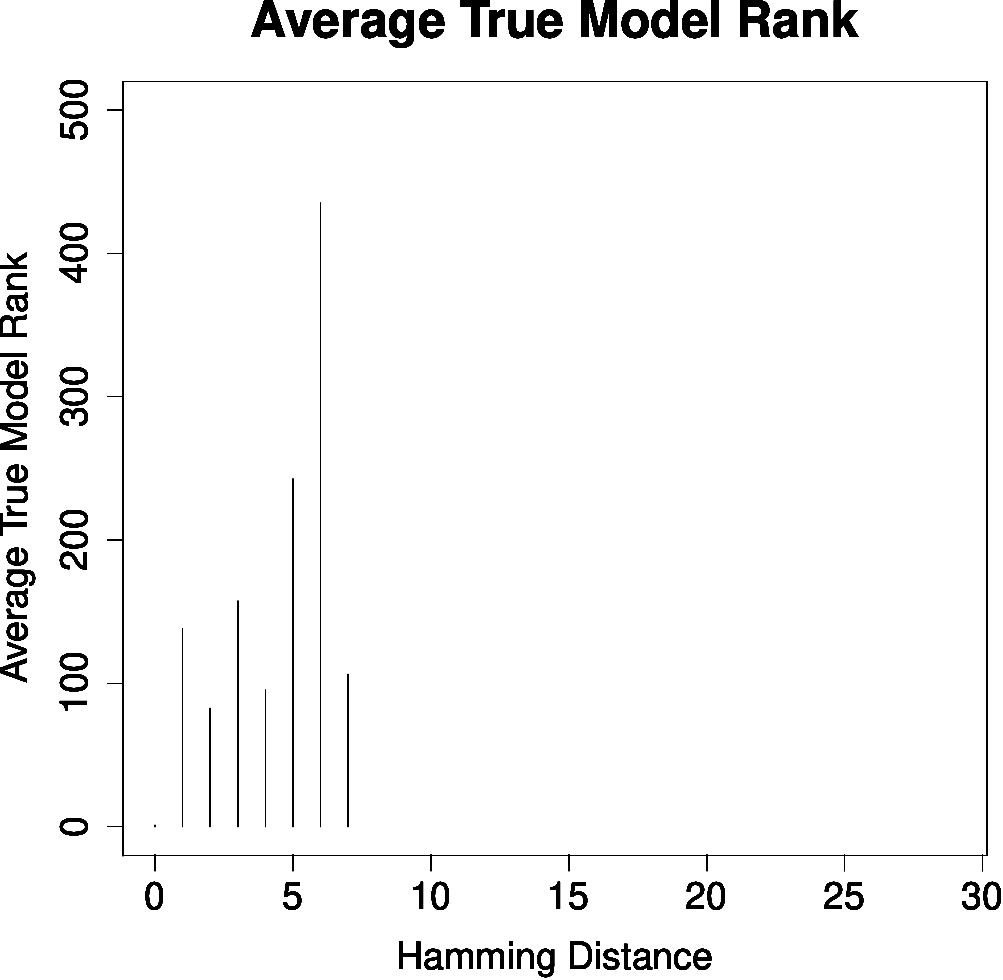}}
	\caption{Simulation study: Bayesian Gaussian process regression variable selection results.}
	\label{fig:gpreg}
\end{figure}

\subsection{AR(1) regression}
\label{subsec:tsreg_simstudy}

Now let us consider the variable selection problem in the following $AR(1)$ context, the Bayes factor asymptotics of which is detailed in Section \ref{subsec:time_series1}:
\begin{align*}
y_t=\rho y_{t-1}+\bbeta_{\bs}^{\prime} \bx_{t,\bs}+\epsilon_t,\quad \mbox{and} ~~
\epsilon_t \stackrel{iid}{\sim}N\left(0,\sigma^2_{\epsilon}\right),\quad \mbox{for}~~ t=1,\ldots,n,
%\label{eq:ar1_model}
\end{align*}
where $y_0\equiv 0$ and $|\rho|<1$.
We reparameterize $\rho$ as $\rho=-1+2\exp(\tilde\rho)/(1+\exp(\tilde\rho))$, and assume that $\tilde\rho\sim N(0,1)$. As before we consider the
Zellner-Siow prior for $\bbeta_{\bs}$ and reparameterize $\sigma^2_{\e}$ as $\exp(-\tau)$ and use the log-prior form (\ref{eq:logprior_tau}), with
$\alpha_{\tau}=\lambda_{\tau}=0.01$. For data generation, we generate $\rho$ from $U(-1,1)$.

To form the likelihood, we considered the product $\prod_{t=1}^n[y_t|y_{t-1}]$, where $[y_t|y_{t-1}]$ stands for the distribution (\ref{eq:ar1_model}) of MB.
We did not consider the correlated error form 
$y_t=\bbeta_{\bs}^{\prime} \bz_{t,\bs}+\tilde\epsilon_t$ considered in Section \ref{subsec:time_series1} for likelihood formation as this would involve 
a multivariate normal distribution which would require $n\times n$ matrix inversions in each step of TTMCMC.

For $BIC(u)$, first note that given the value of $\rho$ at the current TTMCMC iteration, $y_1$ and $y_t-\rho y_{t-1}$ for $t\geq 2$ has a linear regression form, using which we
compute the least squares estimator of the regression coefficients. The rest of the $BIC(u)$ minimization procedure is the same as in the linear regression setup.

Indeed, the remaining methodological and implementation details are also akin to the linear regression situation. 
In this case, typical TTMCMC runs for $(p=10,n=25)$, $(p=20,n=25)$ and $(p=30,n=35)$ took about $9$ minutes, $14$ minutes and $36$ minutes, respectively.
The overall acceptance rate, birth rate, death rate and no-change rates in the first case was about $0.553$, $0.552$, $0.564$ and $0.544$, respectively.
In the second and third scenarios they were $(0.582,0.602,0.601,0.543)$ and $(0.580,0.605,0.605,0.533)$, respectively.

\subsubsection{Results of the AR(1) regression simulation experiments}

Figure \ref{fig:tsreg} displays the results of our variable selection experiments in the AR(1) context. Note that for $(p=10,n=25)$, the Hamming distance
gives the highest probability to $0$, which is also significantly higher than those for the other values. However, as in the Gaussian process regression
experiments, here also the situation deteriorates for 
$(p=20,n=25)$ and $(p=30,n=35)$, as the Hamming distance concentrates around larger and larger values for the latter two scenarios.
This gradual worsening of the performance is also reflected in the respective average log-Bayes factors and the average true model ranks, demonstrating that
compared to linear regression, variable selection in time series regression is a much more delicate problem, with marked sensitivity with respect to larger dimensions.

At the first glance, this lack of robustness with respect to dimension might seem surprising since the structure of AR(1) regression closely resembles that of
linear regression. However, recall from Section \ref{subsec:time_series1} of MB that although the data $y_t$ can be written in a linear regression form with modified
covariate structures involving $\rho$, the regression errors in such as case are correlated, rendering the AR(1) setup a Gaussian process structure. Hence, it 
is not surprising that our AR(1) regression results are much more in resemblance with our Gaussian process regression results, as compared to the linear regression
results, when $p=20$ and $30$.

\begin{figure}
	\centering
	\subfigure [$p=10,n=25$.]{ \label{fig:a10_ts}
	\includegraphics[width=4.5cm,height=4.5cm]{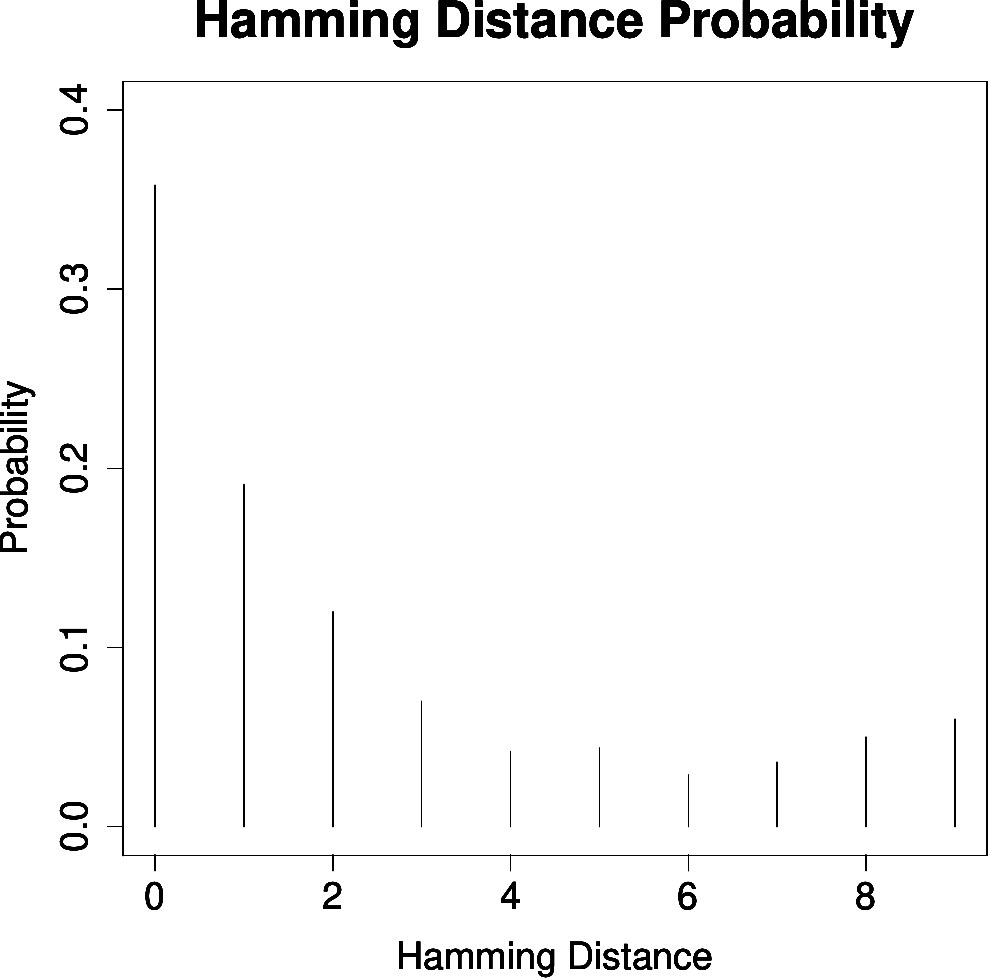}}
	\hspace{2mm}
	\subfigure [$p=10,n=25$.]{ \label{fig:b10_ts}
	\includegraphics[width=4.5cm,height=4.5cm]{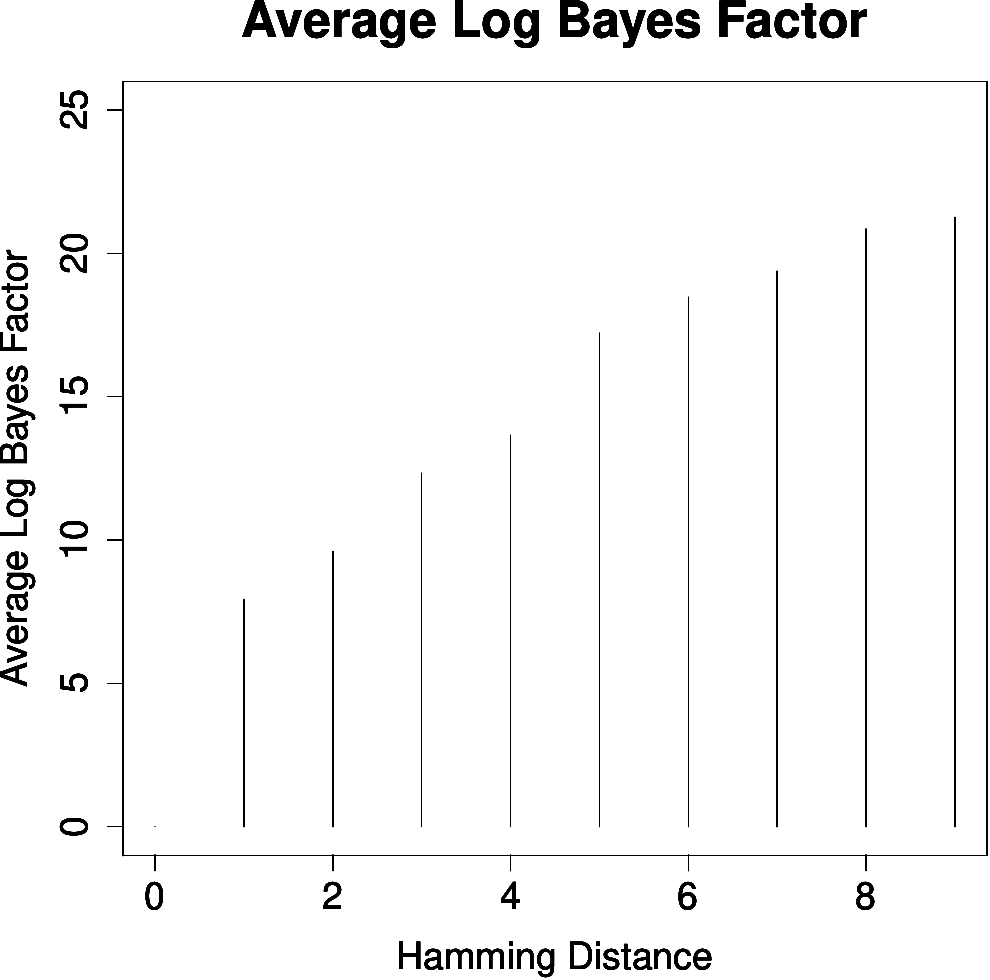}}
	\hspace{2mm}
	\subfigure [$p=10,n=25$.]{ \label{fig:c10_ts}
	\includegraphics[width=4.5cm,height=4.5cm]{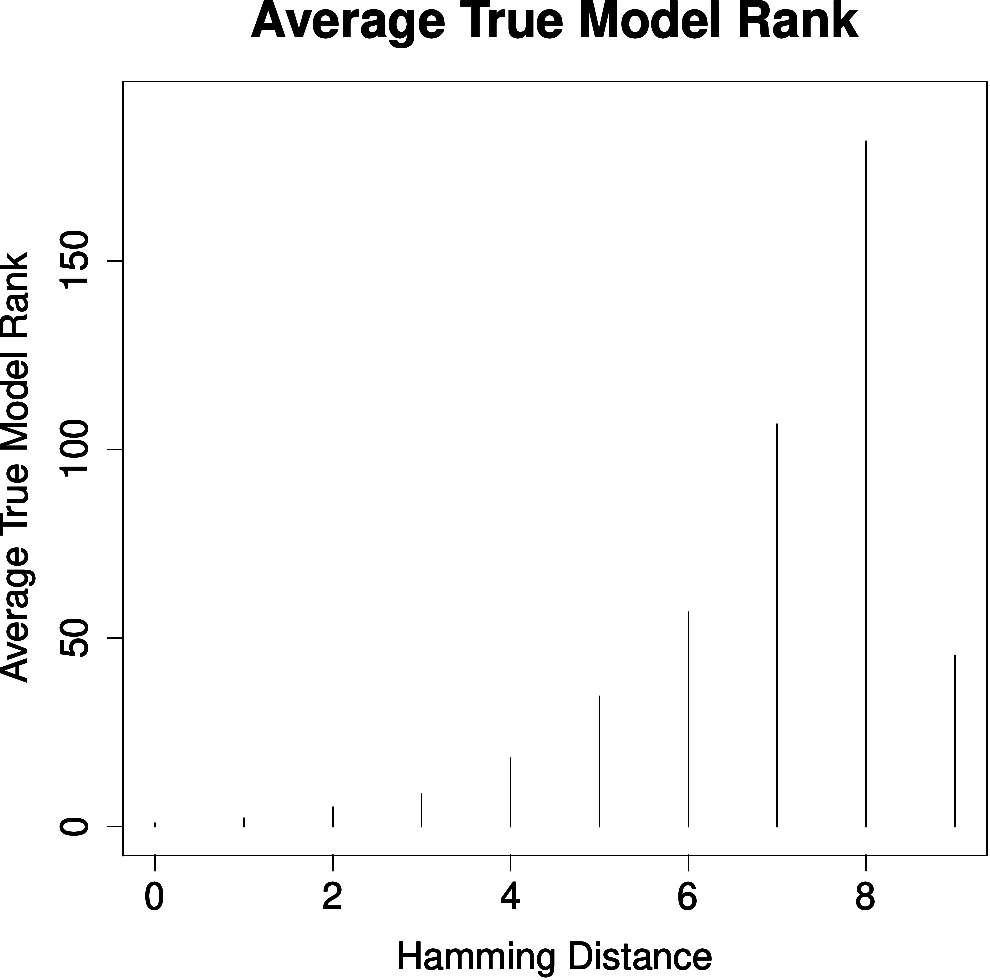}}\\
	\vspace{2mm}
	\subfigure [$p=20,n=25$.]{ \label{fig:a20_ts}
	\includegraphics[width=4.5cm,height=4.5cm]{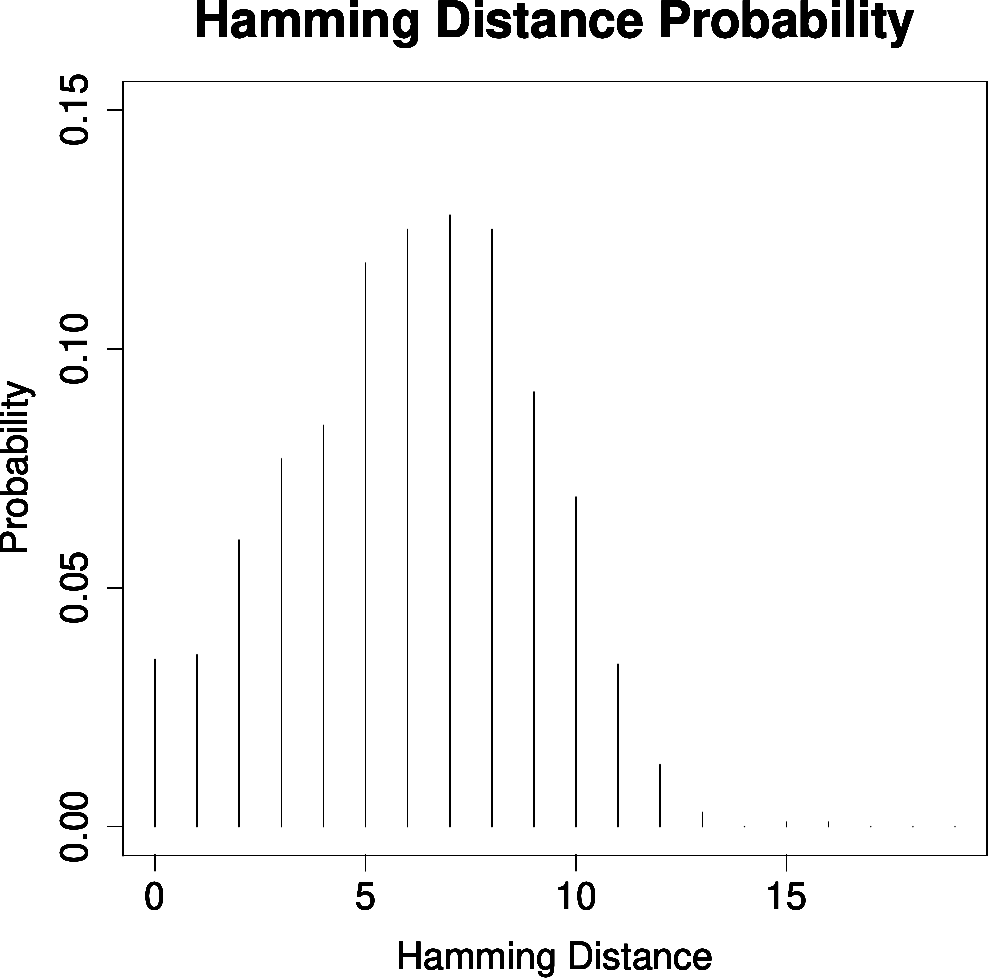}}
	\hspace{2mm}
	\subfigure [$p=20,n=25$.]{ \label{fig:b20_ts}
	\includegraphics[width=4.5cm,height=4.5cm]{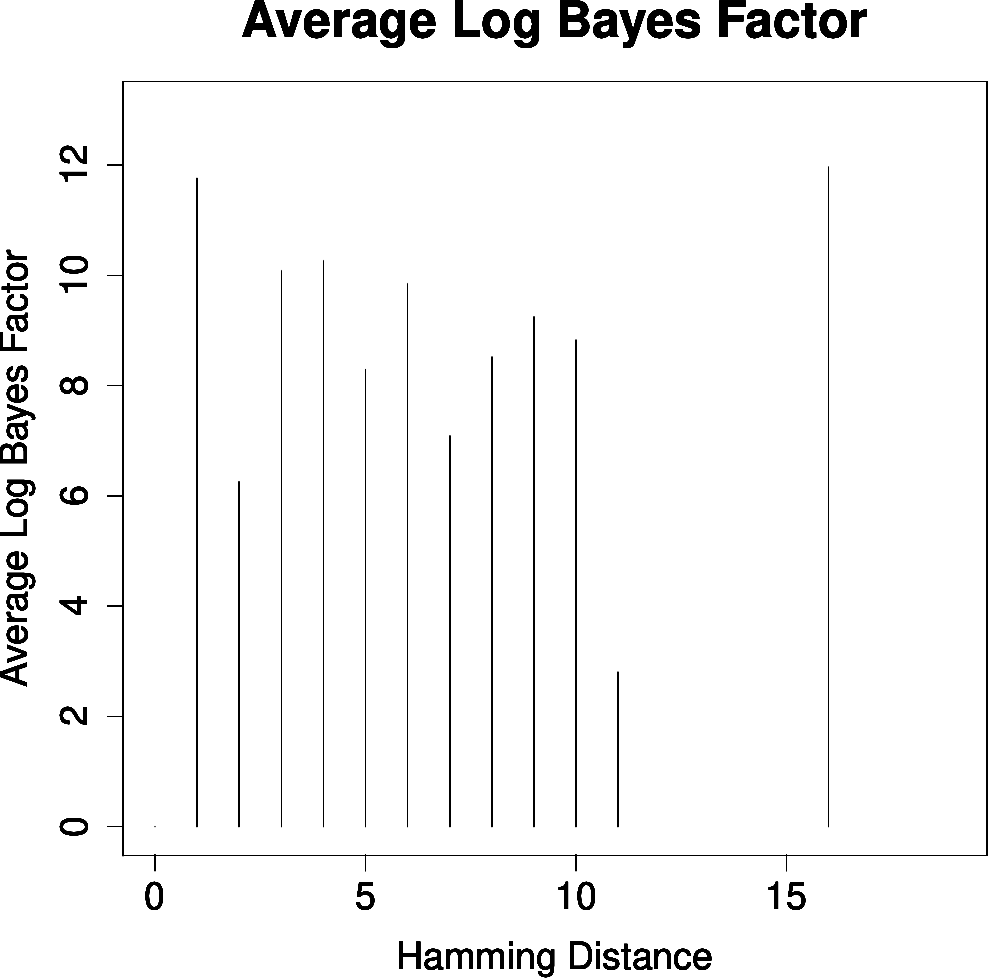}}
	\hspace{2mm}
	\subfigure [$p=20,n=25$.]{ \label{fig:c20_ts}
	\includegraphics[width=4.5cm,height=4.5cm]{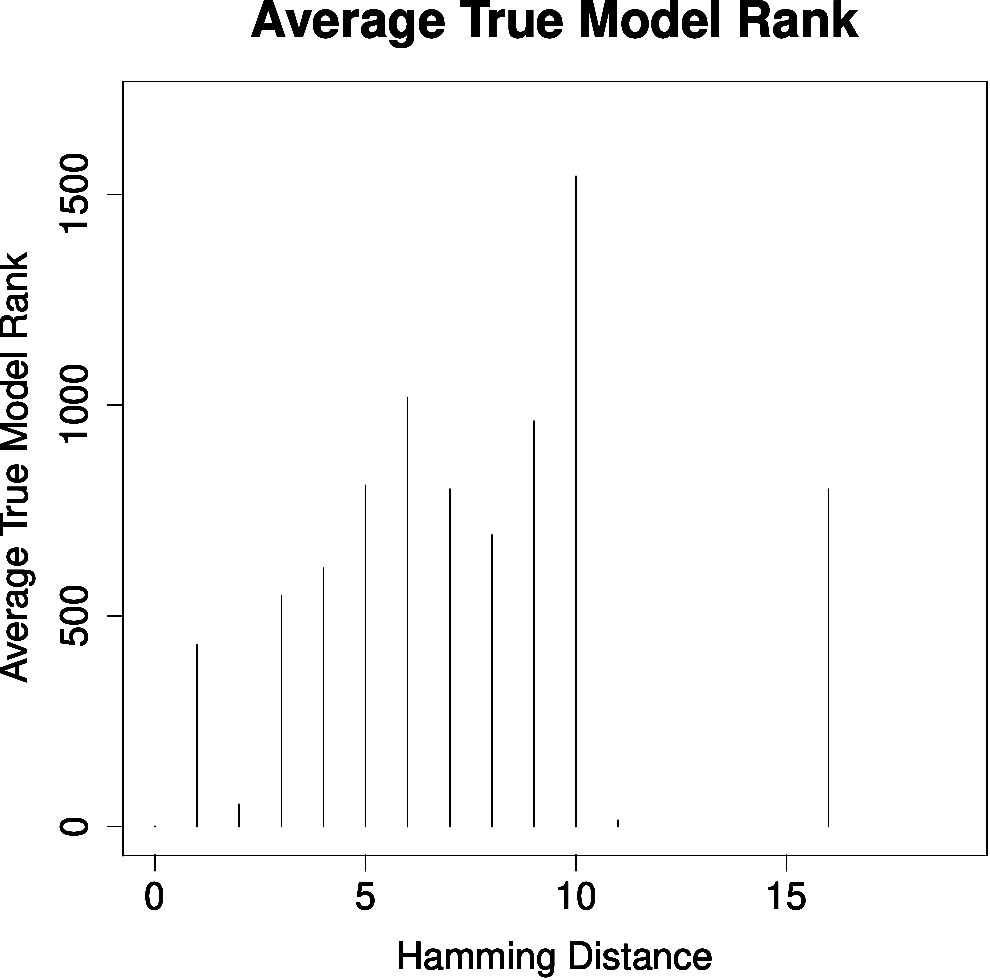}}\\
	\vspace{2mm}
	\subfigure [$p=30,n=35$.]{ \label{fig:a30_ts}
	\includegraphics[width=4.5cm,height=4.5cm]{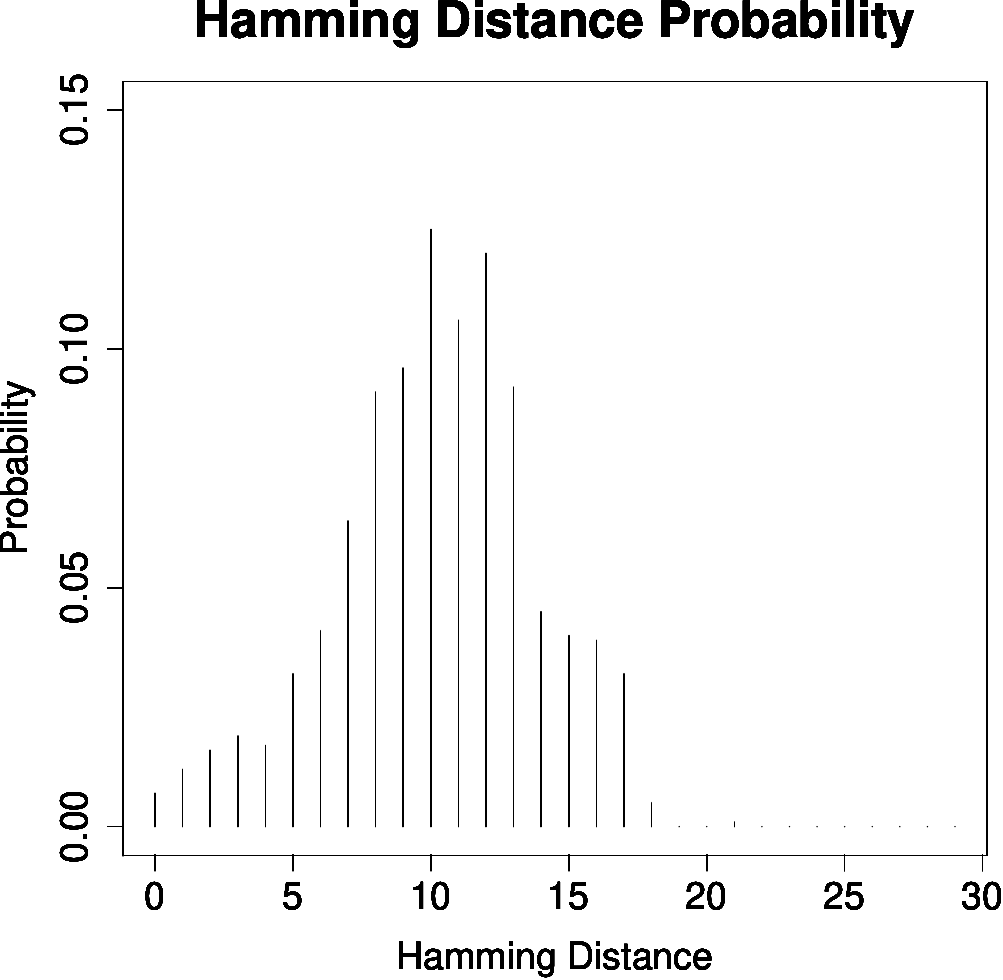}}
	\hspace{2mm}
	\subfigure [$p=30,n=35$.]{ \label{fig:b30_ts}
	\includegraphics[width=4.5cm,height=4.5cm]{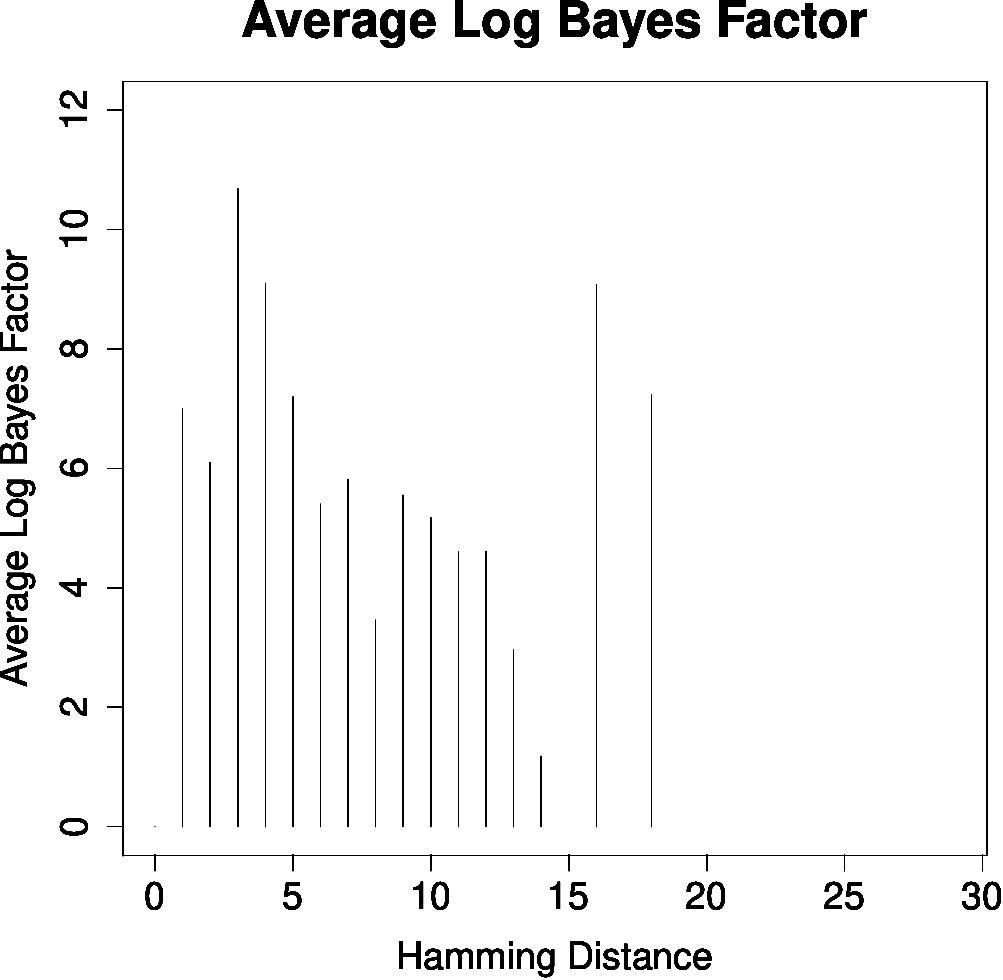}}
	\hspace{2mm}
	\subfigure [$p=30,n=35$.]{ \label{fig:c30_ts}
	\includegraphics[width=4.5cm,height=4.5cm]{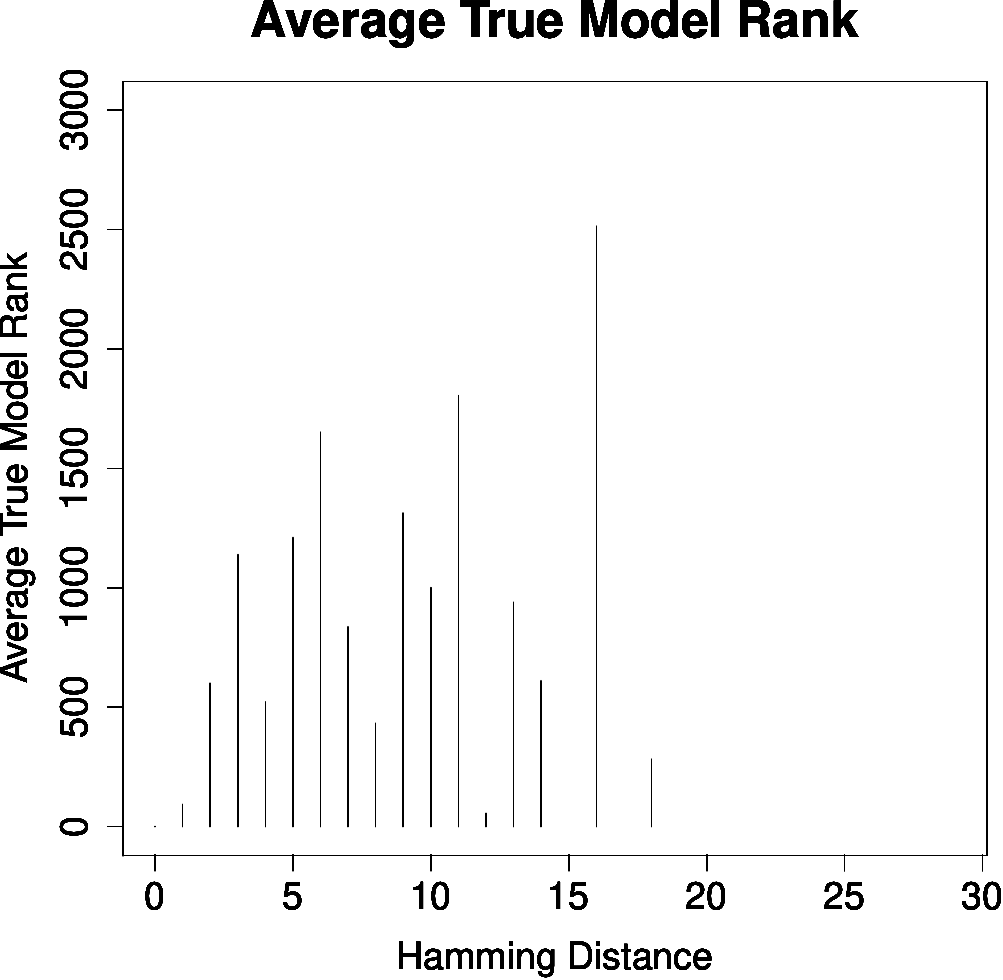}}
	\caption{Simulation study: Bayesian AR(1) regression variable selection results.}
	\label{fig:tsreg}
\end{figure}

\section{Variable selection in a real riboflavin dataset}
\label{sec:realdata}

Let us now apply our TTMCMC based variable selection procedure to a real %high-throughput genomic 
dataset on riboflavin (vitamin $B_2$) production rate, which has been made public by \ctn{Buhl14}. 
The response variable in this dataset is the log-transformed riboflavin production rate corresponding to $p = 4088$ possible covariates measuring the logarithm of
the expression level of $4088$ genes. The sample size is $n=71$. Covariate selection of this dataset using classical methods based on linear regression model 
has been performed by \ctn{Javan14} who report two significant genes YXLD\_at and YXLE\_at.
\ctn{Mein09} on the other hand, found only YXLD\_at to be significant, while the method of \ctn{Buhl13} found no significant gene. Thus, based on linear regression models
and classical methods of variable selection employed so far, either no gene, one gene or two genes, are found to be significant.

We apply our Bayes factor based covariate selection technique using TTMCMC to this dataset, considering both linear and Gaussian process regression.
We consider the same setups as in our simulation experiments, with the same models and priors, with some variation in the prior for $k$, to account for the 
uncertainty with respect to the large number of available covariates. Specifically, with the same discrete normal prior for $k$, we consider the choices of 
$(\mu_k,\sigma^2_k)$ to be $(25,10)$ and $(40,10)$, respectively. We allow a maximum of $50$ covariates in the model, since, as is clear from the aforementioned past 
analyses of this dataset, too many genes can not be significant. In keeping with this, the maximum number of covariates in our posterior simulations turned out to be 
less than $10$ in all our setups for this real data.

As in the simulation experiments, we discard the first $10^4\times 150$ TTMCMC realizations as burn-in and store every $150$-th realization in the next $5\times 10^4\times 150$
iterations, to obtain $5\times 10^4$ TTMCMC realizations for our inference. For the additive transformation, the scales $a_{\theta,j}$ and the constant $c$ in the
mixing-enhancing step are chosen in the same way as in the simulation studies in the $p=30$ setups.
Since for Gaussian process regression minimization of $BIC(u)$ in the TTMCMC step is computationally too demanding for $p=4088$ covariates, we replace
this $BIC(u)$ with that used for linear regression. Theoretically, this is a perfectly valid procedure,  
and our TTMCMC results demonstrate very reasonable final selection of the covariates via Bayes factor.

\subsection{Results for the linear regression model}
For the linear regression model with $(\mu_k,\sigma^2_k)=(25,10)$, the overall acceptance rate, birth rate, death rate, no-change rate turned out to be
$(0.187,0.002,0.003,0.497)$, and the implementation time is about $12$ hours and $11$ minutes. Our Bayes factor computation based on TTMCMC yielded the following best
set of $4$ covariates: (ARGB\_at, EXOA\_at, SIGY\_at, YOAB\_at).

When $(\mu_k,\sigma^2_k)=(40,10)$, the overall acceptance rate, birth rate, death rate, no-change rate are
$(0.175,0.001,0.001,0.496)$ and the time taken is about $11$ hours $33$ minutes.
In this case, the following set of $7$ covariates turned out to be the best:
(ARGB\_at, YDAR\_at, YHDZ\_at, YJIA\_at, YOAB\_at, YUZF\_at, YXLD\_at).
The covariates common to both $(\mu_k,\sigma^2_k)=(25,10)$ and $(\mu_k,\sigma^2_k)=(40,10)$ are (ARGB\_at,YOAB\_at), which does not
contain the covariates found significant by \ctn{Javan14} or \ctn{Mein09}.
More specifically, although the case $(\mu_k,\sigma^2_k)=(40,10)$ contains YXLD\_at, which has been found to be significant by both \ctn{Javan14} and \ctn{Mein09}, none of 
the cases $(\mu_k,\sigma^2_k)=(25,10)$ or $(\mu_k,\sigma^2_k)=(40,10)$ finds YXLE\_at, declared as significant by \ctn{Javan14}.

\subsection{Results for the Gaussian process regression model}

For the Gaussian process regression model with $(\mu_k,\sigma^2_k)=(25,10)$, the implementation time is about $8$ hours $47$ minutes 
and the overall acceptance rate, birth rate, death rate, no-change rate are $(0.185,0.012,0.012,0.530)$.
The following $4$ covariates are selected as the best by our TTMCMC based Bayes factor: (ARGB\_at, YHDZ\_at, YOAB\_at, YXLD\_at).

For $(\mu_k,\sigma^2_k)=(40,10)$, the implementation time was about $9$ hours $56$ minutes and the 
overall acceptance rate, birth rate, death rate, no-change rate are $(0.193,0.024,0.024,0.529)$.
Remarkably, here we obtain exactly the same set of covariates (ARGB\_at, YHDZ\_at, YOAB\_at, YXLD\_at), as for
$(\mu_k,\sigma^2_k)=(25,10)$, as the best set of covariates, which exhibits considerable robustness of the Gaussian process regression model with respect

\subsection{Comparison of the results for the linear regression and the Gaussian process regression models}
The results demonstrate that compared to the linear regression model for this data, 
the Gaussian process regression is far more robust with respect to the prior for $k$; moreover, it leads
to much parsimony compared to linear regression, as can be easily seen from the cardinalities of the best covariate sets.

Note that YOAB\_at is common to all our linear regression and Gaussian process regression implementations and hence we consider this to be an important discovery. 
Also, YXLD\_at is common to our Gaussian process and
linear regression implementations with $(\mu_k,\sigma^2_k)=(40,10)$. Since this gene is found to be significant by \ctn{Javan14} and \ctn{Mein09} as well,
it seems that this may also be an important discovery. But the YXLE\_at gene, although declared significant by \ctn{Javan14}, did not appear in the best set of covariates
in any of our linear regression or Gaussian process regression analysis. 

Now, the question arises that which of the four implementations of linear and Gaussian process regression yields the best result in terms of Bayes factor.
In this regard, we first note that the maximum values of the log of $B_i$ given by (\ref{eq:B_i}) associated with the linear regression models for $(\mu_k,\sigma^2_k)=(25,10)$
and $(\mu_k,\sigma^2_k)=(40,10)$ are given by $51.925$ and $94.697$, respectively, while the same for the Gaussian process regression counterpart are
$46.085$ and $91.147$. Thus, in this regard, the linear regression model with $(\mu_k,\sigma^2_k)=(40,10)$ given its Bayes factor guided best possible set of covariates,
is the best model, followed by Gaussian process regression
with $(\mu_k,\sigma^2_k)=(40,10)$, given its best set of covariates. 
The next best models, given their respective best set of covariates with respect to Bayes factors, in order, 
are provided by the linear regressions model with $(\mu_k,\sigma^2_k)=(25,10)$ and 
the Gaussian process regression model with $(\mu_k,\sigma^2_k)=(25,10)$.

Note that the genes YOAB\_at and YXLD\_at are common to the best two models, which once again vindicates their importance.

\section{Appendix} \label{sec:appendix}
In this section we provide the proofs of all the lemmas stated in the paper. Before proving the lemmas, we state some results which are useful in proving the lemmas.
\begin{result}
	\label{theorem:wang}
	Let $\mathbb C^{m\times n}$ denote the vector space of all $m\times n$ matrices. If $G,H\in\mathbb C^{n\times n}$
	are positive semidefinite Hermitian matrices and $1\leq i_1<\cdots<i_k\leq n$, then the following two inequalities hold:
	\begin{align}
	\sum_{t=1}^k\lambda_{i_t}\left(GH\right)&\leq\sum_{t=1}^k\lambda_{i_t}\left(G\right)\lambda_t\left(H\right) \notag \\
	%\label{eq:wang1}\\
	\sum_{t=1}^k\lambda_t\left(GH\right)&\geq\sum_{t=1}^k\lambda_{i_t}\left(G\right)\lambda_{n-i_t+1}\left(H\right). \notag
	%\label{eq:wang2}
	\end{align}
\end{result}
\noindent A proof of this result can be found in \cite{Wang92}.

\begin{result}\label{eqn:symmetric2}
	\begin{enumerate}[(a)]
		\item For matrices $A_1$ and $A_2$, let $A_1\preceq A_2$ imply that $A_2-A_1$ is nonnegative definite. Then for any symmetric matrix $A$,
		\begin{equation}
		% \lambda_{\min}(A)=\inf_{{\bf u}: \|{\bf u}\|=1} {\bf u}^T A{\bf u}\leq \sup_{{\bf u}: \|{\bf u}\|=1} {\bf u}^T A{\bf u}= \lambda_{\max}(A), \label{eqn:symmetric2}
		\lambda_{\min}(A)I\preceq A\preceq\lambda_{\max}(A)I, \notag
		\end{equation} 
		
		\item For symmetric matrices $A_1$ and $A_2$, %of order $n$, %$\lambda_1(A_1+A_2)\leq\lambda_1(A_1)+\lambda_1(A_2)$ and
		\begin{equation}
		\lambda_{\min}(A_1)+\lambda_{\min}(A_2) \leq \lambda_{\min}(A_1+A_2)\leq \lambda_{\max}(A_1+A_2) \leq \lambda_{\max}(A_1)+\lambda_{\max}(A_2). \notag % \label{eqn:symmetric}
		\end{equation}
		
	\end{enumerate}
\end{result}

\begin{lemma}\label{lm_3}
	Consider the setup of Section \ref{subsec:time_series1}. The eigenvalues $Cov\left(\tilde\epsilon_{t+h},\tilde\epsilon_t\right)=\sigma^2_{\epsilon}(1-\rho^2)^{-1}\Sigma_\epsilon$ are all positive, and bounded.
\end{lemma}
\begin{proof}
	Let $Cov\left(\tilde\epsilon_{t+h},\tilde\epsilon_t\right)=\sigma^2_{\epsilon}(1-\rho^2)^{-1}\Sigma_\epsilon$. We first find the highest and lowest eigenvalues of $\Sigma_{\epsilon}$. It can be shown that the inverse of $\Sigma_{\epsilon}$ is a tridiagonal matrix as follows:
	\begin{eqnarray*}
		\left(1-\rho^2\right)\Sigma_{\epsilon}^{-1}=
		\begin{pmatrix}
			1 & -\rho & 0 & 0& \ldots & 0 & 0\\
			-\rho & 1+\rho^2 & -\rho & 0& \ldots & 0 & 0\\
			0 & -\rho & 1+\rho^2 & -\rho & \ldots & 0 & 0 \\
			&  \ldots & & \ldots & & \ldots & \\
			0 & 0 & 0 & 0 & \ldots &1+\rho^2 & -\rho \\
			0 & 0 & 0 & 0 & \ldots &- \rho& 1 
		\end{pmatrix} 
	\end{eqnarray*}
	%Define $M=\left(1-\rho^2\right)\Sigma_{\epsilon}^{-1}+\rho^2 Diag \left(1,0,\ldots,0,1 \right)$. 
	%Note that the eigenvalues of the symmetric matrix $M$ can be obtained exactly. From \cite{AR1_approx} the eigenvalues of $M$ are as follows: 
	%$$\lambda_k(M)=1-2\rho \cos \frac{k\pi}{n+1}+\rho^2,~~\mbox{for} ~k=1,\ldots,n.$$ 
	
	\cite{AR1_approx} shows that the approximations of the eigenvalues of $\left(1-\rho^2\right)\Sigma_{\epsilon}^{-1}$ (arranged in increasing order if $\rho>0$ and in decreasing order if $\rho<0$) are
	$$\lambda_k\left( \left(1-\rho^2\right)\Sigma_{\epsilon}^{-1}\right) \approx 1-2\rho \cos \frac{k\pi}{n+1}+\rho^2-\frac{4}{n+1} \rho^2 \left(\sin \frac{k\pi}{n+1} \right)^2, $$
	with corresponding error bound
	$$ \xi_k=\frac{2\rho^2}{\sqrt{n+1}} \sin \frac{k\pi}{n+1},~~ \mbox{for~} k=1,\ldots,n.$$

	Combining the above facts, it can be seen that the eigenvalues of $\left(1-\rho^2\right)\Sigma_{\epsilon}^{-1}$ are bounded by $\min\{(1+\rho)^2,(1-\rho)^2\}+o(1)$ and $\max\{(1+\rho)^2,(1-\rho)^2\}+o(1)$. Thus the eigenvalues of $\left(1-\rho^2\right)^{-1}\Sigma_{\epsilon}$ are bounded by $ \left\{(1-|\rho|)^2+o(1) \right\}^{-1}$ and $\left\{(1+|\rho|)^2+o(1) \right\}^{-1}$. Thus all the eigenvalues of $\Sigma_{\epsilon}$ are finite. As $|\rho|<1-\gamma$ for some small enough $\gamma$, clearly for sufficiently large $n$ all the eigenvalues of $
	\left(1-\rho^2\right)\Sigma_{\epsilon}^{-1}$ are positive. 
	Also the highest eigenvalue is less than $2/\gamma^2$.
	%Further note that determinant of $\frac{1}{1-\rho^2}\Sigma_{\epsilon}$ is $\left(1-\rho^2\right)^{-1}$. Thus all the eigenvalues are positive if $|\rho|<1$.
\end{proof}

\begin{lemma}\label{lm_2}
	Consider the setup in Section \ref{subsec:time_series1}. Let $P_{n,\bs}$ be the orthogonal projection matrix onto the column space of $Z_{n,\bs}$. Then the eigenvalues of $\partial P_{n,\bs}/\partial \rho$ are uniformly bounded.
\end{lemma}
\begin{proof}
	For simplicity we write $Z$ instead of $Z_{n,\bs}$. Recall $P_{n,\bs}=Z\left(Z^TZ \right)^{-1} Z^T$. Thus,
	\begin{eqnarray}
	\frac{\partial P_{n,\bs}}{\partial \rho}=\frac{\partial Z}{\partial \rho}\left(Z^TZ \right)^{-1} Z^T - Z \left(Z^TZ \right)^{-1} \frac{\partial \left(Z^TZ \right)}{\partial \rho}\left(Z^TZ \right)^{-1} Z^T+ Z\left(Z^TZ \right)^{-1}\frac{\partial Z^T}{\partial \rho} \label{A1}
	%	&=& A+A^T- Z \left(Z^TZ \right)^{-1} \frac{\partial \left(Z^TZ \right)}{\partial \rho}\left(Z^TZ \right)^{-1} Z^T. 
	\end{eqnarray}
	First note that the $t$-th row of $Z$, ${\bf z}_{t,\bs}=\sum_{k=1}^{t} \rho^{t-k} {\bf x}_{k,\bs}$. Thus,
	\begin{eqnarray}
	Z^TZ=n\sum_{t=1}^{n} \sum_{k_1=1}^{t} \sum_{k_2=1}^{t} \rho^{2t-k_1-k_2} \left(\frac{{\bf x}_{k_1,\bs}{\bf x}_{k_2,\bs}^T}{n}\right). \label{A2}
	\end{eqnarray}
	Since the covariates lie in a compact space, the elements of the $|\bs|\times |\bs|$ matrices ${\bf x}_{k_1,\bs}{\bf x}_{k_2,\bs}^T$ in (\ref{A2}) are uniformly bounded. 
	Further note that 
	%Thus, there exists non-singular matrices $B$ and $C$ such that eigenvalues of $Z^TZ$ are bounded by eigenvalues of $\kappa(\rho) B$ and $  \kappa(\rho) C,$ where $\kappa(\rho)=\sum_{t=1}^{n} \sum_{k_1=1}^{t} \sum_{k_2=1}^{t} \rho^{2t-k_1-k_2}$. Now,
	\begin{eqnarray*}
		%|\kappa(\rho)| &=& 
		\frac{1}{n}\sum_{t=1}^{n} \sum_{k_1=1}^{t} \sum_{k_2=1}^{t} \rho^{2t-k_1-k_2}
		&\leq& \frac{1}{n} \sum_{t=1}^{n} \left( \sum_{k_1=1}^{t} |\rho|^{t-k}  \right)^2=\frac{1}{n(1-|\rho|)^{2}} \sum_{t=1}^{n} \left(1-|\rho|^t \right)^2\\ 
		&=& (1-|\rho|)^{-2}\left[1-2\frac{|\rho|}{n}\frac{\left(1-|\rho|^n \right)}{\left(1-|\rho| \right)} +\frac{\rho^2\left(1-\rho^{2n} \right)}{n\left(1-|\rho|^2 \right)}\right].
		% = \sum_{t=1}^{n} \sum_{k_1=1}^{t} \rho^{t-k_1} \left( 1 + \rho +\ldots+ \rho^{t-1} \right)\\
		% &=& \sum_{t=1}^{n} \sum_{k_1=1}^{t} \rho^{t-k_1} \frac{(1-\rho^{t})}{(1-\rho)}= \sum_{t=1}^{n}\frac{(1-\rho^{t})^2}{(1-\rho)^2} \\
		% &=& \frac{1}{(1-\rho)^2}\sum_{t=1}^{n}(1-2\rho^{t}+\rho^{2t}) = \frac{1}{(1-\rho)^2} \left[n-2\frac{(1-\rho^n)}{(1-\rho)}+\frac{(1-\rho^{2n})}{(1+\rho)(1-\rho)} \right]\\
		% &\approx& \frac{1}{(1-\rho)^2} \left[n-\frac{2}{(1-\rho)}+\frac{1}{(1+\rho)(1-\rho)} \right].
	\end{eqnarray*}
	As $\rho\in[-1+\gamma,1-\gamma]$, the above facts imply  $Z^TZ=nB_n$, where $B_n$ is a $|\bs|\times |\bs|$ matrix whose elements are uniformly bounded. It follows that $\left(Z^TZ\right)^{-1} =n^{-1} B_n^{-1}=n^{-1} C_{n,\bs}$, where the elements of $C_{n,\bs}$ are uniformly bounded.
	
	Let $\partial Z/\partial \rho =H_n$. Then the $t$-th row of $H_n$ is ${\bf h}_{t,\bs}=\sum_{k=1}^{t} (t-k) \rho^{t-k-1} {\bf x}_{k,\bs}$. then the first term of (\ref{A1}),
	\begin{eqnarray*}
		D:=\frac{\partial Z}{\partial \rho}\left(Z^TZ \right)^{-1} Z^T &=& n^{-1}  \begin{pmatrix}
			{\bf h}_{1,\bs}^TC_{n,\bs} {\bf z}_{1,\bs} & \ldots & {\bf h}_{1,\bs}^TC_{n,\bs} {\bf z}_{n,\bs}\\
			\ldots & & \ldots \\
			{\bf h}_{n,\bs}^TC_{n,\bs} {\bf z}_{1,\bs} & \ldots & {\bf h}_{n,\bs}^TC_{n,\bs} {\bf z}_{n,\bs}
		\end{pmatrix}.
	\end{eqnarray*}
	The $(t_1,t_2)$-th element of $D$ is
	\begin{eqnarray*}
		d_{t_1,t_2}&=& \frac{1}{n} {\bf h}_{t_1,\bs}^TC_{n,\bs} {\bf z}_{t_2,\bs} = \frac{1}{n} \left[\sum_{k=1}^{t_1} (t_1-k) \rho^{t_1-k-1} {\bf x}_{k,\bs}^T\right] C_{n,\bs} \left[\sum_{k=1}^{t} \rho^{t-k} {\bf x}_{k,\bs} \right]\\
		&=& \frac{1}{n} \sum_{k_1=1}^{t_1}\sum_{k_1=1}^{t_2} (t_1-k_1) \rho^{t_1-k_1-1+t_2-k_2} {\bf x}_{k_1,\bs}^T C_{n,\bs}  {\bf x}_{k_2,\bs}. 
	\end{eqnarray*}
	Next note that $ {\bf x}_{k_1,\bs}^T C_{n,\bs}  {\bf x}_{k_2,\bs} =\sum_{j=1}^{|\bs|}\sum_{l=1}^{|\bs|} C_{j,l} x_{k_1,j}x_{k_2,l}$. As elements of $C_{n,\bs}$, as well as, the covariate space are uniformly bounded, the above quadratic is bounded for any $(k_1,k_2)$.
	
	\noindent Now, 
	\begin{eqnarray}
	\frac{1}{n} \sum_{k_1=1}^{t_1}\sum_{k_1=1}^{t_2} (t_1-k_1) \rho^{t_1-k_1-1+t_2-k_2} &\leq&  \frac{1}{n}\left( \sum_{k_1=1}^{t_1} (t_1-k_1)|\rho|^{t_1-k_1-1} \right) \left(\sum_{k_1=1}^{t_2}  \rho^{t_2-k_2} \right) \notag \\
	&=&\frac{1}{n} \left\{ \frac{1-|\rho|^{t_1}\left(t_1+|\rho|-t_1|\rho| \right)}{(1-|\rho|)^2} \right\} \left(\frac{1-|\rho|^{t_2}}{1-|\rho|} \right).\label{A3}
	\end{eqnarray}
	By Gerschgorin circle theorem the maximum eigenvalue of $D$ is bounded by $$\max_{t_1\in\{1,\ldots,n\}} \frac{1}{n} \left\{ \frac{1-|\rho|^{t_1}\left(t+|\rho|-t|\rho| \right)}{(1-|\rho|)^2} \right\} \sum_{t_2=1}^{n}  \left(\frac{1-|\rho|^{t_2}}{1-|\rho|} \right)\leq \left(1-|\rho| \right)^{-3},$$
	and the minimum eigenvalue is bounded below by $-\left(1-|\rho| \right)^{-3}$. As $\rho\in[-1+\gamma,1-\gamma]$, eigenvalues of $D$ are uniformly bounded.

	Next note that $A$ and $A^T$ have same eigenvalues.  As the third term of (\ref{A1}) is the transpose of the first term, the third term also has finite eigenvalues.

	Next consider the second term of (\ref{A1}). Note that, 
	%this matrix is symmetric and has bounded eigenvalues as it can be written as sum of two matrices which are complements of each other as follows:
	\begin{eqnarray*}
		Z \left(Z^TZ \right)^{-1} \frac{\partial \left(Z^TZ \right)}{\partial \rho}\left(Z^TZ \right)^{-1} Z^T = P_{n,\bs} \frac{\partial Z }{\partial \rho}\left(Z^TZ \right)^{-1} Z^T + Z \left(Z^TZ \right)^{-1} \frac{\partial Z^T }{\partial \rho}P_{n,\bs}.
	\end{eqnarray*}
	Recall, for two square matrices $A$ and $B$, $AB$ and $BA$ have same eigenvalues. Thus eigenvalues of $P_{n,\bs} \frac{\partial Z }{\partial \rho}\left(Z^TZ \right)^{-1} Z^T$ is same as eigenvalues of $ \frac{\partial Z }{\partial \rho}\left(Z^TZ \right)^{-1} Z^TP_{n,\bs}=\frac{\partial Z }{\partial \rho}\left(Z^TZ \right)^{-1} Z^T$, and eigenvalues of $Z \left(Z^TZ \right)^{-1} \frac{\partial Z^T }{\partial \rho}P_{n,\bs}$ is same as that of $P_{n,\bs}Z \left(Z^TZ \right)^{-1} \frac{\partial Z^T }{\partial \rho}=Z \left(Z^TZ \right)^{-1} \frac{\partial Z^T }{\partial \rho}$, as $P_{n,\bs}Z=Z$.

	Finally, note that for any matrix $A$, $\lambda_{\max}(A+A^T)$ is bounded above by $2\lambda_{\max}(A)$, and $\lambda_{\min}(A+A^T)$ is bounded below by $2\lambda_{\min}(A)$. Hence the second term of (\ref{A1}) has finite eigenvalues. Similarly, the sum of first and third term of (\ref{A1}) has bounded eigenvalues. 
	%Thus the maximum eigenvalue of the matrix above is less than $2 b_0^{-1} c_0^{*} \lambda_{\min}^{-1}(B) \lambda_{\max}(C)$, and the minimum eigenvalue is no less than zero. 
	
	Combining the above results, note that the RHS of (\ref{A1}) can written as sum of two symmetric matrices, and eigenvalues of each of them are bounded. Using the result that for symmetric matrices $A$ and $B$, 
	
	{\centering
		$\lambda_{\min}(A)+\lambda_{\min}(B) \leq \lambda_{\min}(A+B)\leq \lambda_{\max}(A+B) \leq \lambda_{\max}(A)+\lambda_{\max}(B), $
		\par}
	
	\noindent the result follows.
\end{proof}

\bibliographystyle{natbib}
\bibliography{irmcmc}

\end{document}